\documentclass[aop,reqno]{imsart}

\RequirePackage{amsthm,amsmath,amsfonts,amssymb}
\RequirePackage[numbers]{natbib}

\usepackage{SYM}
\usepackage{SYM-diagrams}
\usepackage{hyperref}
\usepackage{mhequ}
\usepackage{SG}

\newcommand{\revision}[1]{#1}

\startlocaldefs

\endlocaldefs

\begin{document}

\begin{frontmatter}
\title{Global well-posedness of the dynamical sine-Gordon model up to $\mathbf{6} \boldsymbol{\pi}$}
\runtitle{Global well-posedness of the dynamical sine-Gordon model up to~$6\pi$}

\begin{aug}
\author[A]{\fnms{Bjoern}~\snm{Bringmann}\ead[label=e1]{bringmann@princeton.edu}}
\and
\author[B]{\fnms{Sky}~\snm{Cao}\ead[label=e2]{skycao@mit.edu}}
\address[A]{Department of Mathematics, Princeton University, Princeton, NJ 08544\printead[presep={,\ }]{e1}}

\address[B]{Department of Mathematics, Massachusetts Institute of Technology, Cambridge, MA 02139\printead[presep={,\ }]{e2}}
\end{aug}

\begin{abstract}
We prove the global well-posedness of the dynamical sine-Gordon model up to the third threshold, i.e., for parameters $\beta^2 < 6\pi$. 
The key novelty in our approach is the introduction of the so-called resonant equation, whose solution is entirely deterministic and completely captures the size of the solution to the dynamical sine-Gordon model. The probabilistic fluctuations in the dynamical sine-Gordon model are then controlled using uniform estimates for modified stochastic objects.
\end{abstract}

\begin{keyword}[class=MSC]
\kwd{35R60,60H17}
\end{keyword}

\begin{keyword}
\kwd{Global well-posedness}
\kwd{Abelian-Higgs}
\kwd{Yang-Mills}
\kwd{Singular SPDEs}
\kwd{Gibbs measures}
\end{keyword}
\end{frontmatter}

\tableofcontents

\section{Introduction}

Following the success of \cite{H14,GIP15,K16,CH16,BHZ19,BCCH21,D21} in establishing a general local theory for singular stochastic partial differential equations (SPDEs), nowadays there is great interest in developing a global theory for the same equations. Unlike the local theory, the global theory generally has to be treated on a case-by-case basis, using specific details about the structure of the equation under consideration.
In terms of global theory, the first model to be treated was the parabolic $\Phi^4_d$ model (for $d = 2, 3$), by \cite{MW17, MW2020, GH21}. The special feature of this model is a good $-\phi^3$ damping term, which is used to absorb all the error terms in the energy estimates. However, in order to move beyond the $\Phi^4_d$ model, one needs techniques for showing global existence even without such a good damping term. Recent works have understood how to do this in various cases, for instance the 2D stochastic Navier-Stokes equation~\cite{HR23}, 2D stochastic pure Yang--Mills~\cite{CS23}, the generalized parabolic Anderson model~\cite{CdLFW24, SZZ24}, and 2D stochastic Abelian-Higgs~\cite{BC24}.

In the present paper, we consider the sine-Gordon model, which is a classic model of mathematical physics (for references to the large body of existing literature, see \cite[Section 1.3]{BB21}). More precisely, we consider the dynamical, massive version of the problem, which is the following singular SPDE: 
\begin{equs}\tag{mSG}\label{eq:mSG}
(\ptl_t - \revision{\tfrac{1}{2}}\Delta + m^2) u = \biglcol \, \sin(\beta u)\, \bigrcol + \zeta, ~~ u(0) = u_0,
\end{equs}
where $m^2 > 0$, $u_0 \colon\T^2 \ra \R$ and $\zeta$ is a space-time white noise on $\R \times \T^2$. Here, $\beta^2 \in [0, 8\pi)$, and there is an infinite sequence of thresholds approaching $8\pi$ given by:
\begin{equs}\label{eq:thresholds}
\beta^2_n = \frac{8n\pi}{n+1}, ~~ n \geq 1.
\end{equs}
At each successive threshold, the local theory for dynamical sine-Gordon gets more and more complicated. Up to the first threshold $4\pi$, the local theory can just be proven by a standard Da Prato-Debussche trick, but to go beyond $4\pi$, one needs more modern techniques. By using regularity structures, the article \cite{CH16} covered the local theory up to the second threshold $\beta^2 \in [0, 16\pi/3)$, and then the later article \cite{CHS2018} managed to handle the entire subcritical regime, i.e., $\beta^2 \in [0, 8\pi)$.

We next discuss how the solution $u$ to \eqref{eq:mSG} is constructed in \cite{CH16, CHS2018}. As usual, $u$ needs to be defined via a limiting procedure, which for the sine-Gordon model is done as follows. Fix a suitable mollifier $\rho \colon\R \times \T^2 \ra \R$ (see Section \ref{section:setup-and-main} for our precise assumptions on $\rho$), and for $\varep > 0$, define the rescaled version
\begin{equs}
\rho_\varep(t, x) := \varep^{-4} \rho(\varep^{-2} t, \varep^{-1} x).
\end{equs}
Let $\zeta_\varep := \zeta \ast \rho_\varep$. Consider the regularized equation
\begin{equs}
(\ptl_t - \revision{\tfrac{1}{2}}\Delta + m^2) u_\varep = \biglcol \, \sin(\beta u_\varep)\, \bigrcol + \zeta_\varep, ~~ u_\varep(0) = u_{0, \varep},
\end{equs}
where $u_{0, \varep}$ is a mollification\footnote{Technically, \cite{CHS2018} did not mollify the initial data, but their result easily extends to this case due to the usual stability results.} of the initial data $u_0$ at scale $\varep$, 
\begin{equs}\label{eq:renormalized-sine}
\biglcol \, \sin(\beta u_\varep)\, \bigrcol := C_{\beta, \rho, \varep} \sin(\beta u_\varep),
\end{equs}
and $C_{\beta, \rho, \varep}$ is a renormalization constant to be defined below, see \eqref{eq:gmc-renorm-constant}.

In \cite[Theorem 1.1]{CHS2018}, it was shown\footnote{Technically, in \cite[Theorem 1.1]{CHS2018} the following results are stated for the remainder $v_\varep$, but they transfer to $u_\varep$ since the difference between $u_\varep$ and $v_\varep$ is an explicit object for which one has explicit stochastic estimates which preclude the possibility of explosion in finite time. Also, technically the paper considers the massless equation, but the arguments directly extend to the massive case as well.} that for $\beta^2 \in [0, 8\pi)$, and $u_0 \in \Cs_x^\eta(\T^2)$ for $\eta \in (\frac{\beta^2}{8\pi} - 1, 0)$, the sequence $u_\varep$ converges to some random variable $u \in \mc{D}'(\R \times \T^2)$, in the sense that there is a stopping time $\tau$ such that for every $T > 0$, we have that $\|u_\varep - u\|_{C_t^0 \Cs_x^\eta([0, T])} \ra 0$ in probability on the event $\{\tau > T\}$. Moreover, we have that $\lim_{t \uparrow \tau} \|u(t)\|_{\Cs_x^\eta} = \infty$ on the event $\{\tau < \infty\}$. 

Discussing the global theory now, global well-posedness for $\beta^2 \in [0, 4\pi)$ follows directly from the local theory, because in this regime, the estimates that one obtains are as if the equation was linear.  Moving beyond this regime, the paper \cite{CdLFW24} proved global well-posedness for the generalized parabolic Anderson model, which includes dynamical sine-Gordon for parameters $\beta^2$ up to roughly $4.528\pi$ as a special case. There is also the paper \cite{SZZ24}, and one might expect that by extending the arguments developed therein, one would be able to show global well-posedness for $\beta^2$ up to roughly $4.628\pi$ (see \cite[Remark 1.3(iii)]{SZZ24}). In the present paper, we show global well-posedness of the dynamical sine-Gordon model all the way up to $6\pi$ (which recalling \eqref{eq:thresholds} is the third threshold).   In the following, we take $\bar{\beta} = \frac{\beta^2}{4\pi} + \kappa$, where $\kappa = 10^{-3}(6\pi - \beta^2)$.

\begin{theorem}[Global well-posedness of the dynamical sine-Gordon model up to $6\pi$]\label{thm:sg-gwp}
Let $\beta^2 < 6\pi$, let $m^2 > 0$, and let $\eta \in (\frac{\bar{\beta}}{2}-1, 0)$. For any deterministic initial data $u_0 \in \Cs_x^{\eta}(\T^2)$, the solution $u$ to \eqref{eq:mSG} a.s. exists globally, and moreover satisfies the following uniform-in-time bounds. Let $p \geq 1$. There are constants $C, c$, where $C$ depends on $\beta, m^2, p, \eta$, and $c$ depends on $\beta, m^2, \eta$, such that for all $T \geq 0$,
\begin{equs}\label{eq:sg-gwp}
\E[\|u(T)\|_{\Cs_x^{\eta}}^p] \leq e^{-c p T} \|u_0\|_{\Cs_x^\eta}^p + C.
\end{equs}
\end{theorem}

\begin{remark} We make the following remarks regarding Theorem \ref{thm:sg-gwp}.
\begin{enumerate}
\item Our proof also covers the case where the initial data $u_0$ is not necessarily deterministic, as long as it is independent of $\zeta |_{(0,\infty)}$, i.e., the noise after time $0$. On the other hand, the previously mentioned paper \cite{CdLFW24} (which covers sine-Gordon up to roughly $4.528\pi$) can handle anticipative initial data.
\item While Theorem \ref{thm:sg-gwp} only concerns the massive case $m^2>0$, our arguments can also be used in the massless case $m^2=0$. In the massless case, a simple modification of our argument yields the global well-posedness of \eqref{eq:mSG}, albeit only with polynomial-in-time bounds. Provided that one only considers the mean-zero component of the solution $u$, however, the polynomial-in-time bounds can be upgraded to uniform-in-time bounds. For this, one first has to note that, in the massless case, \eqref{eq:mSG} is invariant under the gauge-transformation $u \mapsto u + \frac{2\pi}{\beta} n$, where $n\in \mathbb{Z}$. Similar as for the Abelian-Higgs model in \cite{BC24}, it is then possible to obtain uniform-in-time estimates of $u$ in gauge-invariant norms, and these norms control the mean-zero component of $u$.
\item By combining the uniform-in-time estimate from \eqref{eq:sg-gwp}  with a Krylov-Bogolyubov argument (see e.g.~ \cite{dPZ96,TW18}), it should be possible to construct the sine-Gordon EQFT on $\T^2$ for $\beta^2 < 6\pi$. The recent paper \cite{GM2024} has already constructed the sine-Gordon EQFT in infinite volume (i.e., on~$\R^2$), also up to $6\pi$. It would be interesting to try to extend our result to infinite volume, where (as far as we can tell) even showing local existence is unclear. 
\item The threshold $6\pi$ that we achieve is the same one reached in the previously mentioned work \cite{GM2024}, as well as the work \cite{BB21}, which shows that the massive sine-Gordon model satisfies a log-Sobolev inequality. It is interesting that these independent results all stop at the same exact threshold (which is the third\footnote{\cite{GM2024} refer to $6\pi$ as the second threshold, but it depends on how one counts. Based off the formula \eqref{eq:thresholds}, which previously appeared as \cite[(1.3)]{CHS2018}, $6\pi$ is the third threshold.} out of an infinite sequence of thresholds for dynamical sine-Gordon -- recall \eqref{eq:thresholds} and see \cite[Section~1]{CHS2018} for more discussion). The natural followup question is to try to extend any of the results beyond $6\pi$, and ultimately all the way up to $8\pi$, which would cover the full subcritical regime.
\end{enumerate}
\end{remark}

The starting point for the proof of Theorem \ref{thm:sg-gwp} is an idea which originated in our previous article~\cite{BC24} on the global well-posedness of the stochastic Abelian-Higgs model, which is to hide the initial data $u_0$ in the stochastic objects. This leads to modified stochastic objects, and the key point is that these modified objects obey stochastic estimates which display minimal dependence on the initial data. For the sine-Gordon model, we improve on this idea by isolating out a ``resonant piece\revision{''} which turns out to completely govern the evolution of the size of $u$, while the remaining piece is controlled using stochastic estimates which are {\it uniform in the initial data}. To describe these ideas in slightly more detail, we first let the ``resonant piece\revision{''} $\determ_\varep$ be the solution to the following equation:
\begin{equs}\label{intro:eq-theta}
(\ptl_t - \revision{\tfrac{1}{2}}\Delta + m^2) \determ_\varep &= \mc{N}_\varep(\determ_\varep), \\
\determ_\varep(0) &= u_{0, \varep},
\end{equs}
where the nonlinearity $\mc{N}_\varep$ is to be specified. We emphasize that, by definition, $\determ_\varep$ only depends on $u_{0,\varep}$, and is therefore probabilistically independent of the space-time white noise $\zeta$. The probabilistic independence of $\determ_\varep$ and $\zeta$ will be essential for our stochastic estimates of the modified stochastic objects discussed below. For the stochastic Abelian-Higgs model \cite{BC24}, we simply chose $\mc{N}_\varep \equiv 0$ (and the analogue of $\determ$ is denoted by $B$ in that paper). However, in the present paper, we will make a more judicious choice of $\mc{N}_\varep$, and this is one of the main ideas towards the proof of global well-posedness. We will discuss this choice in more detail later, but for now we just mention that with our choice of $\mc{N}_\varep$, the size of $\mc{N}_\varep(\determ_\varep)$ may be bounded using strictly less than one power of $\determ_\varep$ (as long as $\beta^2 < 6\pi$, which is one of several places where we use this assumption), and using this we will be able to directly obtain global well-posedness for $\determ_\varep$. Our Ansatz for $u_\varep$ is then given by
\begin{equs}
u_\varep = \Phi_\varep + \determ_\varep + v_\varep,
\end{equs}
where (as in \cite{CHS2018}) $\Phi_\varep = K \ast \zeta_\varep$, and where the kernel $K$ is as specified in Section \ref{section:setup-and-main} (as usual, it is a slight modification of the heat kernel). The remainder $v_\varep$ satisfies the following equation:
\begin{equation}\label{eq:v-equation}
\begin{aligned}
(\ptl_t - \revision{\tfrac{1}{2}}\Delta + m^2)v_\varep &= \frac{1}{2\icomplex} \big( e^{\icomplex \beta v_\varep} \xi_+^{\determ_\varep, \varep}  - e^{-\icomplex \beta v_\varep}\xi_{-}^{\determ_\varep, \varep} \big) - \mc{N}_\varep(\determ_\varep) + R_\varep, \\
v_\varep(0) &= -\Phi_\varep(0),
\end{aligned}
\end{equation}
where $R_\varep$ is (as usual) a smooth remainder arising from the fact that $K$ is not exactly the heat kernel, but only differs from the heat kernel by a smooth function, and the modified noises $\xi^{\determ_\varep, \varep}_{\pm}$ are defined as
\begin{equs}\label{eq:modified-noise}
\xi_{\pm}^{\determ_\varep, \varep }:= e^{\pm \icomplex \beta \determ_\varep} \biglcol\, e^{\pm \icomplex \beta \Phi_\varep} \bigrcol = C_{\beta, \rho, \varep} e^{\pm \icomplex \beta \determ_\varep}  e^{\pm \icomplex \beta \Phi_\varep},
\end{equs}
where we define the constant
\begin{equs}\label{eq:gmc-renorm-constant}
C_{\beta, \rho, \varep} := \exp\Big(\frac{\beta^2}{2} \E \Phi_\varep(0)^2\Big).
\end{equs}
We previously mentioned that $\determ_\varep$ is globally well-posed. The same is true for $\Phi_\varep$, since it is an explicit stochastic object. Thus, it only remains to estimate $v_\varep$. To do so, we will simply use the local theory, which was established in \cite{CHS2018} by using regularity structures. At this point, we note that our remainder $v_\varep$ differs from the remainder $v_\varep$ defined in \cite{CHS2018} in three essential ways: (1) we use the modified noises $\xi^{\determ_\varep, \varep}_{\pm}$ (2) there is an additional forcing term $\mc{N}_\varep(\determ_\varep)$ (3) the initial data of $v_\varep$ does not depend on $u_0$. These differences all arise from the fact that we isolated out the resonant piece $\determ_\varep$. Note also that the impact of the initial data on the evolution of $v_\varep$ only enters through $\determ_\varep$.

We now discuss the choice of $\mc{N}_\varep$, but postpone the precise definition until Definition \ref{setup:def-resonant-operator} below. The use of the modified noises causes an additional resonance between certain divergent objects appearing in the right hand side of the $v$ equation \eqref{eq:v-equation}. To the reader familiar with \cite{CHS2018}, this additional resonance arises as the difference between the expectations of the two divergent dipoles, which no longer cancel each other because of our use of the modified noise. A natural idea is to define the nonlinearity $\mc{N}_\varep$ so as to cancel out this additional resonance. 

It turns out that by defining the nonlinearity in this way, there is a nice serendipity: the lift of the equation \eqref{eq:v-equation} for $v_\varep$ to the space of modelled distributions turns out to be {\it exactly the same} as the one given in \cite[(2.14)]{CHS2018}. Given this, and the fact that the initial data for $v_\varep$ does not involve $u_0$, one can hope to estimate $v_\varep$ uniformly in $u_0$ just by using the local theory. Towards this end, the last ingredient that we need is the fact that all modified\footnote{We use the adjective ``modified\revision{''} because the objects are defined using the modified noises $\xi^{\determ_\varep, \varep}_{\pm}$.} stochastic objects appearing in the local theory for $v_\varep$ satisfy estimates which are uniform in $u_0$. The proof of this turns out to be a very slight modification of the stochastic estimates from \cite{CHS2018}\footnote{In the case of monopoles and dipoles, it would suffice to follow the more concrete arguments of the earlier paper \cite{CH16}, as we show in Section \ref{section:objects}} and, as we will see, the only additional steps beyond the existing arguments are to use the probabilistic independence of $\determ_\varep$ and $\zeta$ and to use the trivial estimate $|e^{\icomplex \beta \determ_\varep}| \leq 1$.

\begin{remark}\label{remark:resonant-non-resonant-decomp}
We summarize the preceding discussion as follows. Omitting the smoothing parameter $\varep$, we find a nice decomposition $u = \Phi + \determ + v$, where $\Phi$ is the usual linear stochastic object (and so is globally well-posed), and one should think of the deterministic piece $\determ$ as completely capturing the evolution of the size of $u$, while the remainder $v$ describes order-1 probabilistic fluctuations around this size. In particular, global well-posedness for $u$ will follow\footnote{This is slightly misleading since we will still need an iteration argument to get up to time $1$.} once we have global well-posedness for $\determ$, combined with a stochastic estimate for $v$ which is {\it uniform in the initial data $u_0$}. 

A natural question is whether such nice decompositions exist for other singular SPDEs, e.g., Yang--Mills--Higgs, or sine-Gordon beyond $6\pi$. For the latter, once $\beta^2$ goes beyond $6\pi$, it is unclear to us how to estimate both $\determ$ and $v$. It would seem that the nonlinearity $\mc{N}$ appearing in the $\determ$ equation would require more than one power of $\determ$ to estimate, and this is problematic when trying to show global well-posedness for $\determ$. For the remainder $v$, it is unclear to us whether the estimates for all the needed stochastic objects can be made to be uniform in $u_0$.
\end{remark}

We now briefly summarize the structure of the rest of the paper. In Section \ref{section:setup-and-main}, we set some notation which will be used throughout the paper, and then introduce the main ideas which go into the proof of Theorem~\ref{thm:sg-gwp}. In particular, in Section~\ref{section:ansatz}, we discuss in more detail our Ansatz for $u_\varep$ that will enable us to obtain good estimates. Then in Section~\ref{section:modified-model}, we introduce our modification of the neutral BPHZ model from~\cite{CHS2018}, which is used in the local theory, and which we will ultimately be able to estimate uniformly in the initial data. In Section~\ref{section:main-results}, we outline the main propositions about the quantities introduced in Sections~\ref{section:ansatz} and~\ref{section:modified-model}, and then we use these main propositions to give a proof of Theorem \ref{thm:sg-gwp}. Sections \ref{section:resonant}, \ref{section:objects}, and \ref{section:remainder-bound} are devoted to the proofs of the main propositions. \\

\section{Setup and main propositions}\label{section:setup-and-main}

We set some notation which will be used throughout the paper.

\begin{notation}\label{notation:integer-sets}
\revision{For a positive integer $n$, we write $[n] := \{1, \ldots, n\}$.} 
\end{notation}

Next, we recall the definition of H\"{o}lder-Besov spaces.

\begin{definition}[H\"{o}lder-Besov space]
For $\alpha \in \R$, $f \colon\T^2 \ra \C$, we define the norm
\begin{equs}
\|f\|_{\Cs_x^\alpha} := \sup_{N \in \dyadic} \|P_N f\|_{L_x^\infty},
\end{equs}
where $P_N$ denotes the usual Littlewood-Paley projection to frequency scale $N$, see \cite[Section 2.3.1]{BC23}. We then define the corresponding function space $\Cs_x^\alpha = \Cs_x^\alpha(\T^2)$ to be the closure of $C^\infty_x(\T^2)$ with respect to the $\Cs_x^\alpha$-norm.
\end{definition}

Next, fix $\frks = (2, 1, 1)$ the usual parabolic scaling on $\R \times \T^2$. Let $G$ be the heat kernel with mass $m^2$, i.e., the kernel of the operator $(\ptl_t - \revision{\frac{1}{2}}\Delta + m^2)^{-1}$. As usual, we let $K$ be the singular part of $G$, which is chosen such that all integrals of $K$ against  polynomials of (scaled) degree up to $2$ are equal to zero and such that the difference $G-K$ is smooth. We note that $K$ satisfies
\begin{equation}\label{setup:eq-K-estimate}
|K(z)| \lesssim \| z \|_{\frks}^{-2}. 
\end{equation}
For $\varep \in (0, 1]$, we let $\Jc_\varep$ be as in \cite[(3.5)]{HS16}, 
and let $\Jc_\varep^{-}$ be defined as $1/\Jc_\varep$. From \cite[(3.7)]{HS16}, it then follows that uniformly in $\varep \in (0, 1]$, 
\begin{equation}\label{setup:eq-J-estimate}
|\Jc_\varep(z)| \sim (\| z \|_{\frks} + \varep)^{\frac{\beta^2}{2\pi}} \qquad \text{and} \qquad
|\Jc_\varep^-(z)| \sim  (\| z \|_{\frks} + \varep)^{-\frac{\beta^2}{2\pi}}.
\end{equation}
We remark that $\Jc_\varep, \Jc_\varep^{-}$ can also be defined as:
\begin{equs}\label{eq:Jc-covariance}
\Jc_\varep(z) = \E[\xi^\varep_{\pm}(0) \xi^\varep_{\pm}(z)], \quad \Jc_\varep^-(z) = \E[\xi^\varep_{\pm}(0) \xi^\varep_{\mp}(z)],
\end{equs}
where $\xi^\varep_{\pm} := C_{\beta, \rho, \varep} e^{\pm \icomplex \beta \Phi_\varep}$, with $C_{\beta, \rho, \varep}$ as in \eqref{eq:gmc-renorm-constant}. We also fix parameters
\begin{equs}
\kappa &:= 10^{-3} (6\pi - \beta^2), \label{eq:kappa}\\
\bar{\beta} &:= \frac{\beta^2}{4\pi} + \kappa, \label{eq:bar-beta}\\
\eta &\in \Big(\frac{\bar{\beta}}{2} - 1, 0\Big), \label{eq:eta}
\end{equs}
to be used in various places later on. 

Following \cite{CHS2018}, we assume that the mollifier $\rho \colon \R \times \T^2 \ra \R$ is supported on the ball of radius $1$, integrates to $1$, and satisfies $\rho(t, x) = \rho(t, -x)$ for all $t \in \R$, $x \in \T^2$. In this paper, we will make the additional assumption that $\rho$ is supported at negative times, i.e., $\rho = 0$ on $[0, \infty) \times \T^2$. 

\begin{remark}\label{remark:noise-only-positive-times-enter}
Due to our assumptions on $\rho$, in particular the fact that it is supported at negative times, we see that only $\zeta |_{\revision{[}0, \infty)}$ enters into the equation for $u_\varep$. It follows that the limit $u$ itself is also only a function of $\zeta |_{\revision{[}0, \infty)}$ and the initial data $u_0$. This will be convenient in our iteration argument later on -- see the proof of Proposition \ref{prop:iteration-estimate} in Section \ref{section:main-results}.
\end{remark}

\subsection{Ansatz}\label{section:ansatz}

In this section, we go into more detail about our Ansatz for the solution $u$ to 
\begin{equs}\label{eq:SG-general-starting-time}
(\ptl_t - \revision{\tfrac{1}{2}}\Delta + m^2) u &= \biglcol \, \sin(\beta u)\,\bigrcol + \zeta, ~~u(t_0) = u_0.
\end{equs}
Here, we take a general starting time $t_0$, which will be notationally convenient for our iteration argument later. Throughout, we assume that $u_0$ is independent of $\zeta |_{[t_0, \infty)}$, which is slightly more general than the assumption in Theorem \ref{thm:sg-gwp} that $u_0$ is deterministic. The solution $u$ is defined as the limit in $C_t^0 \Cs_x^{\eta}$ of the solutions $u_\varep$ to the regularized equations
\begin{equs}\label{eq:smoothed-SG-general-starting-time}
(\ptl_t - \revision{\tfrac{1}{2}}\Delta + m^2) u_\varep &= \biglcol \, \sin(\beta u_\varep) \, \bigrcol + \zeta_\varep, \quad u_\varep(t_0) = u_{0, \varep}.
\end{equs}
(Recall that the renormalized sine is defined in \eqref{eq:renormalized-sine}, and that $u_{0, \varep}$ is some mollification of $u_0$ at scale $\varep$.) 

\begin{remark}\label{remark:replacement-noise}
By Remark \ref{remark:noise-only-positive-times-enter}, upon replacing $\zeta$ by the noise $\tilde{\zeta} := \bar{\zeta}_{< t_0} + \zeta |_{[t_0, \infty)}$, where $\bar{\zeta}_{< t_0}$ is a space-time white noise on $(-\infty, t_0) \times \T^2$ which is independent of everything else, we have that $u$ is equal to the limit in $C_t^0 \Cs_x^{\eta}$ of the solutions $\tilde{u}_\varep$ to the same regularized equations, except now with $\tilde{\zeta}_\varep$ in place of $\zeta_\varep$. Since $u_0$ is assumed to be independent of $\zeta |_{[t_0, \infty)}$, it follows that $u_0$ is independent of $\tilde{\zeta}$. Thus by possibly modifying $\zeta$ at times before $t_0$, we may always ensure that $u_0$ is independent of $\zeta$, without changing $u$. This is a key point that we will use in our iteration argument later on.
\end{remark}

Next, we give the precise form of the nonlinearity $\mc{N}_\varep$, 
which was previously mentioned in the introduction and which plays an important role in the proof of Theorem~\ref{thm:sg-gwp}. We will now denote this nonlinearity by $\Res_\varep$ in the following, short for ``resonant operator\revision{''}.

\begin{definition}[Resonant operator]\label{setup:def-resonant-operator}
For any $\varep\in (0,1]$, $t_0 \geq 0$, $T>0$, and function $\determ \colon [t_0,t_0 +T]\times \T^2 \rightarrow \R$, we define the function $\Res_\varep\big( \determ \big)\colon  [t_0,t_0+T]\times \T^2 \rightarrow \R$ by
\begin{equation}\label{setup:eq-resonant-operator}
\Res_\varep \big( \determ \big) (z) := -\frac{\beta}{2} \int \dzprime K(\zprime,z) \Jc_\varep^{-}(z-\zprime) \sin \big( \beta (\determ(z)-\determ(\zprime)) \big). 
\end{equation}
\end{definition}

With this definition in hand, our Ansatz for $u_\varep$ is as follows. First, let
\begin{equs}\label{eq:linear-part}
\Phi_\varep := K \ast \zeta_\varep, \quad R_\varep := (G-K) \ast \zeta_\varep. 
\end{equs}
We then let
\begin{equs}
u_\varep = \Phi_\varep + \determ_\varep + v_\varep,
\end{equs}
where the ``resonant piece\revision{''} $\determ_\varep$ satisfies the equation 
\begin{equation}\label{setup:eq-resonant-equation}
\begin{aligned}
(\partial_t - \revision{\tfrac{1}{2}}\Delta + m^2) \determ_\varep &= \Res_\varep ( \determ_\varep ), \\
\determ (t_0)&=u_{0, \varep}, 
\end{aligned}
\end{equation}
and the remainder $v_\varep$ satisfies the equation (recall the modified noises $\xi^{\determ_\varep, \varep}_{\pm}$ from \eqref{eq:modified-noise})
\begin{equation}
\begin{aligned}
(\ptl_t - \revision{\tfrac{1}{2}}\Delta + m^2) v_\varep &= \frac{1}{2\icomplex} \big( e^{\icomplex \beta v_\varep} \xi_+^{\determ_\varep, \varep}  - e^{-\icomplex \beta v_\varep}\xi_{-}^{\determ_\varep, \varep} \big) - \Res_\varep(\determ_\varep) + R_\varep, \\
v_\varep(t_0) &= -\Phi_\varep(t_0, \cdot).
\end{aligned}
\end{equation}
Following \cite{CHS2018}, we apply a pre-processing step and further write $v_\varep = G v_\varep(0) + w_\varep$, where $w_\varep$ then solves
\begin{align}
(\ptl_t - \revision{\tfrac{1}{2}}\Delta + m^2) w_\varep &= \frac{1}{2\icomplex} \Big(e^{\icomplex \beta G v_\varep(0)} e^{\icomplex \beta w_\varep} \xi_+^{\determ_\varep, \varep} - e^{-\icomplex \beta G v_\varep(0)} e^{-\icomplex \beta w_\varep} \xi_-^{\determ_\varep, \varep}\Big)  -\Res_\varep(\determ_\varep) + R_\varep, \notag \\
w_\varep(t_0) &= 0.\label{eq:w-equation}
\end{align}
In the following, we refer to \eqref{setup:eq-resonant-equation} as the resonant equation. We extend $\determ_\varep$ to $\R \times \T^2$ by defining $\determ_\varep(t) := \determ_\varep(t_0)$ for $t < t_0$. Since $\determ_\varep(t_0) = u_{0, \varep}$ is continuous (in fact, smooth), we have that the extension $\determ_\varep \colon\R \times \T^2 \ra \C$ is continuous. In summary, our final Ansatz for $u_\varep$ is
\begin{equs}\label{eq:ansatz}
u_\varep = \Phi_\varep + \determ_\varep - G \Phi_\varep(t_0) + w_\varep.
\end{equs}
Compared with \cite{CH16, CHS2018}, the main difference in our Ansatz is the isolation of the resonant piece $\determ_\varep$, which then affects the equation for $v_\varep$, and thus also the equation for $w_\varep$.

\subsection{Modified model}\label{section:modified-model}

In this subsection, we introduce our modified model. One of the key ideas of this paper is that our modified model will satisfy stochastic estimates which are uniform (in our modification parameter $\determ$), and this will allow us to show estimates for $v$ which are uniform in the initial data $u(0)$ (as claimed in Remark \ref{remark:resonant-non-resonant-decomp}). We will assume that $\beta^2 \in [\frac{16\pi}{3}\pi, 6\pi)$ in the ensuing discussion, as the case $\beta^2 \in [4\pi, \frac{16\pi}{3})$ is strictly simpler. 

\begin{remark}\label{remark:blackbox-0}
Because of our use of the modified noise $\xi^{\determ_\varep, \varep}_{\pm}$ from \eqref{eq:modified-noise}, which in particular is not stationary in law, it is not completely clear to us how to fit our setting into the general black-box results of \cite{H14, CH16, BHZ19}. See also Remark \ref{remark:blackbox}. There are some results which handle non-stationary noises, see e.g., \cite{BB21Locality, HS2023}.~However, the modified noise is also extremely non-Gaussian, and it would seem that (similar to what is done in \cite{CHS2018}) one would have to adapt the analytic part of the black-box to handle this. On the other hand, there are not too many objects when $\beta^2 < 6\pi$, so it is feasible to be very concrete and perform the needed calculations by hand, which is what we mostly do (especially in Appendices \ref{appendix:modified-model-calculations} and \ref{appendix:more-calculations}).
\end{remark}

First, we review the regularity structure preliminaries from \cite{CHS2018} which are needed to show local existence for the dynamical sine-Gordon model. We closely follow the beginning of \cite[Section~2]{CHS2018}. Recall the parabolic scaling $\mathfrak{s} = (2, 1, 1)$. Recall from \eqref{eq:bar-beta} that $\bar{\beta} = \frac{\beta^2}{4\pi} + \kappa$, where $\kappa = 10^{-3}(6\pi - \beta^2)$ is as in \eqref{eq:kappa}. We further fix a parameter $\mu = \bar{\beta} + \kappa = \frac{\beta^2}{4\pi} + 2\kappa$. Let $\mc{T}$ be as defined in \cite[(2.3)]{CHS2018}, i.e., it is the collection of all rooted decorated trees $T^{\mathfrak{n} \mathfrak{l}}$ with homogeneity $|T^{\mathfrak{n} \mathfrak{l}}|_{\mathfrak{s}} < \mu$. Here, given a tree $T$, we denote its set of nodes by $N(T)$, its root by $\varrho_T \in N(T)$, and its set of edges by $K(T)$. The decorations $\mathfrak{n}, \mathfrak{l}$ are maps $\mathfrak{n} \colon N(T) \ra \N^3$ and $\mathfrak{l} \colon N(T) \ra \{+, 0, -\}$. We further define $L(T) := \{u \in N(T) \colon \mathfrak{l}(u) \neq 0\}$. The homogeneity of a tree is
\begin{equs}
|T^{\mathfrak{n} \mathfrak{l}}|_{\mathfrak{s}} := 2|K(T)| - \bar{\beta} \revision{|}L(T)\revision{|} + \sum_{u \in N(T)} |\mathfrak{n}(u)|_{\mathfrak{s}}.
\end{equs}
Let $\mathscr{T}$ be the free vector space generated by $\mc{T}$, and let $(\mathscr{T}, A, G)$ be the associated regularity structure (see the end of \cite[Section 2.2]{CHS2018}). When the polynomial decoration is zero, we will often draw trees $T^{0, \mathfrak{l}}$ as actual trees, e.g., 
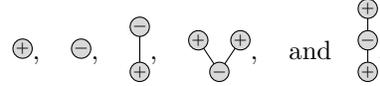
\begin{equs}
\monop, \quad \monom, \quad \dippm, \quad \vtripolemmp, \quad \text{and} \quad \ltripolepmp.
\end{equs}
We will often also use the alternative notation built out of $X, \Xi$, and $\mc{I}$ to represent such trees. For instance, the trees above would be written
\begin{equs}
\Xi_+, \quad \Xi_-, \quad \Xi_+ \mc{I} \Xi_-, \quad  \Xi_- \mc{I} \Xi_+ \mc{I} \Xi_+, \quad \text{and} \quad \Xi_+ \mc{I} (\Xi_- \mc{I} \Xi_+).
\end{equs}
This notation has the advantage of also being able to encode polynomial decorations. Let $\mc{T}^-$ be the trees in $\mc{T}$ with negative homogeneity. Since we are assuming $\beta^2 \in [\frac{16\pi}{3}, 6\pi)$, we have that \revision{(recall from Notation \ref{notation:integer-sets} that $[n] = \{1, \ldots, n\}$)}
\begin{equs}
\mc{T}^- = \Big\{\Xi_{s_1}, X_j \Xi_{s_1}, \Xi_{s_1} \mc{I}\Xi_{s_2}, \Xi_{s_1} \mc{I} \Xi_{s_2} \mc{I} \Xi_{s_3}, \Xi_{s_1} \mc{I}(\Xi_{s_2} \mc{I} \Xi_{s_3}) : j \in [2], s_1, s_2, s_3 \in \{+, -\} \Big\}.
\end{equs}
In the following, we will use the term ``monopole\revision{''} (resp. ``dipole\revision{''}) to denote one of the trees:
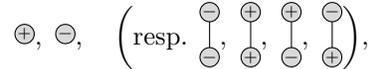
\begin{equs}
\monop, ~~\monom, \quad \bigg( \text{resp.   } \dipmm, ~~\dippp, ~~\dipmp, ~~\dippm \bigg),
\end{equs}
and the term ``tripole\revision{''} to denote one of the following trees:
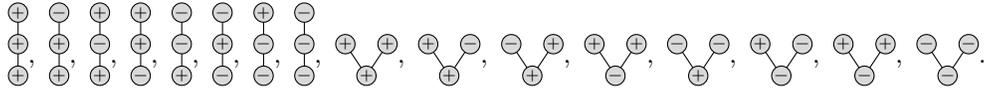
\begin{equs}
\ltripoleppp, ~~ \ltripoleppm, ~~ \ltripolepmp, ~~ \ltripolempp, ~~ \ltripolepmm, ~~ \ltripolempm, ~~ \ltripolemmp, ~~ \ltripolemmm, ~~ \vtripoleppp, ~~\vtripoleppm, ~~ \vtripolepmp, ~~ \vtripolempp, ~~ \vtripolepmm, ~~ \vtripolempm, ~~ \vtripolemmp, ~~ \vtripolemmm.
\end{equs}
We note that the set $\mc{T}^-$ consists precisely of all monopoles, dipoles, and tripoles, combined with the decorated monopoles $X_j \Xi_s$, $j \in [2], s \in \{+, -\}$.

Having reviewed the necessary preliminaries, we now introduce our modified (pre-)model. In the following, let $\determ \colon \R \times \T^2 \ra \C$ be a continuous function. Later on, we will take $\determ = \determ_\varep$, where $\determ_\varep$ is as in \eqref{setup:eq-resonant-equation}, which is why we use the notation $\determ$. Recalling $\Phi_\varep$ from \eqref{eq:linear-part}, we let (clearly this is inspired by \eqref{eq:modified-noise})
\begin{equs}
\xi_+^{\determ, \varep} := e^{\icomplex \beta \determ} \biglcol\, e^{\icomplex \beta \Phi_\varep} \,\bigrcol, \quad \xi_-^{\determ, \varep} := e^{-\icomplex \beta \determ} \biglcol\, e^{-\icomplex \beta \Phi_\varep} \,\bigrcol.
\end{equs}
We now define a modified pre-model $\premodel^{\determ, \varep}$ on $\mc{T}^-$ as follows. First, the action on the decorated monopoles $X_j \Xi_s$ is determined because we want $\premodel^{\determ, \varep}$ to be admissible, i.e., so that $\premodel^{\determ, \varep}(X_j \Xi_s) = \cdot_j \premodel^{\determ, \varep} \Xi_s$. Thus, it remains to specify the modified pre-model on the monopoles, dipoles and tripoles. To save space, we note that each tree has a mirror image obtained by flipping the sign of every node, and so we only specify the pre-model on half of the trees. The pre-model on the remaining half can be obtained by flipping all pluses to minuses (and vice versa) in the definitions. Next, the reader familiar with \cite{CHS2018} should keep in the mind that in the case $\theta \equiv 0$, the modified pre-model $\premodel^{\determ, \varep}$ is precisely the pre-model which induces the neutral BPHZ model defined in \cite{CHS2018}. With these considerations in mind, we let \revision{(in the following, all expectations are with respect to the Gaussian noise $\xi$ only)}
\begin{align}
\premodel^{\determ, \varep} \monop &:= \xi_+^{\determ, \varep}, \label{eq:neutral-bphz-premodel-begin-def}\\
\premodel^{\determ, \varep} \dippm &:= \xi_+^{\determ, \varep} (K \ast \xi_-^{\determ, \varep}) - \E[\xi_+^{\determ, \varep} (K \ast \xi_-^{\determ, \varep})], \quad  \premodel^{\determ, \varep} \dippp := \xi_+^{\determ, \varep} (K \ast \xi_+^{\determ, \varep}), \\
\premodel^{\determ, \varep} \vtripoleppp &:= \xi_+^{\determ, \varep} (K \ast \xi_+^{\determ, \varep})^2, \\
\premodel^{\determ, \varep} \vtripolepmm &:= \xi_+^{\determ, \varep} (K \ast \xi_-^{\determ, \varep})^2 - 2 (K \ast \xi_-^{\determ, \varep}) \E[\xi_+^{\determ, \varep} (K \ast \xi_-^{\determ, \varep})],\\
\premodel^{\determ, \varep} \vtripolepmp &:= \premodel^{\determ, \varep} \vtripoleppm := \big(\xi_+^{\determ, \varep} (K \ast \xi_-^{\determ, \varep}) - \E[\xi_+^{\determ, \varep} (K \ast \xi_-^{\determ, \varep})]\big) (K \ast \xi_+^{\determ, \varep}), \\
\premodel^{\determ, \varep} \ltripoleppp &:= \xi_+^{\determ, \varep} K \ast (\xi_+^{\determ, \varep} (K \ast \xi_+^{\determ, \varep})), \\
\premodel^{\determ, \varep} \ltripoleppm &:= \xi_+^{\determ, \varep} K \ast \big(\xi_+^{\determ, \varep} (K \ast \xi_-^{\determ, \varep}) - \E[\xi_+^{\determ, \varep} (K \ast \xi_-^{\determ, \varep})]\big), \\
\premodel^{\determ, \varep} \ltripolempp &:= \xi_-^{\determ, \varep} K \ast (\xi_+^{\determ, \varep} (K \ast \xi_+^{\determ, \varep})) - \E[\xi_-^{\determ, \varep} (K \ast \xi_+^{\determ, \varep})] (K \ast \xi_+^{\determ, \varep}), \\
\premodel^{\determ, \varep} \ltripolepmp &:= \xi_+^{\determ, \varep} K \ast \big(\xi_-^{\determ, \varep} (K \ast \xi_+^{\determ, \varep}) - \E[\xi_-^{\determ, \varep} (K \ast \xi_+^{\determ, \varep})]\big) - \E[\xi_+^{\determ, \varep} (K \ast \xi_-^{\determ, \varep})] (K \ast \xi_+^{\determ, \varep}). \label{eq:neutral-bphz-premodel-end-def}
\end{align}
We extend $\premodel^{\determ, \varep}$ to the other trees in $\mc{T}$ in the usual way, so that $\premodel^{\determ, \varep}$ is admissible (see e.g. the proof of \cite[Theorem 10.7]{H14}). In the following, let $(\{\model_x^{\determ, \varep}\}_{x \in \R \times \T^2}, \{\Gamma_{xy}^{\determ, \varep}\}_{x, y \in \R \times \T^2}) = \mc{Z}(\premodel^{\determ, \varep})$ be the families of recentered and recentering maps obtained from the modified premodel $\premodel^{\determ, \varep}$. 

\begin{remark}\label{remark:blackbox}
As far as we are aware, due to our non-stationary noises $\xi_{\pm}^{\determ, \varep}$, the black-box results of \cite{BHZ19} do not apply here to automatically conclude that $\mc{Z}(\premodel^{\determ, \varep})$ is indeed a model. The fact that we (almost surely) do in fact have a model will follow from the stochastic estimates of Proposition \ref{prop:stochastic-estimates}.
\end{remark}

In Appendix \ref{appendix:modified-model-calculations}, we will explicitly compute the maps $\{\model^{\determ, \varep}_x\}_{x \in \R \times \T^2}$, which in particular will show that the case $\theta \equiv 0$ gives the neutral BPHZ model of \cite{CHS2018}. The formula we obtain for $\{\model^{\determ, \varep}_x\}_{x \in \R \times \T^2}$ will be needed in proving the stochastic estimates for our modified model. 

\subsection{Main results}\label{section:main-results}

In this section, we introduce the main intermediate results needed in the proof of Theorem \ref{thm:sg-gwp}, and then give a proof of Theorem \ref{thm:sg-gwp} using these intermediate results. Our first main result concerns the global well-posedness of the solution to the resonant equation \eqref{setup:eq-resonant-equation}, and its proof is the content of Section~\ref{section:resonant}.

\begin{proposition}[Resonant equation]\label{setup:prop-resonant-equation} 
Let $4\pi \leq \beta^2<6\pi$, let $0<T\leq T_0$, where $T_0$ is sufficiently small depending on $m$, $\beta$ and $\eta$, and let $u_0 \in \Cs_x^{\eta}$. 
For all $\varep \in (0,1]$, the resonant equation \eqref{setup:eq-resonant-equation} then has a unique solution $\determ_\varep$ in the function space $\Snorm([0,T])$, which is as in Definition \ref{resonant:def-norms} below. Furthermore, for all $\delta>0$, we have the uniform estimate
\begin{equs}\label{setup:eq-resonant-equation-estimate}
\sup_{\varep \in (0,1]} \|\determ_\varep(T)\|_{\Cs_x^{\eta}} \leq e^{-m^2 T} \| u_0 \|_{\Cs_x^{\eta}} + C T^{2-\frac{\beta^2}{4\pi}+\frac{\eta}{2}} \big( \| u_0 \|_{\Cs_x^{\eta}} +1 \big)^{\frac{\beta^2}{2\pi}-2+\delta}, 
\end{equs}
where $C$ depends only on $m$, $\beta$, $\delta$, and  $\eta$. 
\end{proposition}

We remark that the condition $\beta^2 \geq 4\pi$ in Proposition \ref{setup:prop-resonant-equation} has only been imposed since, in the case $\beta^2<4\pi$, the exponents in \eqref{setup:eq-resonant-equation-estimate} need to take a slightly different form. Of course, the case $\beta^2<4\pi$ is much easier than the case $4\pi \leq \beta^2 < 6\pi$, since then the expression $K(\zprime,z) \Jc_\varep^{-}(z-\zprime)$ from \eqref{setup:eq-resonant-operator} is then integrable in~$\zprime$. We also emphasize that, since $\beta^2<6\pi$, the exponent of the last term in \eqref{setup:eq-resonant-equation-estimate} can be made strictly smaller than one. 

Next, we introduce our main result regarding the stochastic estimates for our modified model. It is the analog of \cite[Theorem 2.8]{CHS2018} (except we restrict to $\beta^2 < 6\pi$, as is the case throughout this paper). Its proof is the content of Section \ref{section:objects} and Appendix \ref{section:appendix-moments}. As we will see in those sections, the proof of this proposition is a minor modification of the proof of \cite[Theorem~2.8]{CHS2018}, and the modification essentially boils down to the trivial estimate $|e^{\icomplex \beta \determ}| \leq 1$.

\begin{proposition}[Stochastic estimates]\label{prop:stochastic-estimates}
Let $\beta^2 < 6\pi$. Let $\tau \in \mc{T}^-$. For any $p \geq 1$, there exists $C_{\tau, p}$ such that
\begin{equs}
\E \Big[ \Big| \model^{\determ, \varep}_z \tau (\psi^\lambda_z)\Big|^{2p} \Big] \leq C_{\tau, p} \lambda^{2p (|\tau|_{\mathfrak{s}} + \delta)},
\end{equs}
for some sufficiently small $\delta > 0$, uniformly in continuous functions $\determ \colon \R \times \T^2 \ra \C$, $\varep \in (0, 1]$, $\lambda \in (0, 1]$, and all continuous test functions $\psi$ supported on the unit ball in $\R \times \T^2$ with $L^\infty$ norm bounded by $1$.
\end{proposition}

As alluded to in Remark \ref{remark:resonant-non-resonant-decomp}, by combining the definition of $\determ_\varep$ as the solution to \eqref{setup:eq-resonant-equation} and the stochastic estimates of Proposition \ref{prop:stochastic-estimates}, we will be able to show the following bound for the remainder $w_\varep$, which crucially is uniform in the initial data. The proof of this estimate is in Section \ref{section:remainder-bound}.

\begin{proposition}[Remainder estimate]\label{prop:remainder}
For all $q \geq 1$, there exists a constant $A_q$ such that the following holds for any $\varep \in (0, 1]$, $J \geq 1$, and any random distribution\footnote{Note that $u_0$ enters into the equation for $w_\varep$ through $\determ_\varep$.} $u_0 \in \Cs_x^{\eta}$ which is independent of $\zeta$. Letting $w_\varep$ be defined as in \eqref{eq:w-equation}, we have that
\begin{equs}
\p\big( \|w_\varep\|_{C_t^0 \Cs_x^{2 - \bar{\beta}}([0, 1/J])} > 1 \big) \leq A_q J^{-q}.
\end{equs}
\end{proposition}

By combining Propositions \ref{setup:prop-resonant-equation} and \ref{prop:remainder}, we will prove the following estimate, which will allow us to easily iterate in time. Indeed, Theorem \ref{thm:sg-gwp} will quickly follow once we have established the iteration estimate.

\begin{proposition}[Iteration estimate]\label{prop:iteration-estimate}
Let $\ovl{m}^2 < m^2$. There is a constant $C_1 \geq 1$ which only depends on $\beta^2$, $m^2$, $\ovl{m}^2$, $\eta$, such that the following hold. Let $t_0 \geq 0$ and suppose that $u_0 \in \Cs_x^{\eta}(\T^2)$ is independent of $\zeta |_{\revision{[}t_0, \infty)}$. Let $u$ be the solution to \eqref{eq:SG-general-starting-time}, defined as the limit of the regularized equations \eqref{eq:smoothed-SG-general-starting-time}. Then for any $q \geq 1$, there is a constant $B_q < \infty$ which is uniform in $u_0$ such that for any integer $J \geq 1$, 
\begin{equs}
\p\Big(\|e^{m^2(t - t_0)} u\|_{C_t^0 \Cs_x^{\eta}([t_0, t_0 + 1/J])} > e^{\ovl{m}^2 / J} \|u_0\|_{\Cs_x^{\eta}} + C_1 J^{C_1} \Big) \leq B_q J^{-q}.
\end{equs}
\end{proposition}
\begin{proof}
For notational simplicity, take $t_0 = 0$. Recalling Remarks \ref{remark:noise-only-positive-times-enter} and \ref{remark:replacement-noise}, if we replace $\zeta$ by $\tilde{\zeta} |_{(-\infty, 0) \times \T^2} + \zeta|_{\revision{[}0, \infty) \times \T^2}$, where $\tilde{\zeta}$ is a space-time white noise which is independent of everything else, then this does not change $u_\varep$, and thus neither does it change the limit $u$. Note moreover that $u_0$ is independent of this modified noise. We may therefore assume that $u_0$ is independent of $\zeta$, which will eventually allow us to apply Proposition~\ref{prop:remainder}.
In the following, take $c_0 = m^2$ and $c_1 = \ovl{m}^2$. By possibly adjusting the constant $B_q$, it suffices to assume that $J$ is large enough (depending on $\beta^2, m^2, \eta$) so that $1/J \leq T_0$, where $T_0$ is as in Proposition~\ref{setup:prop-resonant-equation}.

First, by \cite[Theorem 1.1]{CHS2018}, we have that the solutions $u_\varep$ converge to $u$ in the following sense. There is a stopping time $\tau$ such that for all $T > 0$, we have that $\|u_\varep - u\|_{C_t^0 \Cs_x^{\eta}([0, T])} \ra 0$ in probability on the event $\{\tau > T\}$. Furthermore, on the event $\{\tau < \infty\}$, we have that $\lim_{t \uparrow \tau} \|u(t)\|_{\Cs_x^{\eta}(\T^2)} = \infty$. Using this, we have that for all $T > 0$, there is a subsequence $\{u_{\varep_k}\}$ such that on the event $\{\tau > T\}$, we have that $\|u_{\varep_k} - u\|_{C_t^0 \Cs_x^{\eta}([0, T])} \ra 0$ a.s. Then by a diagonal argument, we may obtain a subsequence $\{u_{\varep_k}\}$ such that for all $T \in \mathbb{Q}$, $T > 0$, we have that $\|u_{\varep_k} - u\|_{C_t^0 \Cs_x^{\eta}([0, T])} \ra 0$ a.s. on the event $\{\tau > T\}$. Now with this subsequence, we claim that for all $T \in \mathbb{Q}$, $T > 0$, we have that
\begin{equs}\label{eq:liminf-bound}
\|e^{c_0 t} u\|_{C_t^0 \Cs_x^{\eta}([0, T])} \stackrel{a.s.}{\leq} \liminf_{k} \|e^{c_0 t} u_{\varep_k}\|_{C_t^0 \Cs_x^{\eta}([0, T])}.
\end{equs}
To see this, note that if $\tau > T$, then this follows directly by the a.s. convergence. If $\tau \leq T$, then given $N \geq 1$, take $s < \tau$ such that $\|u(s)\|_{\Cs_x^{\eta}(\T^2)} \geq N$. Taking $s_0 \in [s, \tau)$ rational, it then follows that
\begin{equs}
\liminf_{k} \|e^{c_0 t} u_{\varep_k}\|_{C_t^0 \Cs_x^{\eta}([0, T])} \geq \liminf_{k} \|e^{c_0 t} u_{\varep_k}\|_{C_t^0 \Cs_x^{\eta}([0, s_0])} = \|e^{c_0 t} u\|_{C_t^0 \Cs_x^{\eta}([0, s_0])} \geq N.
\end{equs}
Since $N$ was arbitrary, we obtain the desired estimate \eqref{eq:liminf-bound} when $\tau \leq T$ also.

Fix an integer $J \geq 1$. For any $\lambda \geq 0$, it follows by \eqref{eq:liminf-bound} and Fatou's lemma that
\begin{equs}\label{eq:probability-liminf-bound}
\p\Big(\|e^{c_0 t} u\|_{C_t^0 \Cs_x^{\eta}([0, 1/J])} > \lambda \Big) \leq \liminf_k \p\Big( \|e^{c_0 t} u_{\varep_k}\|_{C_t^0 \Cs_x^{\eta}([0, 1/J])} > \lambda  \Big).
\end{equs}
Recalling our ansatz \eqref{eq:ansatz} for $u_\varep$, we have that
\begin{equs}
&\,\|e^{c_0 t} u_\varep\|_{C_t^0 \Cs_x^{\eta}([0, 1/J])} \\
\leq&\, \|e^{c_0 t} (\Phi_\varep - G \Phi_\varep(0))\|_{C_t^0 \Cs_x^{\eta}([0, 1/J])} + \|e^{c_0 t}  \determ_\varep\|_{C_t^0 \Cs_x^{\eta}([0, 1/J])} + \|e^{c_0 t} w_\varep\|_{C_t^0 \Cs_x^{\eta}([0, 1/J])}.
\end{equs}
By standard stochastic estimates (see e.g. \cite[Lemma 5.6]{BC23}),  we have that
\begin{equs}
\p\Big( \|e^{c_0 t} (\Phi_\varep - G \Phi_\varep(0))\|_{C_t^0 \Cs_x^{\eta}([0, 1/J])} > J \Big) \leq C e^{-c J^2},
\end{equs}
where $C, c$ are uniform in $\varep$.
By Proposition \ref{setup:prop-resonant-equation} with $\delta > 0$ chosen so that $\frac{\beta^2}{2\pi} -2 + \delta = 1-\kappa$ (recall from \eqref{eq:kappa} that $\kappa = 10^{-3}(6\pi - \beta^2)$), we have that (throwing away the positive power of $T = 1/J$ in the first line)
\begin{equs}
\|e^{c_0 t}  \determ_\varep\|_{C_t^0 \Cs_x^{\eta}([0, 1/J])}  &\leq \|u_0\|_{\Cs_x^\eta} +  C e^{c_0} (\|u_0\|_{\Cs_x^\eta} + 1)^{1-\kappa} \\
&= \|u_0\|_{\Cs_x^\eta} + (c_1 J^{-1} (\|u_0\|_{\Cs_x^\eta} + 1))^{1-\kappa} C e^{c_0} (c_1 J^{-1})^{-(1-\kappa)} \\
&\leq \|u_0\|_{\Cs_x^\eta} + c_1 J^{-1} \|u_0\|_{\Cs_x^\eta} + c_1 J^{-1} + \big(C e^{c_0} (c_1 J^{-1})^{-(1-\kappa)}\big)^{\frac{1}{\kappa}} \\
&\leq e^{c_1 / J} \|u_0\|_{\Cs_x^\eta} + \frac{C_1}{2} J^{C_1},
\end{equs}
where we took $C_1$ sufficiently large depending on $C, c_0, c_1, \kappa, \eta$, and where we used Young's inequality in the third line. Finally, using that $u_0$ is independent of $\zeta$, and applying the remainder estimate Proposition~\ref{prop:remainder} (and also using that $C_1$ may be taken large enough so that $e^{c_0} \leq \frac{C_1}{4}$), we have that for all $q \geq 1$, there is a constant $A_q$ which is uniform in $u_0$, $\varep$, and $J$ such that 
\begin{equs}
\p\Big(\|e^{c_0 t} w_\varep\|_{C_t^0 \Cs_x^{\eta}([0, 1/J])} > \frac{C_1}{4} \Big) = \E \Big[\p \Big( \|e^{c_0 t} w_{\varep}\|_{C_t^0 \Cs_x^{\eta}([0, 1/J])} > \frac{C_1}{4} ~\Big|~ u_0 \Big)\Big] \leq A_q J^{-q}.
\end{equs}
By combining the previous few estimates, we obtain that for all $\varep > 0$,
\begin{equs}
\p\Big( \|e^{c_0 t} u_{\varep}\|_{C_t^0 \Cs_x^{\eta}([0, 1/J])} > C_1 J^{C_1} + e^{c_1/J} \|u_0\|_{\Cs_x^{\eta}}  \Big) \leq A_q J^{-q}  + C e^{-c J^2} \leq B_q J^{-q}.
\end{equs}
The desired result now follows by \eqref{eq:probability-liminf-bound}.
\end{proof}

\begin{proof}[Proof of Theorem \ref{thm:sg-gwp}]
In the following, take $c_0 = m^2, c_1 = m^2/2$. For $J \in \dyadic$ and for $0 \leq j \leq J - 1$, define $t_j := j / J$ and
\begin{equs}
E_j^J := \{\tau > t_j\} \cap \Big\{\|e^{c_0 (t - t_j)} u\|_{C_t^0 \Cs_x^{\eta}([t_j, t_{j+1}])} \leq e^{c_1/J} \|u(t_j)\|_{\Cs_x^{\eta}} + C_1 J^{C_1} \Big\}. 
\end{equs}
Define also $E_{\leq j}^J := \bigcap_{i=0}^j E_{i}^J$. On the event $E^J_{\leq j}$, we have that 
\begin{equation}\label{eq:E-J-j-bound}
\begin{aligned}
\|e^{c_0 t} u\|_{C_t^0 \Cs_x^{\eta}([0, t_{j+1}])} &\leq \max\big(\|e^{c_0 t} u\|_{C_t^0 \Cs_x^{\eta}([0, t_j])},  \|e^{c_0 t} u\|_{C_t^0 \Cs_x^{\eta}([t_j, t_{j+1}])}\big) \\
&\leq \max\big(\|e^{c_0 t} u\|_{C_t^0 \Cs_x^{\eta}([0, t_j])}, e^{c_0 t_j} e^{c_1/J} \|u(t_j)\|_{\Cs_x^{\eta}} + e^{c_0t_j} C_1 J^{C_1}\big) \\
&\leq  e^{c_1 / J} \|e^{c_0 t} u\|_{C_t^0 \Cs_x^{\eta}([0, t_j])} + C_1 e^{c_0} J^{C_1}.
\end{aligned}
\end{equation}
Thus by an inductive argument, we have that on the event $E^J_{\leq J-1}$,
\begin{equs}\label{eq:iteration-final-bound}
\|e^{c_0 t} u\|_{C_t^0 \Cs_x^{\eta}([0, 1])} &\leq e^{c_1} \|u_0\|_{\Cs_x^{\eta}} + C_1 e^{c_0}  J^{C_1} (1 + e^{c_1/J} + \cdots + e^{c_1}) \\
&\leq e^{c_1} \|u_0\|_{\Cs_x^{\eta}} + C_1 e^{c_0 + c_1} J^{C_1 + 1}.
\end{equs}
\revision{Moreover, by a similar inductive argument, we have that $\tau > t_j$ on the event $E_{\leq j-1}^J$. Indeed, on the latter event, we have upon iterating \eqref{eq:E-J-j-bound} that $u$ does not blow up on $[0, t_j]$, which implies that $\tau > t_j$ (i.e. that blowup can only happen after time $t_j$).} Next, for any $q \geq 1$, we have that 
\begin{equation}\label{eq:probability-iteration-estimate}
\begin{aligned}
\p((E^J_{\leq J - 1})^c) ~\revision{=}~ \p((E_0^J)^c) + \sum_{j=1}^{J-1} \p(E_{\leq j-1}^J \backslash E_j^J) \leq B_q J^{1-q},
\end{aligned}
\end{equation}
where we applied Proposition \ref{prop:iteration-estimate} in the last step. Here, we also used that $u(t_j)$ is independent of $\zeta |_{\revision{[}t_j, \infty)}$, which follows directly from the construction of $u$ as the limit of $u_\varep$, and the fact that the $\sigma$-algebra $\bigcap_{\varep > 0} \sigma(\zeta|_{(-\infty, t_j + \varep)})$ (where $\sigma(\zeta |_{(-\infty, t_j + \varep)})$ is the $\sigma$-algebra generated by $\zeta |_{(-\infty, t_j+\varep)}$) is independent of $\zeta |_{\revision{[}t_j, \infty)}$. For notational brevity in what follows, for $J \in \dyadic$, define $F_J := \bigcup_{\substack{J' \in \dyadic \\ J' \leq J}} E^{J'}_{\leq J'-1}$. Define also the endpoint case $F_{1/2} := \varnothing$. The estimate \eqref{eq:probability-iteration-estimate} shows that $\p(F_J) \ra 1$ as $J \toinf$. Combining this with the estimate \eqref{eq:iteration-final-bound}, we have that for $p \geq 1$, and some large enough constant $C_2$,
\begin{equs}
\|e^{c_0 t} u\|_{C_t^0 \Cs_x^{\eta}([0, 1])}^p &= \sum_{\substack{J \in \dyadic}} \ind_{F_J \backslash F_{J/2}} \|e^{c_0 t} u\|_{C_t^0 \Cs_x^{\eta}([0, 1])}^p \\
&\leq e^{3c_1 p/2} \|u_0\|_{\Cs_x^{\eta}}^p + \sum_{\substack{J \in \dyadic}} \ind_{F_J \backslash F_{J/2}} \big(C_2 J^{C_1 + 1}\big)^p,
\end{equs}
where in the second inequality, we used that given $c_0$, there is a constant $C$ such that $(x+y)^p \leq e^{c_1 p /2} x^p + C^p y^p$ for all $x, y \geq 0$ and $p \geq 1$ (simply split into cases $x > C' y$ or $x \leq C'y$ for $C'$ sufficiently large depending on $c_1$). Taking expectations on both sides and $q \geq 1$ large enough in the estimate \eqref{eq:probability-iteration-estimate}, we obtain that for some large enough constant $M_p$,
\begin{equs}\label{eq:unit-time-estimate}
\E\Big[\|e^{c_0 t} u\|_{C_t^0 \Cs_x^{\eta}([0, 1])}^p\Big] \leq e^{3c_1 p/2} \E[\|u_0\|_{\Cs_x^{\eta}}^p] + M_p.
\end{equs}
Recalling that $c_0 = m^2, c_1 = m^2/2$, so that $3c_1 /2 < c_0$, the above estimate can easily be iterated to obtain the uniform-in-time estimate in the theorem statement. We omit the details.
\end{proof}




\section{The resonant equation}\label{section:resonant}

We now study the resonant equation \eqref{setup:eq-resonant-equation} and control its solution $\determ_\varep$, i.e., we prove Proposition~\ref{setup:prop-resonant-equation}. The proof of Proposition~\ref{setup:prop-resonant-equation} turns out to be rather elementary, and the main ingredients are the basic inequality $|\sin(\beta \determ)|\leq \min(1,\beta |\determ|)$ and basic estimates of the heat propagator $e^{t\Delta\revision{/2}}$. For notational simplicity, we take $t_0 = 0$ throughout this section.\\

Before turning to the proof of Proposition \ref{setup:prop-resonant-equation}, we
need to introduce additional notation and recall basic estimates of $e^{t\Delta\revision{/2}}$. Throughout this section, we use the variables $t,s\in [0,\infty)$ and $x,y\in \T^2$.
The corresponding space-time coordinates are denoted by $z=(t,x)$ and $z^\prime=(s,y)$, respectively.  In the following definition, we introduce the function spaces and norms which will be used in our analysis of \eqref{setup:eq-resonant-equation}.

\begin{definition}\label{resonant:def-norms}
Let $\eta$ be as in \eqref{eq:eta}, let $T>0$, and let $I=[0,T]$. For any $\determ,F\colon I\times \T^2 \rightarrow \R$, we then define the norms 
\begin{align}
\big\| \determ \big\|_\Snorm = \big\| \determ \big\|_{\Snorm(I)} &:= 
 \| \determ \|_{C_t^0 \Cs_x^{\eta}(I)} + \sup_{\substack{0<s<t\leq T, \\ x,y\in \T^2}} s^{\frac{1-\eta}{2}} \frac{|\determ(z)-\determ(\zprime)|}{\| z - \zprime \|_{\frks}}, \label{resonant:eq-Snorm} \\  
 \big\| F \big\|_{\Nnorm} = \big\| F \big\|_{\Nnorm(I)}&:= \sup_{\substack{0<t\leq T, \\ x\in \T^2}} t^{\frac{1-\eta}{2}} |F(t,x)|. \label{resonant:eq-Nnorm}
\end{align}
The $\Snorm(I)$ and $\Nnorm(I)$-norms will be used to control the solution and nonlinearity of the resonant equation, which motivated our notation. The corresponding function spaces $\Snorm(I)$ and $\Nnorm(I)$ are simply defined as the closure of $C_t^0 \Cs_x^\infty(I)$  with respect to these norms.
\end{definition}

In the next lemma, we recall basic estimates for $e^{t\Delta\revision{/2}}$, which are stated using the norms from Definition~\ref{resonant:def-norms}. This lemma also contains estimates of the Duhamel integral 
\begin{equs}
\Duh \big[ F \big] := \int_0^t \mathrm{d}s\, e^{(t-s)(-m^2+ \Delta\revision{/2})} F(s). 
\end{equs}

\begin{lemma}[Basic estimates] \label{resonant:lem-basic}
For all $\alpha \in \R$, $t>0$, and $u_0 \in \Cs_x^\alpha$, it holds that 
\begin{align}
\| e^{t (-m^2+ \Delta\revision{/2})} u_0 \|_{\Cs_x^\alpha} &\leq e^{-m^2t} \| u_0 \|_{\Cs_x^\alpha}. \label{resonant:eq-heat-contraction}
\end{align}
Furthermore, for all $T\in (0,1]$, $u_0 \in \Cs_x^{\eta}$,  and $F\colon [0,T]\times \T^2 \rightarrow \R$, it holds that  
\begin{equation}\label{resonant:eq-duhamel-integral}
\big\| e^{t(-m^2+ \Delta\revision{/2})} u_0 \big\|_{\Snorm([0,T])} \lesssim  \| u_0 \|_{\Cs_x^{\eta}} \quad \text{and} \quad  \big\| \Duh \big[ F \big] \big\|_{\Snorm([0,T])} \lesssim T^{\frac{1+\eta}{2}} \big\| F \big\|_{\Nnorm([0,T])}. 
\end{equation}
\end{lemma}

\begin{proof}
The first estimate \eqref{resonant:eq-heat-contraction} follows directly from the fact that the kernel of $e^{t\Delta\revision{/2}}$ is $L^1$-normalized and commutes with Littlewood-Paley projections. The second estimate \eqref{resonant:eq-duhamel-integral} is standard and can be obtained using similar arguments as in \cite[Proof of Lemma A.9]{GIP15}.
\end{proof}

\begin{remark} 
Since $\eta$ is negative, the $\eta$-term in the exponent of  $T^{\frac{1+\eta}{2}} $ in \eqref{resonant:eq-duhamel-integral} leads to a loss. This loss is a result of the $C_t^0 \Cs_x^{\eta}$-term in \eqref{resonant:eq-Snorm}, which cannot make optimal use of the time-weighted $L^\infty_x$-norm (rather than time-weighted $\Cs_x^\eta$-norm) in \eqref{resonant:eq-Nnorm}. However, since the absolute value of $\eta$ is rather small, this does not cause any problems.  
\end{remark}

In the next lemma, we obtain estimates of the resonant operator $\Res_\varep$ from Definition \ref{setup:def-resonant-operator}. Once these estimates of $\Res_\varep$ have been obtained, we will be able to prove Proposition \ref{setup:prop-resonant-equation} using a contraction-mapping argument (see the end of this section).

\begin{lemma}[Estimate of resonant operator]\label{resonant:lem-resonant-estimate}
Let $4\pi \leq \beta^2 < 6\pi$ and let $\nu \in \big(\frac{\beta^2}{2\pi}-2, 1\big]$. Furthermore, let 
\begin{equation*}
\gamma := 3- \tfrac{\beta^2}{2\pi}  - (1-\nu) \eta, 
\end{equation*}
let $0<T\leq 1$, and let $\determ\colon [0,T] \times \T^2 \rightarrow \R$. Then, it holds that 
\begin{equation}\label{resonant:eq-resonant-estimate}
\sup_{\varep\in (0,1]} \big\| \Res_\varep ( \determ ) \big\|_{\Nnorm([0,T])} \lesssim T^{\frac{\gamma}{2}} \| \determ \|_{\Snorm([0,T])}^\nu.
\end{equation}
In addition, if $\determ_1,\determ_2 \colon [0,T]\times \T^2 \rightarrow \R$, then it also holds that 
\begin{equation}\label{resonant:eq-resonant-estimate-difference}
\sup_{\varep\in (0,1]} \big\| \Res_\varep(\determ_1) - \Res_\varep(\determ_2) \big\|_{\Nnorm([0,T])} \lesssim T^{\frac{3}{2}-\frac{\beta^2}{4\pi}}  \| \determ_1 - \determ_2 \|_{\Snorm([0,T])}.
\end{equation}
\end{lemma}

\begin{remark}
We remark that, due to our assumptions on $\eta$ and $\nu$, the parameter $\gamma$ is positive. \revision{Also, the assumption $\beta^2 < 6\pi$ crucially allows us to take $\nu \leq 1$. Once $\beta^2 \geq 6\pi$, we would only be able to estimate $\Res(\determ)$ superlinearly in $\determ$ (and moreover we would need a higher derivative of $\determ$).}
\end{remark}

\begin{proof}
From Definition \ref{setup:def-resonant-operator}, \eqref{setup:eq-K-estimate}, and \eqref{setup:eq-J-estimate}, we obtain for all $z=(t,x)\in [0,T] \times \T^2$ that
\begin{equation}\label{resonant:eq-resonant-estimate-p1}
\begin{aligned}
    \big| \Res_\varep(\determ)(z)\big|
&= \bigg| \frac{\beta}{2} \int \dzprime \, K(\zprime;z) \Jc_\varep^{-}(\zprime-z) \sin \big( \beta (\determ(z)-\determ(\zprime)) \big) \bigg| \\ 
&\lesssim \int \dzprime \, \| z-\zprime \|_{\frks}^{-2-\frac{\beta^2}{2\pi}} \Big| \sin \big( \beta (\determ(z)-\determ(\zprime)) \big)\Big|.
\end{aligned}
\end{equation}
We now note that
\begin{equation}\label{resonant:eq-resonant-estimate-p2}
\begin{aligned}
&\, \Big| \sin \big( \beta (\determ(z)-\determ(\zprime)) \big)\Big|
\lesssim \min\big( 1, | \determ(z)-\determ(\zprime)|\big) \\
\lesssim&\,  | \determ(z)-\determ(\zprime)|^\nu \lesssim s^{-\frac{1-\eta}{2} \nu} \| z-\zprime \|_{\frks}^{\nu} \| \determ \|_{\Snorm}^{\nu}.
\end{aligned}
\end{equation}
From \eqref{resonant:eq-resonant-estimate-p2}, it then directly follows that
\begin{equation}\label{resonant:eq-resonant-estimate-p3}
\begin{aligned}
&\, \int \dzprime \, \| z-\zprime \|_{\frks}^{-2-\frac{\beta^2}{2\pi}} \Big| \sin \big( \beta (\determ(z)-\determ(\zprime)) \big)\Big| \\
\lesssim&\, \bigg( \int \dzprime \,  s^{-\frac{1-\eta}{2} \nu}  \| z-\zprime \|_{\frks}^{-2-\frac{\beta^2}{2\pi}+\nu} \bigg) \| \determ \|_{\Snorm}^\nu. 
\end{aligned}
\end{equation}
In order to control the explicit space-time integral in \eqref{resonant:eq-resonant-estimate-p3}, we first use that 
\begin{align}
    \int \dzprime \,  s^{-\frac{1-\eta}{2} \nu}  \| z-\zprime \|_{\frks}^{-2-\frac{\beta^2}{2\pi}+\nu} 
    &= \int_0^t \ds \int_{\T^2} \dy \, s^{-\frac{1-\eta}{2} \nu} \Big( (t-s)^{\frac{1}{2}} + |x-y| \Big)^{-2-\frac{\beta^2}{2\pi}+\nu} \notag \\
    &\lesssim \int_0^t s^{-\frac{1-\eta}{2} \nu} (t-s)^{\frac{1}{2} \big(-\frac{\beta^2}{2\pi}+\nu\big)}. \label{resonant:eq-resonant-estimate-p4}
\end{align}
The inequality in \eqref{resonant:eq-resonant-estimate-p4} is elementary, and can easily be obtained by splitting the $y$-integral into the regions $|x-y|\lesssim (t-s)^{\frac{1}{2}}$ and $|x-y|\gtrsim (t-s)^{\frac{1}{2}}$. Using that both $-\frac{1-\eta}{2}\nu$ and $\frac{1}{2}\big(-\frac{\beta^2}{2\pi}+\nu\big)$ are greater than $-1$, which is due to our assumptions on $\eta$ (recall \eqref{eq:eta}) and $\nu$, we further obtain that 
\begin{equation}
\int_0^t s^{-\frac{1-\eta}{2} \nu} (t-s)^{\frac{1}{2} \big(-\frac{\beta^2}{2\pi}+\nu\big)} \lesssim t^{-\frac{1-\eta}{2}\nu+\frac{1}{2} \big(-\frac{\beta^2}{2\pi}+\nu\big)+1}.
\end{equation}
As a result, it holds that 
\begin{equation}\label{resonant:eq-resonant-estimate-p5}
\begin{aligned}
&\, \sup_{0<t\leq T} \sup_{x\in \T^2} t^{\frac{1-\eta}{2}}  \int \dzprime \,  s^{-\frac{1-\eta}{2} \nu}  \| z-\zprime \|_{\frks}^{-2-\frac{\beta^2}{2\pi}+\nu} \\
\lesssim&\, \sup_{0<t\leq T} t^{\frac{1}{2} \big( (1-\eta)(1-\nu)+\nu-\frac{\beta^2}{2\pi}+2\big)} =  T^{\frac{\gamma}{2}}. 
\end{aligned}
\end{equation}
By combining \eqref{resonant:eq-resonant-estimate-p1}, \eqref{resonant:eq-resonant-estimate-p3}, and \eqref{resonant:eq-resonant-estimate-p5}, we obtain the desired first estimate \eqref{resonant:eq-resonant-estimate}. The proof of \eqref{resonant:eq-resonant-estimate-difference} is similar as the proof of \eqref{resonant:eq-resonant-estimate} with $\nu=1$. The only difference is that, instead of \eqref{resonant:eq-resonant-estimate-p2}, we use the estimate 
\begin{align*}
&\, \Big| \sin \big( \beta (\determ_1(z)-\determ_1(\zprime)) \big) 
- \sin \big( \beta (\determ_2(z)-\determ_2(\zprime)) \big)\Big| \\
\lesssim&\, \Big| \big( \determ_1(z)-\determ_1(\zprime) \big) - \big( \determ_2(z)-\determ_2(\zprime)\big) \Big| 
\lesssim \| z-\zprime\|_{\frks} \big\| \determ_1 - \determ_2 \big\|_{\Snorm}, 
\end{align*}
which then leads to \eqref{resonant:eq-resonant-estimate-difference}. 
\end{proof}

Equipped with Lemma \ref{resonant:lem-resonant-estimate}, we can now prove our estimates of $\determ_\varep$, i.e., our first main proposition. 

\begin{proof}[Proof of Proposition \ref{setup:prop-resonant-equation}:] We first prove the existence of a unique solution of \eqref{setup:eq-resonant-equation} in $\Snorm([0,T])$, which is shown using a standard contraction-mapping argument. To this end, let $C_0\geq 0$ remain to be chosen, let 
\begin{equation}\label{resonant:eq-proof-1}
\mathbb{B}:= \Big\{ \determ \in \Snorm ([0,T])\colon \| \determ \|_{\Snorm([0,T])} \leq C_0 \| u_0 \|_{\Cs_x^{\eta}} \Big\}. 
\end{equation}
and, for any $\determ \in \Snorm([0,T])$, let 
\begin{equation}\label{resonant:eq-proof-2}
\Gamma_\varep \determ := e^{t(-m^2+ \Delta\revision{/2})} \determ + \Duh \big[ \Res_\varep (\determ) \big]. 
\end{equation}
From \eqref{resonant:eq-duhamel-integral} and Lemma \ref{resonant:lem-resonant-estimate} with $\nu=1$, we obtain for all $\determ \in \mathbb{B}$ that
\begin{equation}\label{resonant:eq-proof-3}
\begin{aligned}
\big\| \Gamma_\varep \determ \big\|_{\Snorm([0,T])}
&\lesssim \| u_0 \|_{\Cs_x^{\eta}} + T^{\frac{1+\eta}{2}} T^{\frac{3}{2}-\frac{\beta^2}{4\pi}} \| \determ \|_{\Snorm([0,T])} \\
&\lesssim \| u_0 \|_{\Cs_x^{\eta}} + T^{\frac{1+\eta}{2}} T^{\frac{3}{2}-\frac{\beta^2}{4\pi}} C_0 \| u_0 \|_{\Cs_x^{\eta}}.
\end{aligned}
\end{equation}
Similarly, we also obtain for all  $\determ_1,\determ_2 \in \mathbb{B}$ that 
\begin{equation}\label{resonant:eq-proof-4}
\big\| \Gamma_\varep \determ_1 - \Gamma_\varep \determ_2  \big\|_{\Snorm([0,T])} \lesssim T^{\frac{1+\eta}{2}} T^{\frac{3}{2}-\frac{\beta^2}{4\pi}} \| \determ_1 -\determ_2 \|_{\Snorm([0,T])}.
\end{equation}
We now first choose a sufficiently large $C_0=C_0(\beta,\eta)$, then choose a sufficiently small $T_0=T_0(\beta,\eta,C_0)$, and finally require that $T\leq T_0$. Our estimates \eqref{resonant:eq-proof-3} and \eqref{resonant:eq-proof-4} then show that $\Gamma$ is a contraction on $\mathbb{B}$, which implies that the resonant equation \eqref{setup:eq-resonant-equation} has a unique solution $\determ_\varep$ in $\Snorm([0,T])$. 

It now only remains to prove \eqref{setup:eq-resonant-equation-estimate}, and we therefore now let $\determ_\varep$ be the solution of \eqref{setup:eq-resonant-equation}. Since the desired estimate \eqref{setup:eq-resonant-equation-estimate} only gets worse as $\delta$ increases, it suffices to treat the case in which $\delta$ is sufficiently small depending on $\beta$ and $\eta$. Using \revision{the} definition of the $\Snorm$-norm from \eqref{resonant:eq-Snorm} and \eqref{resonant:eq-heat-contraction}, we first obtain that
\begin{equation}\label{resonant:eq-proof-5}
\begin{aligned}
\| \determ_\varep(T) \|_{ \Cs_x^{\eta}} 
&\leq \big\| e^{T (-m^2+\Delta\revision{/2})} u_0 \big\|_{\Cs_x^{\eta}}
+ \big\| \Duh \big[ \Res_\varep ( \determ_\varep ) \big] \big\|_{C_t^0 \Cs_x^{\eta}([0,T])}\\
&\leq e^{-m^2 T} \| u_0 \|_{\Cs_x^{\eta}} + \big\| \Duh \big[ \Res_\varep ( \determ_\varep ) \big] \big\|_{\Snorm([0,T])}.   
\end{aligned}
\end{equation}
We emphasize that the pre-factor of the $u_0$-term in \eqref{resonant:eq-proof-5} is the same as in \eqref{setup:eq-resonant-equation-estimate}. Since $\determ_\varep\in \mathbb{B}$, it then follows from  \eqref{resonant:eq-duhamel-integral} and Lemma \ref{resonant:lem-resonant-estimate} with $\nu=\frac{\beta^2}{2\pi}-2+\delta$ that
\begin{equation*}
\big\| \Duh \big[ \Res_\varep ( \determ_\varep ) \big] \big\|_{\Snorm([0,T])} 
\lesssim T^{\frac{1+\eta}{2}} T^{\frac{1}{2} \big(  3  - \frac{\beta^2}{2\pi} - (1-\nu) \eta \big)}  \| u_0 \|_{\Cs_x^{\eta}}^{\frac{\beta^2}{2\pi}-2+\delta}. 
\end{equation*}
Since 
\begin{equation*}
\frac{1+\eta}{2} + \frac{1}{2} \Big( 3- \frac{\beta^2}{2\pi} - (1-\nu) \eta \Big) = 2 - \frac{\beta^2}{4\pi} + \frac{\nu\eta}{2} \geq 2 - \frac{\beta^2}{4\pi} +\frac{\eta}{2}, 
\end{equation*}
this then completes the proof.
\end{proof}

\section{Stochastic estimates of monopoles and dipoles}\label{section:objects}

In this section, we prove that the stochastic estimates from Proposition \ref{prop:stochastic-estimates} for the modified model $\Pi^{\determ,\varep}_z$ are satisfied for monopoles and dipoles, i.e., for 
\begin{equation*}
\tau  \in \bigg\{ \, \monop, \, \monom, \, \, \dippp, \, \dippm,\, \dipmp,\, \dipmm \, \bigg\}.
\end{equation*}
We emphasize again that the modified model $\Pi^{\determ,\varep}_z$ depends on the deterministic function $\determ$, and that it is important that the stochastic estimates are uniform in $\determ$. In the case $\determ=0$, the stochastic estimates for monopoles and dipoles were previously obtained in \cite{HS16}. In the general case, the basic idea is to combine the arguments in \cite{HS16} with the estimate $|e^{\pm \icomplex \beta \determ}|\leq 1$. This basic idea, however, is much easier to state than implement, since the identities and estimates for the moments in \cite{HS16} are rather involved.

In Appendix \ref{section:appendix-moments} below, we present the complete proof of Proposition \ref{prop:stochastic-estimates}, which relies on the arguments in \cite{CHS2018} rather than \cite{HS16}. In particular, the proof in Appendix \ref{section:appendix-moments} also yields the estimates from this section (Lemma \ref{objects:lem-monopole} and Lemma \ref{objects:lem-dipole}). This  section is therefore not strictly necessary for the proof of our main theorem, but it still serves an expository purpose. The reader should find it  easier to follow the necessary modifications to the arguments in \cite{CHS2018}, which will be presented in Appendix \ref{section:appendix-moments}, after seeing similar modifications of the arguments in \cite{HS16}.

\begin{remark}[Notation for test functions]
In this section, our notation for test-functions will slightly differ from that of Proposition \ref{prop:stochastic-estimates}. In particular, we will use $\varphi$ instead of $\psi$ to denote a test-function. This is done to stay close to the notation of \cite{CH16}, whose proof we are closely following. 
\end{remark}

\begin{lemma}[Stochastic estimates of monopoles]\label{objects:lem-monopole}
Let $0<\beta^2<8\pi$ and let $\determ \colon \R \times \T^2\rightarrow \R$. 
Furthermore, let $p\geq 1$, let $\varphi$ be a test-function which is supported on the unit ball in $\R \times \T^2$ with $L^\infty$ norm bounded by $1$, let $\lambda \in (0,1]$ be a scale, and let $z\in  \R \times \T^2$ be a space-time point. Finally, let $\tau \in \{ \raisebox{-0.2ex}{\monop}, \raisebox{-0.2ex}{\monom} \}$. Then, it holds that 
\begin{equation}\label{objects:eq-first-order}
\E \Big[ \big| \big \langle \Pi^{\determ,\varep}_z \tau  , \varphi_z^\lambda \big\rangle \big|^{2p} \Big]^{\frac{1}{2p}} 
\lesssim_p \lambda^{-\frac{\beta^2}{4\pi}}.
\end{equation}
\end{lemma}

\begin{proof} 
We obtain from the definition of the model that $\Pi^{\determ,\varep}_z \raisebox{-0.3ex}{\monop}=e^{\icomplex \beta \determ} \xi_+^\varep$ and $\Pi^{\determ,\varep}_z \raisebox{-0.3ex}{\monom}=e^{-\icomplex \beta \determ}\xi_-^\varep$, see also \eqref{appendix:step-model-monopole} below. Due to this, $\Pi^{\determ,\varep}_z \raisebox{-0.3ex}{\monom}$ is the complex-conjugate of $\Pi^{\determ,\varep}_z \raisebox{-0.3ex}{\monop}$, and it therefore suffices to bound $\Pi^{\determ,\varep}_z \raisebox{-0.3ex}{\monop}$. 
 Due to H\"{o}lder's inequality, it further suffices to only treat the case $p=N$, where $N$ is a positive integer. 

The desired estimate \eqref{objects:eq-first-order} follows from minor modifications of the proof of \cite[Theorem~3.2]{HS16}, which corresponds to $\determ=0$, and we only describe the necessary modifications. To this end, we first recall that $\Jc_\varep$ is the covariance-function of $\xi_{\pm}^\varep$ and $\xi_{\pm}^\varep$, while $\Jc_\varep^-$ is the covariance-function of $\xi_{\pm}^\varep$ and $\xi_{\mp}^\varep$, see \eqref{eq:Jc-covariance}. Arguing as on  \cite[p. 943]{HS16}, we then obtain the moment formula 
\begin{equation}\label{objects:eq-first-p2} 
\begin{aligned}
&\hspace{-4ex}\E \Big[ \big| \big \langle e^{\icomplex \beta \determ} \xi_+^\varep, \varphi_z^\lambda \big\rangle \big|^{2N} \Big] \\ 
= \int \Big(&
\prod_{i} \big( e^{\icomplex\beta \determ(z_i)} \varphi_z^\lambda(z_i) \big) 
\, \prod_j \big( e^{-\icomplex \beta \determ(y_j)} \varphi_z^\lambda(y_j)\big) \\
& \hspace{-6ex}\times 
\prod_{l<m} \Jc_\varep(z_l-z_m) \prod_{n<o} \Jc_\varep(y_n-y_o) \prod_{i,j} \Jc^{-}_\varep(z_i-y_j) \Big) \dz_N\hdots \dz_1 \dy_N \hdots \dy_1.
\end{aligned}
\end{equation}
In \eqref{objects:eq-first-p2}, the indices, $i,j,l,m,n$, and $o$ are all implicitly restricted to the set $\{1,\hdots,N\}$. 
Using the boundedness of the complex exponential and the support properties of the test-function $\varphi$, we now estimate
\begin{equation}\label{objects:eq-first-p3} 
\big| e^{\icomplex\beta \determ(z_i)}\big| \leq 1,~ 
\big| e^{-\icomplex \beta \determ(y_j)}  \big| \leq 1, ~ 
\big| \varphi_z^\lambda(z_i) \big| \leq \ind_\Lambda(z_i-z), ~ \text{and} ~~ 
\big| \varphi_z^\lambda(y_j) \big| \leq \ind_\Lambda(y_j-z),
\end{equation}
where $\Lambda$ is the parabolic ball of radius $\lambda$. By first inserting the estimates from \eqref{objects:eq-first-p3} into \eqref{objects:eq-first-p2} and then using a change of variables to reduce to the case $z=0$, we obtain that 
\begin{equation}\label{objects:eq-first-p4}
\begin{aligned}
&\,\E \Big[ \big| \big \langle e^{\icomplex \beta \determ} \xi_+^\varep, \varphi_z^\lambda \big\rangle \big|^{2N} \Big] \\
\leq&\,  \int_{\Lambda^{2N}}
\prod_{l<m} \Jc_\varep(z_l-z_m) \prod_{n<o} \Jc_\varep(y_n-y_0) \prod_{i,j} \Jc^{-}_\varep(z_i-y_j) \dz_N\hdots \dz_1 \dy_N \hdots \dy_1.
\end{aligned}
\end{equation}
We emphasize that, in contrast to the left-hand side, the right-hand side of \eqref{objects:eq-first-p4} does not depend on $\determ$. The right-hand side of \eqref{objects:eq-first-p4} is exactly the same expression as in \cite[(3.8)]{HS16} and, as a result of the argument in \cite{HS16}, we have that
\begin{equation}\label{objects:eq-first-p5}
\begin{aligned}
&\, \int_{\Lambda^{2N}}
\prod_{l<m} \Jc_\varep(z_l-z_m) \prod_{n<o} \Jc_\varep(y_n-y_0) \prod_{i,j} \Jc^{-}_\varep(z_i-y_j) \dz_N\hdots \dz_1 \dy_N \hdots \dy_1 \\
\lesssim&\, \lambda^{-\frac{\beta^2}{2\pi} N} = \big( \lambda^{-\frac{\beta^2}{4\pi}} \big)^{2N}. 
\end{aligned}
\end{equation}
By combining \eqref{objects:eq-first-p4} and \eqref{objects:eq-first-p5}, we then obtain the desired estimate \eqref{objects:eq-first-order}.
\end{proof}

\begin{remark}
The estimates in \cite[Theorem 3.2]{HS16} control the moments of $\langle \xi_+^\varep,\varphi_z^\lambda \rangle$ and the only information about $\varphi_z^\lambda$ used in their proofs is the bound $|\varphi_z^\lambda|\leq \ind_\Lambda$.   Since the bound is also satisfied by $e^{\icomplex \beta \determ} \varphi_z^\lambda$ and  
\begin{equation*}
\big \langle e^{\icomplex \beta \determ} \xi_+^\varep , \varphi^\lambda_z \big \rangle = \big \langle  \xi_+^\varep , e^{\icomplex \beta \determ} \varphi^\lambda_z \big \rangle,
\end{equation*}
it is then clear that the proof of \cite[Theorem 3.2]{HS16} can be used to obtain \eqref{objects:eq-first-order}. However, this trick of pairing up $e^{\pm \icomplex \beta \determ}$-factors with the test-function $\varphi_z^\lambda$ only works for the monopoles, and it breaks down for the dipoles and tripoles. 
\end{remark}

\begin{lemma}[Stochastic estimates of dipoles]\label{objects:lem-dipole}
Let $4\pi \leq \beta^2 < 6\pi$ and let $\determ \colon \R \times \T^2\rightarrow \R$. Furthermore, let $p\geq 1$, let $\varphi$ be a test-function which is supported on the unit ball in $\R \times \T^2$ with $L^\infty$ norm bounded by $1$, let $0<\lambda\leq 1$ be a scale, and let $z\in \R \times \T^2$ be a space-time point. Finally, let 
\begin{equation*}
\tau \in \bigg\{ \, \dippp, \, \dippm,\, \dipmp,\, \dipmm \, \bigg\}.
\end{equation*}
Then, it then holds that 
\begin{equation}\label{objects:eq-second-order}
\begin{aligned}
\E \Big[ \Big| \big \langle \Pi^{\determ,\varepsilon}_z \tau , \varphi^\lambda_z \big \rangle  \Big|^{2p} \Big]^{\frac{1}{2p}}
\lesssim \lambda^{2-\frac{\beta^2}{2\pi}-\kappa}.
\end{aligned}
\end{equation}
\end{lemma}

\begin{proof} 
The proof is a minor modification of the argument in \cite[Section 4]{HS16}. The only difference lies in new  $e^{\pm \icomplex\beta  \determ}$-factors, and the difficulty is to find the location in the argument of \cite[Section 4]{HS16} at which the trivial estimate  $|e^{\pm \icomplex\beta  \determ}|\leq 1$ can be used.

Due to Hölder's inequality, it suffices to prove the estimate for $p=N$, where $N$ is a positive integer.
Since $(\xi_+^\varep,\xi_-^\varep)$ and $(\xi_-^\varep,\xi_+^\varep)$ have the same law, it suffices to prove the estimates for $\makeinline{\dippp}$ and $\makeinline{\dippm}$. We only prove the estimate for $\makeinline{\dippm}$, because the estimate for $\makeinline{\dippp}$ is easier (see \cite[Section 4.4]{HS16}). To match the notation in \cite[Section 4]{HS16}, rather than the notation in \cite{CHS2018}, we recall that $\psi_\varep^+$ and $\psi_\varep^-$ are defined as 
\begin{equation*}
\psi_\varep^\pm \overset{\textup{def}}{=} \biglcol\,  e^{\pm \icomplex\beta  \Phi_\varep} \, \bigrcol \hspace{-0.4ex}= e^{\pm \icomplex\beta  \Phi_\varep} e^{\frac{\beta^2}{2} \mathcal{Q}_\varep(0)}.
\end{equation*}
We also recall that, by definition, 
\begin{equation}\label{objects:eq-second-1}
\begin{aligned}
\Pi^{\determ,\varep}_z \bigg[ \,  \dippm \, \bigg](\widebar{z})
&= e^{\icomplex\beta  \determ(\zbar)} \psi_\varep^+(\zbar) \Big( \big( K \ast (e^{-\icomplex\beta \determ} \psi_\varep^- ) \big) (\zbar) 
-  \big( K \ast (e^{-\icomplex\beta \determ} \psi_\varep^- ) \big) (z) \Big) \\
&- \E \Big[ e^{\icomplex\beta  \determ(\zbar)} \psi_\varep^+(\zbar) \big( K \ast (e^{-\icomplex\beta \determ} \psi_\varep^- ) \big) (\zbar)  \Big].
\end{aligned}
\end{equation}
Similar as in \cite[(4.4)]{HS16}, it is convenient to rewrite \eqref{objects:eq-second-1} as 
\begin{align}
&\, \int \dw \, e^{\icomplex\beta  (\determ(\zbar)-\determ(w))} \big( K(\zbar-w) - K(z-w) \big) \Big( \psi_\varep^+(\zbar) \psi_\varep^-(w) - \E \big[ \psi_\varep^+(\zbar) \psi_\varep^-(w) \big] \Big) \label{objects:eq-second-2} \\
-&\, \int \dw  \, e^{\icomplex\beta  (\determ(\zbar)-\determ(w))} K(z-w) \E \Big[ \psi_\varep^+(\zbar)  \psi_\varep^-(w)  \Big]. 
\label{objects:eq-second-3}
\end{align}
The second term \eqref{objects:eq-second-3} can easily be estimated. Indeed, similarly as in \cite[Lemma 4.1]{HS16}, it holds that 
\begin{align*}
\big| \eqref{objects:eq-second-3} \big| &\leq \int \dw \, \big| e^{\icomplex\beta  (\determ(\zbar)-\determ(w))}\big| K(z-w) \Jc_\varep^{-}(\zbar-w) \\
&\lesssim \int \dw \, \| z-w\|_\frks^{-2} \| \zbar -w \|_{\frks}^{-\frac{\beta^2}{2\pi}}
\lesssim \| z - \zbar \|_{\frks}^{2-\frac{\beta^2}{2\pi}}.
\end{align*}
As a result, it holds that 
\begin{align*}
\Big| \big\langle \eqref{objects:eq-second-3}, \varphi_z^\lambda \big\rangle \Big| \lesssim \int \dzbar\,  \| z - \zbar \|_{\frks}^{2-\frac{\beta^2}{2\pi}} |\varphi_z^\lambda(\zbar)| \lesssim \lambda^{2-\frac{\beta^2}{2\pi}},  
\end{align*}
which yields an acceptable contribution to \eqref{objects:eq-second-order}. It therefore remains to control the first term \eqref{objects:eq-second-2}, i.e., it remains to show that 
\begin{equation}\label{objects:eq-second-4}
\begin{aligned}
&\E \bigg[ \bigg| \int \dzbar \dw \, e^{\icomplex\beta  (\determ(\zbar)-\determ(w))} \varphi_z^\lambda (\zbar)\big( K(\zbar-w)-K(z-w) \big) \\
&\hspace{3ex}\times \Big( \psi_\varep^+(\zbar) \psi_\varep^-(w) - \E \big[ \psi_\varep^+(\zbar) \psi_\varep^-(w) \big] \Big) \bigg|^{2N} \bigg] 
\lesssim \lambda^{2N \big( 2 - \frac{\beta^2}{2\pi}-\kappa \big)}.
\end{aligned}
\end{equation}
In \cite[Section 4]{HS16}, the authors use translation-invariance to reduce to the case $z=0$. Due to the $e^{\pm \icomplex \beta \determ}$-factors, however, the left-hand side of \eqref{objects:eq-second-4} is not translation-invariant. In order to still  match the notation in  \cite[Section 4]{HS16} as closely as possible, we introduce $\determ_z(\cdot):=\determ(\cdot + z)$ and, using a change of variables, rewrite the left-hand side of \eqref{objects:eq-second-4} as
\begin{equation}\label{objects:eq-second-5}
\begin{aligned}
&\E \bigg[ \bigg| \int \dzbar \dw \, e^{\icomplex\beta  (\determ_z(\zbar)-\determ_z(w))} \varphi_0^\lambda(\zbar) \big( K(\zbar-w)-K(-w) \big) \\
&\hspace{3ex}\times\Big( \psi_\varep^+(\zbar) \psi_\varep^-(w) - \E \big[ \psi_\varep^+(\zbar) \psi_\varep^-(w) \big] \Big) \bigg|^{2N} \bigg]. 
\end{aligned}
\end{equation}
In the following, we use the  notation introduced on \cite[pp. 955-956]{HS16}. We also introduce the short-hand $\determ_z(e) := \determ_z(e_\downarrow)-\determ_z(e_\uparrow)$, which is similar to the short-hand $K(e)$ introduced on \cite[p. 956]{HS16}. Using the same argument as in the proofs of  \cite[(4.7) and (4.10)]{HS16}, we then obtain that 
\begin{equation}\label{objects:eq-second-6}
\eqref{objects:eq-second-5} = 
\int_{(\R^3)^{4N}} \dx \, 
\prod_{e\in \Rc} \Big(e^{\icomplex \beta \determ_z(e)} \varphi_0^\lambda(e_\downarrow) K(e) \Big) \mathcal{H}(x,\Jc_\varep).  
\end{equation}
We emphasize that the only difference between \eqref{objects:eq-second-6} and \cite[(4.10) for $k=1$]{HS16} lies in the $e^{\icomplex \beta \determ_z}$-factors. 
Using the triangle-inequality and the trivial estimate $|e^{\icomplex \beta \determ_z(e)}|\leq 1$, we now obtain that 
\begin{equation}\label{objects:eq-second-7}
\Big| \eqref{objects:eq-second-6} \Big| \leq
\int_{(\R^3)^{4N}} \dx \, 
\prod_{e\in \Rc} \Big(|\varphi_0^\lambda(e_\downarrow)| K(e) \Big) \big|\mathcal{H}(x,\Jc_\varep)\big|.  
\end{equation}
After using \cite[Proposition 4.16]{HS16} to  estimate\footnote{While \cite[Proposition 4.16]{HS16} is only stated as an estimate of $\mathcal{H}(x,\Jc)$, the same estimate holds for $|\mathcal{H}(x,\Jc)|$, which can be seen from the proof of \cite[Proposition 4.16]{HS16}. This is also how \cite[Proposition 4.16]{HS16} is used in \cite{HS16}, see e.g. the second display of \cite[Proof of Theorem 4.3]{HS16}.} $|\mathcal{H}(x,\Jc_\varep)|$, we then arrive at the same expression as in the second display of \cite[Proof of Theorem 4.3]{HS16}. The remaining part of the proof of \eqref{objects:eq-second-order} is therefore exactly as in \cite[Proof of Theorem 4.3]{HS16}.
\end{proof}

\section{Remainder bound}\label{section:remainder-bound}

In this section, we prove Proposition \ref{prop:remainder}. Let $\mc{R}_\varep^{\determ}$ denote the reconstruction operator defined in terms of $\mc{Z}(\premodel^{\determ, \varep})$. We first show that the equation for $w_\varep$ from \eqref{eq:w-equation} can be lifted to a fixed point problem on the space of modelled distributions, and moreover the equation in this abstract space has {\it no dependence on the initial data $u_0$}, due to our definition of the modified model. In the following, recall that $\mu = \frac{\beta^2}{4\pi} + 2\kappa$, where $\kappa$ is as in \eqref{eq:kappa}. \revision{In the following, the space $\mc{D}^{\mu, 0}_0$ of singular modelled distributions is as defined in \cite[Definition 6.2]{H14}, and $\mathcal{P}$ is the integration operator corresponding to convolution with the kernel $K$ (see \cite[Section 5]{H14})}


\begin{proposition}\label{prop:reconstruction}
Let $\varep \in (0, 1]$ and let $\determ_\varep$ be the solution to the resonant equation \eqref{setup:eq-resonant-equation}. For some $\tau \in (0, 1]$, there exists a unique solution $W \in \mc{D}_0^{\mu, 0}$ to the following fixed point problem on $[0, \tau] \times \T^2$:
\begin{equs}\label{eq:W-fixed-pt-model-distribution}
W = \mc{P} \ind_{t > 0} \Big( \frac{1}{2\icomplex} \Big(e^{\icomplex \beta G v_\varep(0)} e^{\icomplex \beta W} \Xi_+ - e^{-\icomplex \beta G v_\varep(0)} e^{-\icomplex \beta W} \Xi_-\Big) + R_\varep \Big).
\end{equs}
Moreover, the time of existence $\tau$ depends inverse polynomially on the size of the model $\vertiii{\mc{Z}(\premodel^{\determ_\varep, \varep})}_{\mu; [-1, 2] \times \T^2}$ and $v_\varep(0)$. Finally, $w_\varep = \mc{R}_\varep^{\determ_\varep} W$, i.e., the reconstruction of $W$ is the solution to \eqref{eq:w-equation}.
\end{proposition}

The local existence statement was proven in the proof of \cite[Theorem 1.1]{CHS2018}. Thus, the main thing to prove is that the reconstruction of $W$ is the solution to \eqref{eq:w-equation}. We will provide the proof shortly, after some preliminary discussion. Recalling Remarks \ref{remark:blackbox-0} and \ref{remark:blackbox}, our proof is by explicit calculation, but we would not be surprised if there is a general result that applies, e.g. a suitable modification of the results of \cite{CHS2018} to the case of non-stationary noise. We assume that $\frac{16\pi}{3} \leq \beta^2 < 6\pi$, because the case $4\pi \leq \beta^2 < \frac{16\pi}{3}$ is simpler, as tripoles do not appear.
In this case, the basis elements of $\mc{T}$ with negative or zero homogeneity are:
\begin{equs}
\mbf{1}, \Xi_{s_1}, X_j \Xi_{s_1}, \Xi_{s_1} \mc{I}\Xi_{s_2}, \Xi_{s_1} \mc{I} \Xi_{s_2} \mc{I} \Xi_{s_3}, \Xi_{s_1} \mc{I}(\Xi_{s_2} \mc{I} \Xi_{s_3}), ~~ s_1, s_2, s_3 \in \{+, -\}. 
\end{equs}
Next, we list the basis elements of $\mc{T}$ which (1) have homogeneity in $[0, \mu)$ and (2) lie either in the image of $\mc{I}$, or in the polynomial structure:
\begin{equs}\label{eq:W-span}
\mbf{1}, X_1, X_2, \mc{I} \Xi_{s_1}, \mc{I}(\Xi_{s_1} \mc{I} \Xi_{s_2}), ~~ s_1, s_2 \in \{+, -\}.
\end{equs}
Define $F$ by
\[ F(W) := \frac{1}{2\icomplex} \Big(e^{\icomplex \beta G v_\varep(0)} e^{\icomplex \beta W} \Xi_+ - e^{-\icomplex \beta G v_\varep(0)} e^{-\icomplex \beta W} \Xi_-\Big),\]
so that the fixed point problem \eqref{eq:W-fixed-pt-model-distribution} may be written
\begin{equs}\label{eq:W-abstract-equation}
W = \mc{P} \ind_{t > 0} \big( F(W) + R_\varep\big).
\end{equs}
By the discussion in the proof of \cite[Theorem 1.1]{CHS2018}, we have that $F \colon \mc{D}_0^{\mu, 0} \ra \mc{D}_{-\bar{\beta}}^{\mu - \bar{\beta}, 2 \eta - \bar{\beta}}$. 

Finally, before we get to the proof of Proposition \ref{prop:reconstruction}, we will need some knowledge of $\model^{\determ_\varep, \varep}_z$, which follows from the calculations of Appendix \ref{appendix:modified-model-calculations}. In particular, by Corollary \ref{cor:reconstruction-calculation} and the definition \eqref{setup:def-resonant-operator} of $\Res_\varep$, 
\begin{equation}\label{eq:dipole-resonance}
\begin{aligned}
\model^{\determ_\varep, \varep}_z (\Xi_- \mc{I} \Xi_+)(z) - \model^{\determ_\varep, \varep}_z (\Xi_+ \mc{I} \Xi_-)(z) 
&= -\frac{4\icomplex}{\beta} \Res_\varep(\determ_\varep)(z).
\end{aligned}
\end{equation}
Moreover, by the same corollary, we have that
\begin{equation}\label{eq:reconstruction-zero}
\begin{aligned}
\model_z^{\determ_\varep, \varep}(\Xi_- \mc{I} \Xi_-)(z) &= \model_z^{\determ_\varep, \varep}(\Xi_+ \mc{I} \Xi_+)(z) = 0, \\
\big(\model_z^{\determ_\varep, \varep} \tau\big)(z) &= 0 \text{ for any tripole $\tau$.}
\end{aligned}
\end{equation}

\begin{proof}[Proof of Proposition \ref{prop:reconstruction}]
From the definition of $F$, the solution $W$ is spanned by the basis elements in \eqref{eq:W-span}. For brevity, write $\mc{R} = \mc{R}^{\determ_\varep}_\varep$. Applying $\mc{R}$ to both sides of \eqref{eq:W-abstract-equation}, we get that
\[ \mc{R} W = G\mc{R} F(W) + G R_\varep, \]
and so it remains to verify that 
\begin{equs}
\mc{R} F(W) = \frac{1}{2\icomplex} \Big( e^{\icomplex \beta G v_\varep(0)} e^{\icomplex \beta \mc{R} W} \xi_+^{\determ_\varep, \varep} - e^{-\icomplex \beta G v_\varep(0)} e^{-\icomplex \beta \mc{R} W} \xi_-^{\determ_\varep, \varep}\Big) - \Res_\varep(\determ_\varep).
\end{equs}
Let $f_0 = \langle W, \mbf{1} \rangle$ and $\tilde{W} = W - f_0 \mbf{1}$. We have that
\begin{equs}
e^{\icomplex \beta W} &= e^{\icomplex \beta f_0} \Big( \mbf{1} + \icomplex \beta \tilde{W} + \frac{(\icomplex \beta)^2}{2} \tilde{W}^2\Big), \label{eq:exp-W-plus}\\
e^{-\icomplex \beta W} &= e^{-\icomplex \beta f_0} \Big( \mbf{1} -\icomplex \beta \tilde{W} + \frac{(-\icomplex \beta)^2}{2} \tilde{W}^2\Big), \label{eq:exp-W-minus}
\end{equs}
where we used that $\beta^2 < 6\pi$ so that the higher-order terms in the expansion have homogeneity greater than $\mu$ (note that $\tilde{W}$ has minimum homogeneity $2 - \bar{\beta}$). Inserting these expansions into $F(W)$, we have that
\begin{equs}
F(W) &= \frac{1}{2\icomplex} \Big(e^{\icomplex \beta (G v_\varep(0) + f_0)} \Xi_+ - e^{-\icomplex \beta (G v_\varep(0) + f_0)} \Xi_-\Big) \\
&+ \frac{\beta}{2} \tilde{W} \Big(e^{\icomplex \beta (G v_\varep(0) + f_0)} \Xi_+ + e^{-\icomplex \beta (Gv_\varep(0) + f_0)} \Xi_-\Big) \\
&-\frac{\beta^2}{4\icomplex} \tilde{W}^2 \Big(e^{\icomplex
 \beta (Gv_\varep(0) + f_0)} \Xi_+ - e^{-\icomplex \beta (G v_\varep(0) + f_0)} \Xi_-\Big).
\end{equs}
Additionally, since $W \in \mc{D}_0^{\mu, 0}$, and since $W$ solves \eqref{eq:W-abstract-equation}, we have the following expansion for $W$ \revision{(here $f_i = \langle W, X_i \rangle$ for $i = 1, 2$)}: 
\begin{equs}
W = &f_0 \mbf{1} + f_1 X_1 + f_2 X_2 + \frac{1}{2\icomplex} \big(e^{\icomplex \beta (G v_\varep(0) + f_0)} \mc{I} \Xi_+ - e^{-\icomplex \beta (G v_\varep(0) + f_0)}  \mc{I} \Xi_-\big) \\
&+\frac{\beta}{4\icomplex} \Big(e^{2\icomplex \beta (G v_\varep(0) + f_0)} \mc{I}(\Xi_+ \mc{I} \Xi_+) - e^{-2\icomplex  \beta (G v_\varep(0) + f_0)} \mc{I}(\Xi_- \mc{I} \Xi_-)\Big) \\
&+ \frac{\beta}{4\icomplex} \Big(\mc{I}(\Xi_- \mc{I} \Xi_+) - \mc{I}(\Xi_+\mc{I} \Xi_-)\Big).
\end{equs}
Note that $\mc{R} W = f_0$, since $\mc{R} X_j = 0$ and $\mc{R} \mc{I} \tau = 0$ for all $\tau \in \mc{T}^-$. For brevity, let $\tilde{f}_0 = G v_\varep(0) + f_0 = G v_\varep(0) + \mc{R} W$. 
The desired result now follows from the three claims:
\begin{align}
\frac{1}{2\icomplex} \mc{R} \Big(\big(e^{\icomplex\beta \tilde{f}_0} \Xi_+ - e^{-\icomplex \beta \tilde{f}_0} \Xi_-\big)\Big) &= \frac{1}{2\icomplex}\big( e^{\icomplex \beta \tilde{f}_0} \xi_+^{\determ_\varep, \varep} - e^{-\icomplex \beta \tilde{f}_0} \xi_-^{\determ_\varep, \varep} \big), \label{eq:W-reconstruction-intermediate-1} \\
\frac{\beta}{2} \mc{R} \Big( \tilde{W} (e^{\icomplex \beta \tilde{f}_0} \Xi_+ + e^{-\icomplex \beta \tilde{f}_0}\Xi_-) \Big)(z) &=   -\Res_\varep(\determ_\varep)(z) , \label{eq:W-reconstruction-intermediate-2} \\
\mc{R} \Big(\tilde{W}^2 \big(e^{\icomplex \beta \tilde{f}_0} \Xi_+ - e^{-\icomplex \beta \tilde{f}_0} \Xi_-\big)\Big) &= 0.\label{eq:W-reconstruction-intermediate-3}
\end{align}
The first identity \eqref{eq:W-reconstruction-intermediate-1} follows immediately from the definition of $\model^{\determ_\varep, \varep}_z$ and the fact that $(\mc{R} f)(z) = (\model_z^{\determ_\varep, \varep} f)(z)$ (since our model is made of continuous functions). For the second identity \eqref{eq:W-reconstruction-intermediate-2}, we first note that in the product of $\tilde{W}$ with $\Xi_{\pm}$, we only need to look at the $\mc{I} \Xi_{\pm}$ terms in $\tilde{W}$, as the product between $\Xi_{\pm}$ and any polynomial term reconstructs to zero, and the product between $\Xi_{\pm}$ and an integrated dipole (i.e., a term of the form $\mc{I}(\Xi_{s_1} \mc{I} \Xi_{s_2})$) gives a tripole, which also reconstructs to zero by \eqref{eq:reconstruction-zero}. In summary, we thus have that
\begin{equs}
&\,\frac{\beta}{2}  \mc{R} \Big( \tilde{W} (e^{\icomplex \beta \tilde{f}_0} \Xi_+ + e^{-\icomplex \beta \tilde{f}_0}\Xi_-) \Big)(z) \\
=&\, \frac{\beta}{4\icomplex} \model_z^{\determ_\varep, \varep} \Big( (e^{\icomplex \beta \tilde{f}_0} \mc{I} \Xi_+ - e^{-\icomplex \beta \tilde{f}_0} \mc{I} \Xi_-) (e^{\icomplex \beta \tilde{f}_0} \Xi_+ + e^{-\icomplex \beta \tilde{f}_0} \Xi_-)\Big)(z).
\end{equs}
Expanding out the product and using the identities \eqref{eq:dipole-resonance} and \eqref{eq:reconstruction-zero}, we have that the above is further equal to 
\begin{equs}
\frac{\beta}{4\icomplex} \Pi_z^{\determ_\varep, \varep}(\Xi_- \mc{I} \Xi_+ - \Xi_+ \mc{I} \Xi_-)(z) = -\Res_\varep(\determ_\varep)(z).
\end{equs}
Finally, the identity \eqref{eq:W-reconstruction-intermediate-3} follows because anything involving a polynomial reconstructs to zero and all tripoles reconstruct to zero (by \eqref{eq:reconstruction-zero}).
\end{proof}

By combining Proposition~\ref{prop:reconstruction} with the stochastic estimates Proposition~\ref{prop:stochastic-estimates}, we now prove Proposition~\ref{prop:remainder}. 

\begin{proof}[Proof of Proposition \ref{prop:remainder}]
By the stochastic estimates (Proposition \ref{prop:stochastic-estimates}) and arguing similarly as in the proof of \cite[Theorem 10.7]{H14}, we have that a.s., our modified model $\mc{Z}(\premodel^{\determ_\varep, \varep})$ extends uniquely to an admissible model on the entire regularity structure $(A, \mathscr{T}, \mc{G})$, and moreover for any $p \geq 1$,
\begin{equs}
\E\Big[ \vertiii{\mc{Z}(\premodel^{\determ_\varep, \varep})}^p_{[-1, 2] \times \T^2} \Big] \lesssim 1,
\end{equs}
where the implicit constant is uniform in $\varep$ and $u_0$ (which enters into $\determ_\varep$). The desired result now follows by Proposition \ref{prop:reconstruction}, the fact that $w_\varep(0) = 0$, and general fixed point results which give that the first time that $\|w_\varep\|_{\Cs_x^{2-\bar{\beta}}}$ goes above 1 depends inverse-polynomially on the size of the model and the size of $e^{\icomplex \beta G v_\varep(0)}$. To bound the size of the latter, recall that $v_\varep(0) = -\Phi_\varep(0)$, and note that the size of $\Phi_\varep(0)$ (measured in the appropriate norm) has sub-Gaussian tails.
\end{proof}

\begin{appendix}

\section{From modified pre-model to modified model}\label{appendix:modified-model-calculations}

In this section, we explicitly compute the action of $\model^{\determ, \varep}_{x_\ostar}$ (which recall is obtained from our modified pre-model $\premodel^{\determ, \varep}$ defined in Section \ref{section:modified-model}) in several representative cases. The calculations for all other trees are very similar to the three particular cases that we look at. These calculations will be needed in Appendices \ref{section:appendix-moments} and \ref{appendix:more-calculations}.  \revision{As we will see, the modification by $\determ$ does not affect any of the calculations, in that the steps are exactly the same as for the case $\determ = 0$.} As a small point of notation, we switch to denoting the basepoint by $x_{\ostar}$ instead of $z$, because later in Appendix \ref{section:appendix-moments} we will closely follow the paper \cite{CHS2018}, which uses this alternative notation. We assume throughout that $4\pi \leq \beta^2 < 6\pi$, as otherwise there is no need to consider dipoles or tripoles.

\begin{lemma}\label{appendix:lem-premodel-to-model} 
For any $\varep>0$ and $x_\ostar \in (\R \times \T^2)^{\{ \ostar\}}$,  we have that
\begin{align}
\model^{\determ, \varep}_{x_\ostar}\bigg[\,  \dipmp\, \bigg] &=  \xi^{\determ,\varep}_-  \big( K \ast \xi^{\determ,\varep}_+ \big) - \xi^{\determ,\varep}_- \big(K \ast \xi^{\determ,\varep}_+\big)(x_\ostar) - \E \Big[ \xi^{\determ,\varep}_- \big(K \ast \xi^{\determ,\varep}_+\big)\Big], 
\label{appendix:eq-premodel-to-model-dipole} \\
\model^{\determ, \varep}_{x_\ostar}\bigg[\,  \vtripolempp\, \bigg] &= \xi^{\determ,\varep}_- \big( K \ast \xi^{\determ,\varep}_+ - (K \ast \xi^{\determ,\varep}_+ )(x_\ostar) \big)^2 \label{appendix:eq-premodel-to-model-vtripole}  \\
&- 2\E \Big[ \xi^{\determ,\varep}_- \big( K \ast \xi^{\determ,\varep}_+\big) \Big] 
\, \big( K \ast \xi^{\determ,\varep}_+ -  (K \ast \xi^{\determ,\varep}_+ )(x_\ostar) \big), \notag \\
\model^{\determ, \varep}_{x_\ostar}\left[\rule{0em}{7mm}\right.
\raisebox{-1ex}{\ltripolepmp} \left.\rule{0em}{7mm}\right] &= \xi^{\determ, \varep}_+ \bigg(K \ast \bigg( \model^{\determ, \varep}_{x_\ostar}\bigg[\,  \dipmp\, \bigg]\bigg) - K \ast \bigg(\model^{\determ, \varep}_{x_\ostar}\bigg[\,  \dipmp\, \bigg]\bigg)(x_{\ostar})\bigg)  \label{appendix:eq-premodel-to-model-ltripole-begin}\\
&\quad - \E\Big[\xi^{\determ, \varep}_+ \big( K \ast \xi^{\determ, \varep}_- \big) \Big] \Big(K \ast \xi^{\determ, \varep}_+ - (K \ast \xi^{\determ, \varep}_+)(x_{\ostar})\Big) \\
&\quad - \sum_{|k|_{\mathfrak{s}} = 1} (\cdot - x_{\ostar})^k \xi^{\determ, \varep}_+ \bigg(D^k K \ast \model^{\determ, \varep}_{x_\ostar}\bigg[\,  \dipmp\, \bigg]\bigg)(x_{\ostar}). \label{appendix:eq-premodel-to-model-ltripole-end}
\end{align}
\end{lemma}

\begin{proof}
We recall that the pre-model $\premodel^{\theta,\varep}$ was defined in \eqref{eq:neutral-bphz-premodel-begin-def}-\eqref{eq:neutral-bphz-premodel-end-def} above. The model $(\model^{\theta,\varep}_{x_\ostar})$, together with the linear maps $(F^{\theta,\varep}_{x_\ostar})$ on the model space $\mathscr{T}$, are then defined using the relation 
\begin{equation}\label{appendix:eq-F-to-model}
\model^{\theta,\varep}_{x_\ostar} = \premodel^{\theta,\varep} F^{\theta,\varep}_{x_\ostar}
\end{equation}
together with the properties of the linear maps given by 
\begin{align}
F^{\theta,\varep}_{x_\ostar} X^k &= (X-x_\ostar)^k, \qquad F^{\theta,\varep}_{x_\ostar} \Xi_\pm = \Xi_\pm, \label{appendix:eq-F-basic} \\ 
F^{\theta,\varep}_{x_\ostar} ( \tau \sigma) &= F^{\theta,\varep}_{x_\ostar}(\tau) \, F^{\theta,\varep}_{x_\ostar}(\sigma), \label{appendix:eq-F-product} \\ 
F^{\theta,\varep}_{x_\ostar} (  \mc{I}_m \tau ) &= \mc{I}_m F^{\theta,\varep}_{x_\ostar} ( \tau) 
 + \sum_{|k|_\frks < |\mc{I}_m \tau|_\frks } \frac{X^k}{k!} f^{\theta,\varep}_{x_\ostar}( \mc{I}_{m+k} \tau), \label{appendix:eq-F-integral}
\end{align}
where 
\begin{equation}\label{appendix:eq-model-to-f}
f^{\theta,\varep}_{x_\ostar} ( \mc{I}_n \tau ) = - \sum_{|l|_\frks < |\mc{I}_n\tau|_\frks} \frac{(-x_\ostar)^l}{l!} \big( D^{n+l} K \ast \model^{\theta,\varep}_{x_\ostar} \tau \big)(x_\ostar).
\end{equation}
We emphasize that the linear maps $(F^{\theta,\varep}_{x_\ostar})$ enter into the definition of the model $(\model^{\theta,\varep}_{x_\ostar})$ through \eqref{appendix:eq-F-to-model} and that the model $(\model^{\theta,\varep}_{x_\ostar})$ enters into the definition of the linear maps  $(F^{\theta,\varep}_{x_\ostar})$ through \eqref{appendix:eq-model-to-f}. Nevertheless, the model and linear maps can be computed recursively. \\

\emph{Step 1: Computing $\model^{\theta,\varep}_{x_\ostar}\, \raisebox{-0.1em}\monop$ and $\model^{\theta,\varep}_{x_\ostar} \, \raisebox{-0.1em}\monom$.} From the definition of $\premodel^{\theta,\varep}$, \eqref{appendix:eq-F-to-model}, and \eqref{appendix:eq-F-basic}, we have that
\begin{equation}\label{appendix:step-model-monopole}
\model^{\theta,\varep}_{x_\ostar} \, \raisebox{-0.1em} \monop = 
\premodel^{\theta,\varep} \, \raisebox{-0.1em} \monop = 
\xi^{\theta,\varep}_+ 
\qquad \text{and} \qquad
\model^{\theta,\varep}_{x_\ostar} \, \raisebox{-0.1em} \monom =
\premodel^{\theta,\varep} \, \raisebox{-0.1em} \monom = 
\xi^{\theta,\varep}_-. 
\end{equation}

\emph{Step 2: Computing $F^{\theta,\varep}_{x_\ostar} \, \scalebox{0.8}{$\integratedmonop$}$ and $F^{\theta,\varep}_{x_\ostar} \, \scalebox{0.8}{$\integratedmonom$}$.}
In this step, we show that
\begin{align}\label{appendix:step-F-integrated-monopole}
F^{\theta,\varep}_{x_\ostar} \integratedmonop = \integratedmonop - \big( K \ast \xi^{\theta,\varep}_+ \big)(x_\ostar) \, \mathbf{1}
\qquad \text{and} \qquad 
F^{\theta,\varep}_{x_\ostar} \integratedmonom = \integratedmonom - \big( K \ast \xi^{\theta,\varep}_- \big)(x_\ostar) \, \mathbf{1}.
\end{align}
By symmetry, it suffices to show the first identity in \eqref{appendix:step-F-integrated-monopole}. Using \eqref{appendix:eq-F-integral}, we first obtain that 
\begin{equation}\label{appendix:eq-integrated-monopole-e1}
F^{\theta,\varep}_{x_\ostar} \integratedmonop = F^{\theta,\varep}_{x_\ostar} \mc{I} \Xi_+ 
= \mc{I} F^{\theta,\varep}_{x_\ostar} \Xi_+ + \sum_{\substack{|k|_\frks<|\mc{I}\Xi_+|_\frks}} \frac{X^k}{k} f^{\theta,\varep}_{x_\ostar}(\mc{I}_k \Xi_+). 
\end{equation}
From \eqref{appendix:eq-F-basic}, it follows that the first summand in \eqref{appendix:eq-integrated-monopole-e1} agrees with the first summand in \eqref{appendix:step-F-integrated-monopole}. Since $|\mc{I}\Xi_+|_\frks = 2 + |\Xi_+|_\frks \in (0,1)$, it follows that the second summand in \eqref{appendix:eq-integrated-monopole-e1} is given by $f^{\theta,\varep}_{x_\ostar}(\mc{I}\Xi_+)\mathbf{1}$. From \eqref{appendix:eq-model-to-f} and \eqref{appendix:step-model-monopole}, it follows that 
\begin{equation*}
f^{\theta,\varep}_{x_\ostar}(\mc{I}\Xi_+) = - \big( K \ast \model^{\theta,\varep}_{x_\ostar} \Xi_+ \big) (x_\ostar) = - \big( K \ast \xi^{\theta,\varep}_+ \big)(x_\ostar), 
\end{equation*}
which completes the proof of \eqref{appendix:step-F-integrated-monopole}. \\  

\emph{Step 3: Computing $F^{\theta,\varep}_{x_\ostar}\, \scalebox{0.8}{$\dipmp$}$.} 
In this step, we show that 
\begin{equation}\label{appendix:step-F-dipole}
F^{\theta,\varep}_{x_\ostar} \dipmp = \dipmp - \big( K \ast \xi^{\theta,\varep}_+ \big) (x_\ostar)\, \raisebox{-0.1em} \monom. 
\end{equation}
This follows from \eqref{appendix:eq-F-basic}, \eqref{appendix:eq-F-product}, and \eqref{appendix:step-F-integrated-monopole}. Indeed, it holds that 
\begin{align*}
&\, F^{\theta,\varep}_{x_\ostar} \dipmp = F^{\theta,\varep}_{x_\ostar} \bigg( \raisebox{-0.1em}\monom \, \,  \integratedmonop \bigg)
= F^{\theta,\varep}_{x_\ostar} \Big(  \raisebox{-0.1em}\monom \Big) \, \, F^{\theta,\varep}_{x_\ostar} \bigg( \,  \integratedmonop \, \bigg) \\ 
=&\, \raisebox{-0.1em}\monom \, \bigg( \, \integratedmonop - \big( K \ast \xi^{\theta,\varep}_+ \big) (x_\ostar) \mathbf{1} \bigg) 
= \dipmp - \big( K \ast \xi^{\theta,\varep}_+ \big) (x_\ostar)  \,  \raisebox{-0.1em}\monom. 
\end{align*}

\emph{Step 4: Computing $\model^{\theta,\varep}_{x_\ostar}\, \scalebox{0.8}{$\dipmp$}$.}
In this step, we show \eqref{appendix:eq-premodel-to-model-dipole}. Using \eqref{appendix:eq-F-product} and \eqref{appendix:step-F-dipole}, as well as the definition of $\premodel^{\theta,\varep}$, it holds that 
\begin{align*}
 &\, \model^{\theta,\varep}_{x_\ostar}\, \dipmp 
= \premodel^{\theta,\varep} F^{\theta,\varep}_{x_\ostar} \, \dipmp 
= \premodel^{\theta,\varep} \bigg( \,  \dipmp - \big( K \ast \xi^{\theta,\varep}_+ \big) (x_\ostar) \, \raisebox{-0.1em} \monom \bigg) \\ 
=&\,  \xi^{\determ,\varep}_-  \big( K \ast \xi^{\determ,\varep}_+ \big)  - \E \big[ \xi^{\determ,\varep}_- \big(K \ast \xi^{\determ,\varep}_+\big)\big] - \xi^{\determ,\varep}_- \big(K \ast \xi^{\determ,\varep}_+\big)(x_\ostar),
\end{align*}
which agrees with the right-hand side of \eqref{appendix:eq-premodel-to-model-dipole}.\\

\emph{Step 5: Computing $F^{\theta,\varep}_{x_\ostar} \, \raisebox{0.15em}{\scalebox{0.8}{$\vtripolempp$}}$.} In this step, we show that 
\begin{align}\label{appendix:step-F-vtripole}
F^{\theta,\varep}_{x_\ostar} \, \vtripolempp  = \vtripolempp - 2 \big( K \ast \xi^{\theta,\varep}_+ \big)(x_\ostar) \, \dipmp + \big( K \ast \xi^{\theta,\varep}_+ \big)^2(x_\ostar) \, \raisebox{-0.1em}{\monom}. 
\end{align}
This follows from \eqref{appendix:eq-F-basic}, \eqref{appendix:eq-F-product}, and \ref{appendix:step-F-integrated-monopole}. Indeed, it holds that 
\begin{align*}
F^{\theta,\varep}_{x_\ostar} \, \vtripolempp 
= F^{\theta,\varep}_{x_\ostar} \bigg( \, \raisebox{-0.1em}{\monom} \, \integratedmonop \, \integratedmonop \,  \bigg)
= F^{\theta,\varep}_{x_\ostar} \big( \,  \raisebox{-0.1em}{\monom} \,  \big) 
\bigg( F^{\theta,\varep}_{x_\ostar}  \bigg( \, \integratedmonop \,  \bigg) \bigg)^2 
= \, \raisebox{-0.1em}{\monom} \bigg( \, \integratedmonop - \big( K \ast \xi^{\theta,\varep}_+ \big)(x_\ostar) \, \mathbf{1} \bigg)^2. 
\end{align*}
After expanding the product, we arrive at the right-hand side of \eqref{appendix:step-F-vtripole}. \\ 

\emph{Step 6: Computing $\model^{\theta,\varep}_{x_\ostar} \, \raisebox{0.15em}{\scalebox{0.8}{$\vtripolempp$}}$.} In this step, we show \eqref{appendix:eq-premodel-to-model-vtripole}. From \eqref{appendix:step-F-vtripole}, as well as the definition of $\premodel^{\theta,\varep}$, it follows that 
\begin{align*}
\model^{\theta,\varep}_{x_\ostar} \, \vtripolempp 
&= \premodel^{\theta,\varep} F^{\theta,\varep}_{x_\ostar} \, \vtripolempp \\
&= \premodel^{\theta,\varep} \, \vtripolempp 
- 2 \big( K \ast \xi^{\theta,\varep}_+ \big)(x_\ostar) \, \premodel^{\theta,\varep} \, \dipmp
 + \big( K \ast \xi^{\theta,\varep}_+ \big)^2(x_\ostar) 
 \premodel^{\theta,\varep} \raisebox{-0.1em}\monom \\ 
&= \xi^{\theta,\varep}_- \big( K \ast \xi^{\theta,\varep}_+ \big)^2 
- 2 \big(K\ast \xi^{\theta,\varep}_+ \big) \E \big[ \xi^{\theta,\varep}_- \big( K \ast \xi^{\theta,\varep}_+ \big) \big] \\
&- 2 \big( K \ast \xi^{\theta,\varep}_+ \big)(x_\ostar) \Big( \xi^{\theta,\varep}_- \big( K \ast \xi^{\theta,\varep}_+ \big) - \E \big[ \xi^{\theta,\varep}_- \big( K \ast \xi^{\theta,\varep}_+ \big) \big] \Big) \\ 
&+ \big( K \ast \xi^{\theta,\varep}_+ \big)^2(x_\ostar) \xi^{\theta,\varep}_-. 
\end{align*}
After rearranging the summands, we arrive at the right-hand side of \eqref{appendix:eq-premodel-to-model-vtripole}. \\ 

\emph{Step 7: Computing $F^{\theta,\varep}_{x_\ostar}$ for the line tripole.} In this step, we show that 
\begin{equation}\label{appendix:step-F-ltripole}
\begin{aligned}
F^{\theta,\varep}_{x_\ostar} \, \raisebox{-0.35em}\ltripolepmp 
&= \, \raisebox{-0.35em}\ltripolepmp - \big( K \ast \xi^{\theta,\varep}_+ \big)(x_\ostar) \, \dippm \\
&- \sum_{\substack{|k|_\frks + |l|_\frks \leq 1}} \frac{(-x_\ostar)^l}{l!} \bigg( D^{k+l} K \ast \model^{\theta,\varep}_{x_\ostar} \dipmp  \, \bigg) (x_\ostar) \, \frac{X^k}{k!} \, \raisebox{-0.1em}\monop. 
\end{aligned}
\end{equation}
To this end, we first obtain from \eqref{appendix:eq-F-basic} and \eqref{appendix:eq-F-product} that 
\begin{equation*}
F^{\theta,\varep}_{x_\ostar} \, \raisebox{-0.35em}\ltripolepmp 
= F^{\theta,\varep}_{x_\ostar} \Bigg( \raisebox{-0.1em}\monop \, \raisebox{-0.1em}\integrateddipmp \, \Bigg)
= F^{\theta,\varep}_{x_\ostar} \big( \, \raisebox{-0.1em}\monop \, \big)\, 
F^{\theta,\varep}_{x_\ostar} \Bigg(\, \raisebox{-0.1em}\integrateddipmp \, \Bigg)
=  \, \raisebox{-0.1em}\monop \, 
F^{\theta,\varep}_{x_\ostar} \Bigg(\, \raisebox{-0.1em}\integrateddipmp \, \Bigg).
\end{equation*}
In order to prove \eqref{appendix:step-F-ltripole}, it therefore remains to prove that 
\begin{equation}\label{appendix:eq-F-integrateddip}
\begin{aligned}
F^{\theta,\varep}_{x_\ostar}\,  \raisebox{-0.1em}\integrateddipmp 
&= \, \raisebox{-0.1em}\integrateddipmp  - \big( K \ast \xi^{\theta,\varep}_+ \big)(x_\ostar) \, \integratedmonom  \\
&- \sum_{\substack{|k|_\frks + |l|_\frks \leq 1}} \frac{(-x_\ostar)^l}{l!} \bigg( D^{k+l} K \ast \model^{\theta,\varep}_{x_\ostar} \dipmp  \, \bigg) (x_\ostar) \, \frac{X^k}{k!}.
\end{aligned}
\end{equation}
To this end, we first note that
\begin{equation*}
\Bigg| \, \, \raisebox{-0.1em}\integrateddipmp \, \,  \Bigg|_{\frks} = 4 - 2 \bar{\beta} = 4 - 2 \frac{\beta^2}{4\pi} - 2 \kappa \in (1,2). 
\end{equation*}
From \eqref{appendix:eq-F-integral}, it then follows that
\begin{equation}\label{appendix:eq-F-ltripole-p1}
F^{\theta,\varep}_{x_\ostar}\,  \raisebox{-0.1em}\integrateddipmp 
= F^{\theta,\varep}_{x_\ostar}\, \bigg(  \mc{I} \, \dipmp \, \bigg) 
= \mc{I} F^{\theta,\varep}_{x_\ostar} \bigg( \,   \dipmp \, \bigg) 
+ \sum_{|k|_\frks \leq 1} \frac{X^k}{k!} f^{\theta,\varep}_{x_\ostar} \bigg( \mc{I}_k \, \dipmp \, \bigg).
\end{equation}
From \eqref{appendix:step-F-dipole}, it follows that the first term in \eqref{appendix:eq-F-ltripole-p1} is given by 
\begin{equation}\label{appendix:eq-F-ltripole-p2}
\mc{I} F^{\theta,\varep}_{x_\ostar} \bigg( \,   \dipmp \, \bigg)  = \mc{I}  \bigg( \, \dipmp  \, \bigg)- \big( K \ast \xi^{\theta,\varep}_+ \big)(x_\ostar) \, \mc{I} \big( \raisebox{-0.1em}\, \monom \, \big) 
= \, \raisebox{-0.1em}\integrateddipmp - \big( K \ast \xi^{\theta,\varep}_+ \big)(x_\ostar) \, \integratedmonom. 
\end{equation}
From \eqref{appendix:eq-model-to-f}, we obtain for all $|k|_\frks\leq 1$ that 
\begin{equation}\label{appendix:eq-F-ltripole-p3}
f^{\theta,\varep}_{x_\ostar} \bigg( \mc{I}_k \, \dipmp \, \bigg)
= - \sum_{|l|_\frks \leq 1 - |k|_\frks} \frac{(-x_\ostar)^l}{l!} \bigg( D^{k+l} K \ast \model^{\theta,\varep}_{x_\ostar} \dipmp  \, \bigg) (x_\ostar).
\end{equation}
By combining \eqref{appendix:eq-F-ltripole-p1}, \eqref{appendix:eq-F-ltripole-p2}, and \eqref{appendix:eq-F-ltripole-p3}, we then obtain \eqref{appendix:eq-F-integrateddip}. \\

\emph{Step 8: Computing $\model^{\theta,\varep}_{x_\ostar}$ for the line tripole.} In this step, we show the identity in \eqref{appendix:eq-premodel-to-model-ltripole-begin}-\eqref{appendix:eq-premodel-to-model-ltripole-end}. Using \eqref{appendix:eq-F-to-model} and \eqref{appendix:step-F-ltripole}, we have that 
\begin{equation}\label{appendix:eq-model-ltripole-p0}
\begin{aligned}
\model^{\theta,\varep}_{x_\ostar} \, \raisebox{-0.35em}\ltripolepmp 
= \premodel^{\theta,\varep} F^{\theta,\varep}_{x_\ostar} \, \raisebox{-0.35em}\ltripolepmp  
&= \premodel^{\theta,\varep} \, \raisebox{-0.35em}\ltripolepmp - \big( K \ast \xi^{\theta,\varep}_+ \big)(x_\ostar) \, \premodel^{\theta,\varep} \, \dippm \\ 
&- \sum_{\substack{|k|_\frks + |l|_\frks \leq 1}} \hspace{-1ex} \frac{(-x_\ostar)^l}{l!} \bigg( D^{k+l} K \ast \model^{\theta,\varep}_{x_\ostar} \dipmp  \, \bigg) (x_\ostar) \, \premodel^{\theta,\varep} \Big( \frac{X^k}{k!} \, \raisebox{-0.1em}\monop \Big).
\end{aligned}
\end{equation}
From the definition of the pre-model $\premodel^{\theta,\varep}$, it follows that the first term in \eqref{appendix:eq-model-ltripole-p0} is given by 
\begin{align}
&\hspace{2ex} \premodel^{\theta,\varep} \, \raisebox{-0.35em}\ltripolepmp \notag \\
&= \xi_+^{\determ, \varep} K \ast \big(\xi_-^{\determ, \varep} (K \ast \xi_+^{\determ, \varep}) - \E[\xi_-^{\determ, \varep} (K \ast \xi_+^{\determ, \varep})]\big) - \E[\xi_+^{\determ, \varep} (K \ast \xi_-^{\determ, \varep})] (K \ast \xi_+^{\determ, \varep}) \label{appendix:eq-model-ltripole-p1} \\ 
&= \xi_+^{\theta,\varep} \bigg( K \ast \model^{\theta,\varep}_{x_\ostar} \, \dipmp \, \bigg) 
+ \xi^{\theta,\varep}_+ \big( K \ast \xi^{\theta,\varep}_- \big) \big( K \ast  \xi^{\theta,\varep}_+ \big)(x_\ostar)  
- \E[\xi_+^{\determ, \varep} (K \ast \xi_-^{\determ, \varep})] (K \ast \xi_+^{\determ, \varep}). \notag
\end{align}
From the definition of the pre-model $\premodel^{\theta,\varep}$, it also follows that the second term in \eqref{appendix:eq-model-ltripole-p0} is given by
\begin{equation}\label{appendix:eq-model-ltripole-p2}
\begin{aligned}
&\, - \big( K \ast \xi^{\theta,\varep}_+ \big)(x_\ostar) \, \premodel^{\theta,\varep} \, \dippm \\
=&\, - \big( K \ast \xi^{\theta,\varep}_+ \big)(x_\ostar) \, 
\xi^{\theta,\varep}_+ \big( K \ast \xi^{\theta,\varep}_- \big)
+  \big( K \ast \xi^{\theta,\varep}_+ \big)(x_\ostar) \, 
\E \big[ \xi^{\theta,\varep}_+ \big( K \ast \xi^{\theta,\varep}_- \big) \big].
\end{aligned}
\end{equation}
Finally, the third term in \eqref{appendix:eq-model-ltripole-p0} is given by 
\begin{equation}\label{appendix:eq-model-ltripole-p3}
\begin{aligned}
&\, - \sum_{\substack{|k|_\frks + |l|_\frks \leq 1}} \frac{(-x_\ostar)^l}{l!} \bigg( D^{k+l} K \ast \model^{\theta,\varep}_{x_\ostar} \dipmp  \, \bigg) (x_\ostar) \, \premodel^{\theta,\varep} \Big( X^k \, \raisebox{-0.1em}\monop \Big) \\ 
=&\, - \sum_{\substack{|k|_\frks + |l|_\frks \leq 1}} \frac{(-x_\ostar)^l}{l!} \bigg( D^{k+l} K \ast \model^{\theta,\varep}_{x_\ostar} \dipmp  \, \bigg) (x_\ostar) \, \frac{(\cdot)^k}{k!} \xi^{\theta,\varep}_+ \\
=&\,  - \sum_{\substack{|k|_\frks  \leq 1}} \frac{(\cdot-x_\ostar)^k}{k!} \bigg( D^{k} K \ast \model^{\theta,\varep}_{x_\ostar} \dipmp  \, \bigg) (x_\ostar) \xi^{\theta,\varep}_+.
\end{aligned}
\end{equation}
It is then a straightforward calculation to check that the sum of \eqref{appendix:eq-model-ltripole-p1}, \eqref{appendix:eq-model-ltripole-p2}, and \eqref{appendix:eq-model-ltripole-p3} agrees with the right-hand side of \eqref{appendix:eq-premodel-to-model-ltripole-begin}-\eqref{appendix:eq-premodel-to-model-ltripole-end}. 
\end{proof}

\section{General stochastic estimates}\label{section:appendix-moments}

The stochastic estimates for the modified dipoles and tripoles follow by small modifications of the arguments in \cite[Sections 3-5]{CHS2018}. Instead of reproducing all the arguments here, we indicate the (very slight) adjustments to the various results and proofs. As in the case of the monopole, the point is that after stripping away all the technical details, the core reason why the stochastic estimates still hold for our modified objects is the elementary inequality $|e^{\icomplex \beta \theta}| \leq 1$ for any $\theta \in \R$. 

We assume $\beta^2 < 6\pi$, as is the case all throughout the paper. This is more restrictive than the assumption $\beta^2 < 8\pi$ in \cite{CHS2018}. This will give simplifications in various places that we will comment on later. 

\begin{remark}Our assumption imposing a smaller range of $\beta$ seems essential, as it appears to us that the stochastic estimates of Proposition \ref{prop:stochastic-estimates} can only be made to be uniform in $\determ$ when $\beta^2 < 6\pi$. This is due to the fact that \cite[Lemma 3.10]{CHS2018} is trivial when $\beta^2 < 6\pi$, however once $\beta$ is beyond this threshold, a parity argument is needed, and unfortunately this argument does not seem to extend so well to the case of our modified model.
\end{remark}

We first indicate the adjustments to \cite[Proposition 3.3]{CHS2018}, which gives formulas for the model and moments of the model. Let all notation be as in \cite{CHS2018}. Fix a dipole or tripole $\bar{T}^{\bar{\mathfrak{n}}\bar{\mathfrak{l}}} \in \mc{T}^-$ and $p \in \N$. Because we are considering a negative-homogeneity tree (and $\beta^2 < 6\pi$), each node of $\bar{T}^{\bar{\mathfrak{n}}\bar{\mathfrak{l}}}$ is labeled with a plus or minus, so that $\bar{\mathfrak{l}} \colon N(T) \ra \{+, -\}$. The first modification to make is in the definition of the operators $(H_{\mc{J}, \mc{F}, S}^\determ, S \in \mc{F})$, which are now $\determ$-dependent. We recursively define
\begin{align}
\relax [H^{\determ}_{\mc{J}, \mc{F}, S} (\varphi)](x) := &\int_{\tilde{N}_{\mc{F}}(S)} dy \prod_{u \in \tilde{N}_{\mc{F}}(S)} e^{\icomplex \nodelabel(u) \beta \determ(y_u)}  \mc{J}^{L_{\mc{F}}(S)^{(2)}}(y \sqcup x_{\varrho_S}) \mrm{Ker}^{\mathring{K}_{\mc{F}}(S)}(y \sqcup x_{\varrho_S}) \notag \\
&\cdot H_{\mc{J}, \mc{F}, C_{\mc{F}}(S)}^\determ\Big[ \mrm{Ker}^{K^\ptl_{\mc{F}}(S)} \cdot (-\mathscr{Y}_S \varphi)\Big](x_{\tilde{N}(S)^c} \sqcup y),
\label{eq:H-determ}
\end{align}
where the base case $H^\determ_{\Jc, \mc{F}, \varnothing}$ is defined to be the identity operator. The only difference between the above definition and \cite[Equation (3.5)]{CHS2018} is the additional $\prod_{u \in \tilde{N}_{\mc{F}}(S)} e^{\icomplex \nodelabel(u) \beta \determ(y_u)}$ term which arises from the definition of our modified monopoles. Next, we define (cf. \cite[Equation~(3.8)]{CHS2018})
\begin{align}
M^{\determ}[\psi, \mc{J}] := &\sum_{\substack{\mc{G} \in \mathbb{F} \\ \mathscr{D} \sse \mathfrak{C} \backslash K(\mc{G})}} \int_{N(\mc{G}, D_{2p})}\hspace{-1ex} dy \prod_{u \in N(\mc{G}, D_{2p})} 
\hspace{-1ex} e^{\icomplex \nodelabel(u) \beta \determ(y_u)} \mc{J}^{L(\mc{G}, D_{2p})^{(2)}}(y) \cdot \mrm{Ker}^{K(\mc{G}, D_{2p}) \backslash \mathscr{D}}(y) \notag \\
& \cdot \bigg(\prod_{j=1}^{2p} \psi(y_{\varrho_{\bar{T}_j}})\bigg) \cdot \mrm{RKer}^{K(\mc{G}, D_{2p}) \cap \mathscr{D}}(z) \cdot X_{\bar{\mathfrak{n}}, \ostar}^{N(\mc{G}, D_{2p})}(z)\label{eq:M-determ} \\
&\cdot H^{\determ}_{\mc{J}, \mc{G}, \bar{\mc{G}}}\Big[ \mrm{RKer}^{K^{\downarrow}(\bar{\mc{G}}) \cap \mathscr{D}} \cdot \mrm{Ker}^{K^{\downarrow}(\bar{\mc{G}}) \backslash \mathscr{D}} X_{\bar{\mathfrak{n}}, \ostar}^{\tilde{N}(\bar{\mc{G}})}\Big](z) \notag
\end{align}
Again, the only difference between the above definition and \cite[Equation (3.8)]{CHS2018} is the presence of the term $\prod_{u \in N(\mc{G}, D_{2p})} e^{\icomplex \nodelabel(u) \beta \determ(y_u)}$ and the use of $H^{\determ}_{\mc{J}, \mc{G}, \bar{\mc{G}}}$ instead of $H_{\mc{J}, \mc{G}, \bar{\mc{G}}}$.

\begin{remark}
In \cite{CHS2018}, general ``$\varep$-assignments\revision{''} are considered, in order to show convergence of the models. In our case, we only need bounds which are uniform in $\varep$, so we will not need to discuss $\varep$-assignments here. Related to this, while \cite{CHS2018} considers a general collection of functions $\Jc$, we will always take $\Jc$ to denote the collection of functions $(\Jc_{\varep, e}, e \in L(D_{2p})^{(2)})$, where $\Jc_{\varep, e} = \Jc_\varep^{\mrm{sign}(e)}$ for all $e \in L(D_{2p})^{(2)}$. This is a slight abuse of notation, since we previously used $\Jc_\varep$ to denote the covariance function \eqref{eq:Jc-covariance}. We allow this abuse of notation because we do not need to consider general $\varep$-assignments, so that the function $\Jc_{\varep, e}$ corresponding to each pair $e$ is the same (up to taking reciprocals).
\end{remark}
 
The proof of the first identity in the following proposition is deferred to Appendix \ref{appendix:more-calculations}. It involves combining the calculations of Appendix \ref{appendix:modified-model-calculations} with some more calculations. We mention that these additional calculations may help the reader with absorbing some of the notation which has been introduced. The second identity in the proposition follows from the first, in the same way that \cite[(3.10)]{CHS2018} follows from \cite[(3.9)]{CHS2018}. 

\begin{proposition}[Analog of Proposition 3.3 of \cite{CHS2018}]\label{appendix:prop-model-and-moments}
Let $\varep > 0$ and let $(\Pi^{\determ, \varep}, \Gamma^{\determ, \varep})$ be the model defined in Section \ref{section:modified-model}. Then, for $x_{\ostar} \in (\R \times \T^2)^{\{\ostar\}}$, any test function $\psi$, any $j \in [2p]$, we have that (using the shorthand $\bar{T} = \bar{T}_j$)
\begin{align}
\Pi^{\determ, \varep}_{x_\ostar}[\bar{T}^{\bar{\mathfrak{n}} \bar{\mathfrak{l}}}](\psi) = \sum_{\substack{\mc{G} \in \mbb{F}_j \\ \mathscr{D} \sse \mathfrak{C}_j \backslash K(\mc{G})}} &\int_{N(\mc{G}, \bar{T})} dy \xi^{L(\mc{G}, \bar{T}), \varep}(y) \prod_{u \in N(\mc{G}, \bar{T})} e^{\icomplex \beta \nodelabel(u) \determ(y_u)} \cdot \mrm{Ker}^{K(\mc{G}, \bar{T}) \backslash \mathscr{D}}(y) \notag \\
&\cdot \psi(z_{\varrho_{\bar{T}}}) \cdot \mrm{RKer}^{K(\mc{G}, \bar{T}) \cap \mathscr{D}}(z) \cdot X_{\bar{\mathfrak{n}}, \ostar}^{N(\mc{G}, \bar{T})}(z) \label{appendix:eq-model} \\
&\cdot H_{\mc{J}, \mc{G}, \bar{\mc{G}}}^\determ \Big[ \mrm{RKer}^{K^{\downarrow}(\bar{\mc{G}}) \cap \mathscr{D}} \cdot \mrm{Ker}^{K^{\downarrow}(\bar{\mc{G}}) \backslash \mathscr{D}} X_{\bar{\mathfrak{n}}, \ostar}^{\tilde{N}(\bar{\mc{G}})}\Big](z), \notag 
\end{align}
where $z = x_{\ostar} \sqcup y$. Moreover, 
\begin{equs}
\E\bigg[ \bigg| \Pi^{\determ, \varep}_{x_\ostar}[\bar{T}^{\bar{\mathfrak{n}} \bar{\mathfrak{l}}}](\psi)\bigg|^{2p}\bigg] = M^\determ[\psi, \mc{J}].
\end{equs}
\end{proposition}

Next, note that since $\beta^2 < 6\pi$, we have that $\mrm{Div}$ is a set of dipoles, and thus $|S^0|_{\mathfrak{s}} = 2 - \frac{\beta^2}{2\pi} > -1$ for all $S \in \mrm{Div}$. It follows that $\mathscr{Y}_S = \mathscr{Y}_S^{(0)}$. Thus, if we define $\mathring{H}^{\determ}_{\mc{J}, \mc{F}, S}$ exactly as in \eqref{eq:H-determ}, except we use $\mathscr{Y}_S^{(0)}$ instead of $\mathscr{Y}_S$, then we immediately obtain that $\mathring{H}^{\determ}_{\mc{J}, \mc{F}, S} = H^{\determ}_{\mc{J}, \mc{F}, S}$. This is the analog of \cite[Lemma~3.10]{CHS2018}. Similarly, since every $S \in \mrm{Div}$ is a dipole, we have that for any $\mc{F} \in \mathbb{F}$ and any $S \in \mc{F}$,
\begin{equs}
P^\ptl_{\mc{F}}(S) = \varnothing,
\end{equs}
where $P^\ptl_{\mc{F}}(S)$ is as defined in \cite[Equation (3.19)]{CHS2018}. Thus if we define
\begin{align}
\relax [\bar{H}^{\determ}_{\mc{J}, \mc{F}, S} (\varphi)](x) := &\int_{\tilde{N}_{\mc{F}}(S)} dy \prod_{u \in \tilde{N}_{\mc{F}}(S)} e^{\icomplex \nodelabel(u) \beta \determ(y_u)} \mc{J}^{L_{\mc{F}}(S)^{(2)}}(y \sqcup x_{\varrho_S}) \mrm{Ker}^{\mathring{K}_{\mc{F}}(S)}(y \sqcup x_{\varrho_S}) \notag \\
&\cdot \bar{H}_{\mc{J}, \mc{F}, C_{\mc{F}}(S)}^\determ\Big[ \mc{J}^{P^\ptl_\mc{F}(S)} \mrm{Ker}^{K^\ptl_{\mc{F}}(S)} (-\mathscr{Y}_S \varphi)\Big](x_{\tilde{N}(S)^c} \sqcup y),\label{eq:bar-H-determ}
\end{align}
it immediately follows that $\bar{H}^\determ_{\mc{J}, \mc{F}, S} = H^\determ_{\mc{J}, \mc{F}, S}$, which is the analog of \cite[Lemma 3.11]{CHS2018}. We will use $\bar{H}^\determ$ in the following, to mirror \cite{CHS2018} as close as possible. Next, following \cite[(3.24) and (3.25)]{CHS2018}, we define $\bar{M}^\determ[\psi, \mc{J}]$ by 
\begin{equation}
\bar{M}^\determ[\psi, \mc{J}] := \sum_{\substack{\mc{G} \in \mathbb{F} \\ \mathscr{D} \sse \mathfrak{C} \backslash K(\mc{G})}} \int_{N(\mc{G}, D_{2p})} dy \prod_{u \in N(\mc{G}, D_{2p})} e^{\icomplex \beta \nodelabel(u) \determ(y_u)} \mc{W}^{\determ}[\psi, \mc{G}, \mathscr{D}, \mc{J}](x_{\ostar} \sqcup y),
\end{equation}
where
\begin{equation}
\begin{aligned}
\mc{W}^{\determ}[\psi, \mc{G}, \mathscr{D}, \mc{J}](z) := &\mc{J}^{L(\mc{G}, D_{2p})^{(2)}} \cdot \mrm{Ker}^{K(\mc{G}, D_{2p}) \backslash \mathscr{D}}(z) \\
&\cdot \bigg(\prod_{j=1}^{2p} \psi(z_{\varrho_{\bar{T}_j}})\bigg) \cdot \mrm{RKer}^{K(\mc{G}, D_{2p}) \cap \mathscr{D}}(z) \cdot X^{N(\mc{G}, D_{2p})}_{\bar{\mathfrak{n}}, \ostar}(z) \\
&\cdot \bar{H}^{\determ}_{\mc{J}, \mc{G}, \bar{\mc{G}}}\Big[ \mc{J}^{P^\ptl_{\mc{G}}(D_{2p})} \cdot \mrm{RKer}^{K^{\downarrow}(\bar{\mc{G}}) \cap \mathscr{D}}  \mrm{Ker}^{K^{\downarrow}(\bar{\mc{G}}) \backslash \mathscr{D}} X_{\bar{\mathfrak{n}}, \ostar}^{\tilde{N}(\bar{\mc{G}})} \Big](z).
\end{aligned}
\end{equation}

The next proposition is the analog of \cite[Proposition 3.12]{CHS2018}. 

\begin{proposition}[Analog of Proposition 3.12 of \cite{CHS2018}]
We have that
\begin{equs}
\bar{M}^\determ[\psi, \mc{J}] = M^\determ[\psi, \mc{J}].
\end{equs}
\end{proposition}

The proof is exactly the same as for \cite[Proposition 3.12]{CHS2018} (indeed, for us it is even simpler, since in our case we never have nested trees), in that if we let $\bar{\mc{G}} = \{T_1, \ldots, T_n\}$, then one reduces to checking that $\mathscr{Y}_{T_1} \cdots \mathscr{Y}_{T_n}(\mc{J}^{P^\ptl_\mc{G}(D_{2p})}) = 1$. 

\begin{proof}
It suffices to show that for fixed $\mc{G} \in \mathbb{F}$ and $\mathscr{D} \sse \mathfrak{C} \backslash K(\mc{G})$, we have that for any $\varphi$,
\begin{equs}
\bar{H}^{\determ}_{\mc{J}, \mc{G}, \bar{\mc{G}}}\Big[ \mc{J}^{P^\ptl_{\mc{G}}(D_{2p})}\varphi\Big](z) = \bar{H}^{\determ}_{\mc{J}, \mc{G}, \bar{\mc{G}}}\big[\varphi\big](z).
\end{equs}
Suppose that $\mc{G} = \{T_1, \ldots, T_n\}$. Since the only divergent trees for us are dipoles, we have that $\bar{\mc{G}} = \mc{G}$. Writing out the definitions, we have that
\begin{equs}
\relax [\bar{H}^\determ_{\mc{J}, \mc{G}, \mc{G}} \varphi](x) = &\int_{\tilde{N}(\mc{G})} dy \prod_{u \in \tilde{N}(\mc{G})} e^{\icomplex \beta \nodelabel(u) \determ(y_u)} \prod_{i \in [n]} \mc{J}^{L(T_i)^{(2)}}(x_{\varrho_{T_i}} \sqcup y) \mrm{Ker}^{T_i}(x_{\varrho_{T_i}} \sqcup y) \\
&\cdot(-1)^n \big(\mathscr{Y}_{T_1} \circ \cdots \circ \mathscr{Y}_{T_n}(\mc{J}^{P^\ptl_\mc{G}(D_{2p})} \varphi) \big)(x_{\tilde{N}(\mc{G})^c} \sqcup y).
\end{equs}
By the definition of $\mathscr{Y}_T$, we have that
\begin{align*}
&\, \big(\mathscr{Y}_{T_1} \circ \cdots \circ \mathscr{Y}_{T_n}(\mc{J}^{P^\ptl_\mc{G}(D_{2p})} \varphi)\big)(x_{\tilde{N}(\mc{G})^c} \sqcup y) \\
=&\, (\mc{J}^{P^\ptl_\mc{G}(D_{2p})})(\mrm{Coll}_{\mc{G}}( x_{\tilde{N}(\mc{G})^c} \sqcup y)) \varphi(\mrm{Coll}_{\mc{G}}( x_{\tilde{N}(\mc{G})^c} \sqcup y)).
\end{align*}
By the usual cancellation identities involving $\mc{J}$, we have that 
\begin{equs}
\mc{J}^{P^\ptl_\mc{G}(D_{2p})}(\mrm{Coll}_{\mc{G}}( x_{\tilde{N}(\mc{G})^c} \sqcup y)) = 1,
\end{equs}
and the desired result follows.
\end{proof}

Following the notation of \cite[Section 3.4]{CHS2018}, define
\begin{equs}
\mc{W}_\lambda^\determ[\mc{G}, \mathscr{D}, \mc{J}] := \int_{N(\mc{G}, D_{2p})} dy \prod_{u \in N(\mc{G}, D_{2p})} e^{\icomplex \beta \nodelabel(u) \determ(y_u)} \mc{W}^{\determ}[\psi^\lambda_{x_\ostar}, \mc{G}, \mathscr{D}, \mc{J}](x_{\ostar} \sqcup y),
\end{equs}
so that
\begin{equs}
\bar{M}^\determ_\lambda[\mc{J}] := \bar{M}^\determ[\psi^\lambda_{x_\ostar}, \mc{J}] = \sum_{\substack{\mc{G} \in \mathbb{F} \\ \mathscr{D} \sse \mathfrak{C} \backslash K(\mc{G})}} \mc{W}_\lambda^\determ[\mc{G}, \mathscr{D}, \mc{J}].
\end{equs}

Based on the preceding discussion, Proposition \ref{prop:stochastic-estimates} now follows directly from the estimate (the analog of \cite[Equation (3.28)]{CHS2018}) 
\begin{equs}\label{eq:bar-M-lambda-estimate}
|\bar{M}^\determ_\lambda[\mc{J}]| \lesssim \|\mc{J}\| \lambda^{2p|\bar{T}^{\bar{\mathfrak{n}}\bar{\mathfrak{l}}} |_{\mathfrak{s}}},
\end{equs}
whose proof is the content of the remaining discussion.

\begin{remark}
Technically, \eqref{eq:bar-M-lambda-estimate} is not exactly the same as the claimed estimate of Proposition \ref{prop:stochastic-estimates}, as the latter has an additional $\delta$ factor in the exponent of $\lambda$. However, this may be circumvented by taking $\bar{\beta}$ to be closer to $\frac{\beta^2}{4\pi}$ (which decreases the homogeneity $|\bar{T}^{\bar{\mbf{n}} \nodelabel}|_{\mbf{s}}$), and then proving the estimate \eqref{eq:bar-M-lambda-estimate} with this new value of $\bar{\beta}$ (which enters through the homogeneity $|\bar{T}^{\bar{\mbf{n}} \nodelabel}|_{\mbf{s}}$). This is what is implicitly done in \cite[Section 3-5]{CHS2018}, which we are trying to follow as closely as possible.
\end{remark}

Next, we discuss the necessary adjustments to \cite[Section 4]{CHS2018}. Given a scale assignment \mbox{$\mbf{n} \in \N^{\mc{E}}$}, an interval of forests $\mathbb{M}$, $S \in b(\mathbb{M})$, we recursively define
\begin{equation}
\begin{aligned}
\relax [\bar{H}^{\determ, \mbf{n}}_{\mc{J}, \mathbb{M}, S}(\varphi)](x) := &\int_{\tilde{N}_{b(\mathbb{M})}(S)} dy \prod_{u \in \tilde{N}_{b(\mathbb{M})}(S)} e^{\icomplex \beta \nodelabel(u) \determ(y_u)} \mc{J}_{\mbf{n}}^{L_{b(\mbb{M})}(S)^{(2)}}(x_{\varrho_S} \sqcup y) \\
&\hspace{-15ex}\cdot \mrm{Ker}_{\mbf{n}}^{\mathring{K}_{b(\mbb{M})}(S)}(x_{\varrho_S} \sqcup y) \cdot \bar{H}^{\determ, \mbf{n}}_{\mc{J}, \mbb{M}, C_{b(\mbb{M})}(S)}\Big[ \mc{J}_{\mbf{n}}^{P^\ptl_{b(\mbb{M})}(S)} \mrm{Ker}_{\mbf{n}}^{K^{\ptl}_{b(\mbb{M})}(S)} \cdot (\mathscr{Y}_{S, \mbb{M}}^{\#} \varphi) \Big](x \sqcup y),
\end{aligned}
\end{equation}
where the base case $\bar{H}^{\determ, \mbf{n}}_{\Jc, \mbb{M}, \varnothing}$ is defined to be the identity operator. Again, the only difference between the above and the corresponding definition of $\bar{H}^{\mbf{n}}_{\mc{J}, \mbb{M}, S}$ at the beginning of \cite[Section 4]{CHS2018} is the presence of the term $ \prod_{u \in \tilde{N}_{b(\mathbb{M})}(S)} e^{\icomplex \beta \nodelabel(u) \determ(y_u)}$. Similarly, we define $\mc{W}^{\determ, \mbf{n}}_\lambda[\mc{J}, \mbb{M}, \mbb{G}]$ exactly as in \cite[Equation (4.3)]{CHS2018}, except the integrand has an extra factor $\prod_{u \in N(b(\mbb{M}), D_{2p})} e^{\icomplex \beta \nodelabel(u) \determ(y_u)}$, and we use $\bar{H}^{\determ, \mbf{n}}_{\mc{J}, \mbb{M}, \ovl{b(\mbb{M})}}$ in place of $\bar{H}^{\mbf{n}}_{\mc{J}, \mbb{M}, \ovl{b(\mbb{M})}}$. We next state the analog of \cite[Lemma 4.2]{CHS2018}, which has the exact same proof as the cited lemma.

\begin{lemma}[Analog of Lemma 4.2 of \cite{CHS2018}]
We have that
\begin{equs}
\sum_{\mbf{n} \in \N^{\mc{E}}} \mc{W}^{\determ, \mbf{n}}_\lambda[\mc{J}, \{\mc{F}\}, \{\mathscr{C}\}] = \mc{W}_\lambda^\determ[\mc{J}, \mc{F}, \mathscr{C}].
\end{equs}
For fixed $\mbf{n} \in \N^{\mc{E}}$, interval of forests $\mbb{M}$, and interval of cut sets $\mbb{G}$ with $b(\mbb{G}) \sse \mathfrak{C} \backslash K(b(\mbb{M}))$, we have that
\begin{equs}
\sum_{\substack{\mc{F} \in \mbb{M} \\ \mathscr{C} \in \mbb{G}}} \mc{W}^{\determ, \mbf{n}}_\lambda[\mc{J}, \{\mc{F}\}, \{\mathscr{C}\}] = \mc{W}^{\determ, \mbf{n}}_\lambda[\mc{J}, \mbb{M}, \mbb{G}].
\end{equs}
\end{lemma}
\begin{proof}
The first identity follows by linearity, in exactly the same way as \cite[Equation (4.4)]{CHS2018}. Thus, we focus on the second identity. We first show that for any $\mbb{M}$, we have that
\begin{equs}
\sum_{\mc{F} \in \mbb{M}} \mc{W}^{\determ, \mbf{n}}_\lambda[\mc{J}, \{\mc{F}\}, \varnothing] = \mc{W}^{\determ, \mbf{n}}_\lambda[\mc{J}, \mbb{M}, \varnothing].
\end{equs}
To see this, we induct on the size of $\delta(\mbb{M})$. The base case $\delta(\mbb{M}) = \varnothing$ follows by definition. Next, fix $\revision{\ell} > 0$, and suppose that the claim has been shown for all $\mbb{M}$ such that $|\delta(\mbb{M})| < \ell$. Fix $\mbb{M}$ with $\delta(\mbb{M}) = \revision{\ell}$. Let $T \in \delta(\mbb{M})$, and define the two intervals of forests $\mbb{M}_1 := \{\mc{F} \in \mbb{M} : \mc{F} \not\ni T\}$, $\mbb{M}_2 = \{\mc{F} \in \mbb{M} : \mc{F} \ni T\}$. Note that we may write
\begin{equs}
\mbb{M}_1 = [s(\mbb{M}), b(\mbb{M}) \backslash \{T\}], ~~ \mbb{M}_2 = [s(\mbb{M}) \cup \{T\}, b(\mbb{M})].
\end{equs}
Consequently,
\begin{equs}
s(\mbb{M}_1) &= s(\mbb{M}), \quad \quad\quad  b(\mbb{M}_1) = b(\mbb{M}) \backslash \{T\}, \\
s(\mbb{M}_2) &= s(\mbb{M}) \cup \{T\}, \quad b(\mbb{M}_2) = b(\mbb{M}), \\
\delta(\mbb{M}_1) &= \delta(\mbb{M}_2) = \delta(\mbb{M}) \backslash \{T\}.
\end{equs}
Note also that $\mbb{M} = \mbb{M}_1 \cup \mbb{M}_2$. Applying the inductive assumption to $\mbb{M}_1, \mbb{M}_2$, it suffices to show that
\begin{equs}\label{eq:W-M-decomp}
\mc{W}^{\determ, \mbf{n}}_\lambda[\mc{J}, \mbb{M}, \varnothing] = \mc{W}^{\determ, \mbf{n}}_\lambda[\mc{J}, \mbb{M}_1, \varnothing] + \mc{W}^{\determ, \mbf{n}}_\lambda[\mc{J}, \mbb{M}_2, \varnothing].
\end{equs}
By definition, we have that
\begin{equs}
\mc{W}^{\determ, \mbf{n}}_\lambda[\mc{J}, \mbb{M}, \varnothing] = &\int_{N(b(\mbb{M}), D_{2p}) \sqcup \{\ostar\}} \dy ~ \delta(y_{\ostar}) \bigg(\prod_{u \in N(b(\mbb{M}), D_{2p})} e^{\icomplex \beta \nodelabel(u) \determ(y_u)}\bigg) X(y) J(y) K(y) \\
&\cdot\bar{H}^{\determ, \mbf{n}}_{\mc{J}, \mbb{M}, \ovl{b(\mbb{M})}}\Big[ \mc{J}_{\mbf{n}}^{P^\ptl_{b(\mbb{M})}(D_{2p})} \mrm{Ker}_{\mbf{n}}^{K^\downarrow(\ovl{b(\mbb{M})})} X^{\tilde{N}(\ovl{b(\mbb{M})})}_{\bar{\mathfrak{n}}, \ostar, \mbf{n}} \Big](y).
\end{equs}
Here,
\begin{equs}
X(y) := X^{N(b(\mbb{M}), D_{2p})}_{\bar{\mathfrak{n}}, \ostar, \mbf{n}} (y), ~
J(y) := \mc{J}_{\mbf{n}}^{L(b(\mbb{M}), D_{2p})^{(2)}}(y), ~~ \text{and} ~~
K(y) := \mrm{Ker}_{\mbf{n}}^{K(b(\mbb{M}), D_{2p})}(y).
\end{equs}
\revision{We remark that starting from here, our assumption that $\beta^2 < 6\pi$ results in simplifications as compared to the general proof of \cite[Lemma 4.2]{CHS2018}. Because we take advantage of these simplifications, our proof begins to diverge from that of \cite{CHS2018}. In particular,} note that for us, every element of $b(\mbb{M})$ is maximal, and thus $\ovl{b(\mbb{M})} = b(\mbb{M})$. By the inductive definition of $\bar{H}^{\determ, \mbf{n}}_{\mc{J}, \mbb{M}, b(\mbb{M})}$, we have that
\begin{align*}
\relax [\bar{H}^{\determ, \mbf{n}}_{\mc{J}, \mbb{M}, b(\mbb{M})} (\varphi)](y) = &\int_{\tilde{N}(T)} dy_{\tilde{N}(T)} \prod_{u \in \tilde{N}(T)} \bigg(e^{\icomplex \beta \nodelabel(u) \determ(y_u)}\bigg) \mc{J}_{\mbf{n}}^{L(T)^{(2)}}(y) \mrm{Ker}_{\mbf{n}}^{\mathring{K}_{b(\mbb{M})}(T)}(y) \allowdisplaybreaks[0]\\
&\cdot \bar{H}^{\determ, \mbf{n}}_{\mc{J}, \mbb{M}, b(\mbb{M}) \backslash \{T\}}\Big[ (\mrm{Id} - \mathscr{Y}_T) \varphi\Big](y) .
\end{align*}
Here, we use that $P^\ptl_{b(\mbb{M})}(S) = K^\ptl_{b(\mbb{M})}(S) = \varnothing$ in the case of dipoles and tripoles. We are also being a bit loose with the variable $y$; on the LHS, $y$ is indexed by nodes not in $\tilde{N}(T)$, while in the RHS, we integrate over the non-root vertices of $T$ (which for us is just a single vertex since $T$ is always a dipole), and the $y$ appearing as inputs to functions on the RHS is the concatenation of the input $y$ from the LHS and the integration variables.

We may split the above into two terms $H_1(\varphi) + H_2(\varphi)$, where 
\begin{align*}
&\, \relax [H_1(\varphi)](y) \\:=&\, \int_{\tilde{N}(T)} \hspace{-1ex} dy_{\tilde{N}(T)} \prod_{u \in \tilde{N}(T)} \hspace{-1ex}  \bigg(e^{\icomplex \beta \nodelabel(u) \determ(y_u)}\bigg) \mc{J}_{\mbf{n}}^{L(T)^{(2)}}(y) \mrm{Ker}_{\mbf{n}}^{\mathring{K}_{b(\mbb{M})}(T)}(y) \bar{H}^{\determ, \mbf{n}}_{\mc{J}, \mbb{M}, b(\mbb{M}) \backslash \{T\}}[\varphi](y) 
\end{align*}
and 
\begin{align*}
&\, \relax [H_2(\varphi)](y) \\
:=&\, \int_{\tilde{N}(T)}  \hspace{-1ex}  dy_{\tilde{N}(T)} \prod_{u \in \tilde{N}(T)} \hspace{-1ex}  \bigg(e^{\icomplex \beta \nodelabel(u) \determ(y_u)}\bigg) \mc{J}_{\mbf{n}}^{L(T)^{(2)}}(y) \mrm{Ker}_{\mbf{n}}^{\mathring{K}_{b(\mbb{M})}(T)}(y) \bar{H}^{\determ, \mbf{n}}_{\mc{J}, \mbb{M}, b(\mbb{M}) \backslash \{T\}}[ -\mathscr{Y}_T \varphi ](y).  
\end{align*}
To show the identity \eqref{eq:W-M-decomp}, it suffices to show that
\begin{align}
\mc{W}^{\determ, \mbf{n}}_\lambda[\mc{J}, \mbb{M}_1, \varnothing] = &\int_{N(b(\mbb{M}), D_{2p}) \sqcup \{\ostar\}} \dy ~ \delta(y_{\ostar}) \bigg(\prod_{u \in N(b(\mbb{M}), D_{2p})} e^{\icomplex \beta \nodelabel(u) \determ(y_u)}\bigg) X(y) J(y) K(y) \notag \\
&\cdot H_1\Big[\mc{J}_{\mbf{n}}^{P^\ptl_{b(\mbb{M})}(D_{2p})} \mrm{Ker}_{\mbf{n}}^{K^\downarrow(\ovl{b(\mbb{M})})} X^{\tilde{N}(\ovl{b(\mbb{M})})}_{\bar{\mathfrak{n}}, \ostar, \mbf{n}} \Big](y),
\label{eq:M-1-identity}
\end{align}
and
\begin{equs}
\mc{W}^{\determ, \mbf{n}}_\lambda[\mc{J}, \mbb{M}_2, \varnothing] = &\int_{N(b(\mbb{M}), D_{2p}) \sqcup \{\ostar\}} \dy ~ \delta(y_{\ostar}) \bigg(\prod_{u \in N(b(\mbb{M}), D_{2p})} e^{\icomplex \beta \nodelabel(u) \determ(y_u)}\bigg) X(y) J(y) K(y) \\
&\cdot H_2\Big[\mc{J}_{\mbf{n}}^{P^\ptl_{b(\mbb{M})}(D_{2p})} \mrm{Ker}_{\mbf{n}}^{K^\downarrow(\ovl{b(\mbb{M})})} X^{\tilde{N}(\ovl{b(\mbb{M})})}_{\bar{\mathfrak{n}}, \ostar, \mbf{n}} \Big](y).
\end{equs}
We first show the identity for $\mbb{M}_2$, since it has a shorter proof. Since $T \in s(\mbb{M}_2)$, by definiton of $\bar{H}^{\determ, \mbf{n}}$, we have that (recalling that $b(\mbb{M}_2) = b(\mbb{M})$)
\begin{equs}
H_2(\varphi) = \bar{H}^{\determ, \mbf{n}}_{\mc{J}, \mbb{M}_2, b(\mbb{M}_2)}(\varphi).
\end{equs}
The desired identity for $\mbb{M}_2$ now follows by combining this with the definition of $\mc{W}^{\determ, \mbf{n}}_\lambda[\mc{J}, \mbb{M}_2, \varnothing]$, and the fact that $b(\mbb{M}_2) = b(\mbb{M})$.

In the case of $\mbb{M}_1$ (i.e., we now want to show the identity \eqref{eq:M-1-identity}), note that since $b(\mbb{M}_1) = b(\mbb{M}) \backslash \{T\}$, we have that $N(b(\mbb{M}_1), D_{2p}) = N(b(\mbb{M}), D_{2p}) \cup \tilde{N}(T)$. We may thus write the RHS of \eqref{eq:M-1-identity}
\begin{align*}
&\int_{N(b(\mbb{M}_1), D_{2p}) \cup \{\ostar\}} \dy \delta(y_{\ostar}) \bigg(\prod_{u \in N(b(\mbb{M}_1), D_{2p})} e^{\icomplex \beta \nodelabel(u) \determ(y_u)}\bigg) X(y) J(y) K(y) \\
&\hspace{5ex}\cdot \mc{J}_{\mbf{n}}^{L(T)^{(2)}}(y) \mrm{Ker}_{\mbf{n}}^{\mathring{K}_{b(\mbb{M})}(T)}(y) \bar{H}^{\determ, \mbf{n}}_{\mc{J}, \mbb{M}, b(\mbb{M}) \backslash \{T\}}\Big[\mc{J}_{\mbf{n}}^{P^\ptl_{b(\mbb{M})}(D_{2p})} \mrm{Ker}_{\mbf{n}}^{K^\downarrow(\ovl{b(\mbb{M})})} X^{\tilde{N}(\ovl{b(\mbb{M})})}_{\bar{\mathfrak{n}}, \ostar, \mbf{n}}\Big](y).
\end{align*}
From the definition of $\mc{W}^\determ_\lambda[\mc{J}, \mbb{M}_1, \varnothing]$, we want to show that this is equal to
\begin{equs}\label{eq:W-M-1-formula}
\int_{N(b(\mbb{M}_1), D_{2p}) \cup \{\ostar\}} &\dy \delta(y_{\ostar}) \bigg(\prod_{u \in N(b(\mbb{M}_1), D_{2p})} e^{\icomplex \beta \nodelabel(u) \determ(y_u)}\bigg) X_1(y) J_1(y) K_1(y) \\
&\cdot\bar{H}^{\determ, \mbf{n}}_{\mc{J}, \mbb{M}_1, b(\mbb{M}_1)}\Big[\mc{J}_{\mbf{n}}^{P^\ptl_{b(\mbb{M}_1)}(D_{2p})} \mrm{Ker}_{\mbf{n}}^{K^\downarrow(\ovl{b(\mbb{M}_1)})} X^{\tilde{N}(\ovl{b(\mbb{M}_1)})}_{\bar{\mathfrak{n}}, \ostar, \mbf{n}}\Big](y),
\end{equs}
where
\begin{equs}
X_1(y) &:= X^{N(b(\mbb{M}_1), D_{2p})}_{\bar{\mathfrak{n}}, \ostar, \mbf{n}} (y), \\
J_1(y) &:= \mc{J}_{\mbf{n}}^{L(b(\mbb{M}_1), D_{2p})^{(2)}}(y), \\
K_1(y) &:= \mrm{Ker}_{\mbf{n}}^{K(b(\mbb{M}_1), D_{2p})}.
\end{equs}
This identity will be shown by shifting some of the factors outside of $\bar{H}^{\determ, \mbf{n}}_{\mc{J}, \mbb{M}_1, b(\mbb{M}_1)}$ into the argument of this operator. Towards this end, note that since $N(b(\mbb{M}_1), D_{2p}) = N(b(\mbb{M}), D_{2p}) \cup \tilde{N}(T)$, we have that
\begin{equs}
X_1 = X \cdot X^{\tilde{N}(T)}_{\bar{\mathfrak{n}}, \ostar, \mbf{n}}.
\end{equs}
Similarly, we have that $L(b(\mbb{M}_1), D_{2p})^{(2)} = L(b(\mbb{M}), D_{2p})^{(2)} \cup L(T)^{(2)} \cup P$, where $P$ is the set of all pairs of nodes such that one node is in $T$ and the other node is in $D_{2p} \backslash b(\mbb{M})$. Note that $L(T)^{(2)} \cup P$ is the set of pairs of nodes which involve at least one node of $T$, and no nodes of $b(\mbb{M}) \backslash \{T\}$. Thus,
\begin{equs}
J_1 = J \cdot \mc{J}_{\mbf{n}}^{L(T)^{(2)}} \cdot \mc{J}_{\mbf{n}}^P.
\end{equs}
Next, we have that
\begin{equs}
&\, K(b(\mbb{M}_1), D_{2p}) := K(D_{2p}) \backslash \bar{K}^{\downarrow}(b(\mbb{M}_1)) \\
=&\, \big(K(D_{2p}) \backslash \bar{K}^{\downarrow}(b(\mbb{M}))\big) \cup \bar{K}^{\downarrow}(T) = K(b(\mbb{M}), D_{2p}) \cup \bar{K}^{\downarrow}(T).
\end{equs}
Moreover, $\bar{K}^{\downarrow}(T) = K(T) \sqcup K^{\downarrow}(T)$. Thus
\begin{equs}
K_1 = K \cdot \mrm{Ker}_{\mbf{n}}^{K(T)} \cdot \mrm{Ker}_{\mbf{n}}^{{K}^{\downarrow}(T)}.
\end{equs}
We also note that in our case, $\mathring{K}_{b(\mbb{M})}(T) = K(T)$. Inserting these identities into \eqref{eq:W-M-1-formula}, we obtain
\begin{align*}
&\hspace{-2ex}\int_{N(b(\mbb{M}_1), D_{2p}) \cup \{\ostar\}} \dy \delta(y_{\ostar}) \bigg(\prod_{u \in N(b(\mbb{M}_1), D_{2p})} e^{\icomplex \beta \nodelabel(u) \determ(y_u)}\bigg) \\
& X(y) J(y) K(y) \mc{J}_n^{L(T)^{(2)}}(y) \mrm{Ker}_{\mbf{n}}^{\mathring{K}_{b(\mbb{M})}(T)}(y)  \\
& X^{\tilde{N}(T)}_{\bar{\mathfrak{n}}, \ostar, \mbf{n}}(y) \mc{J}_{\mbf{n}}^P(y) \mrm{Ker}_{\mbf{n}}^{{K}^{\downarrow}(T)}(y) \cdot\bar{H}^{\determ, \mbf{n}}_{\mc{J}, \mbb{M}_1, b(\mbb{M}_1)}\Big[\mc{J}_{\mbf{n}}^{P^\ptl_{b(\mbb{M}_1)}(D_{2p})} \mrm{Ker}_{\mbf{n}}^{K^\downarrow(\ovl{b(\mbb{M}_1)})} X^{\tilde{N}(\ovl{b(\mbb{M}_1)})}_{\bar{\mathfrak{n}}, \ostar, \mbf{n}}\Big](y).
\end{align*}
To finish the proof of the identity \eqref{eq:M-1-identity}, we claim that
\begin{equs}
X^{\tilde{N}(T)}_{\bar{\mathfrak{n}}, \ostar, \mbf{n}}(y)& \mc{J}_{\mbf{n}}^P(y) \mrm{Ker}_{\mbf{n}}^{{K}^{\downarrow}(T)}(y) \cdot\bar{H}^{\determ, \mbf{n}}_{\mc{J}, \mbb{M}_1, b(\mbb{M}_1)}\Big[\mc{J}_{\mbf{n}}^{P^\ptl_{b(\mbb{M}_1)}(D_{2p})} \mrm{Ker}_{\mbf{n}}^{K^\downarrow(\ovl{b(\mbb{M}_1)})} X^{\tilde{N}(\ovl{b(\mbb{M}_1)})}_{\bar{\mathfrak{n}}, \ostar, \mbf{n}}\Big](y) \\
&= \bar{H}^{\determ, \mbf{n}}_{\mc{J}, \mbb{M}, b(\mbb{M}) \backslash \{T\}}\Big[\mc{J}_{\mbf{n}}^{P^\ptl_{b(\mbb{M})}(D_{2p})} \mrm{Ker}_{\mbf{n}}^{K^\downarrow(\ovl{b(\mbb{M})})} X^{\tilde{N}(\ovl{b(\mbb{M})})}_{\bar{\mathfrak{n}}, \ostar, \mbf{n}}\Big](y).
\end{equs}
Towards this end, we first note that 
\begin{equs}
\bar{H}^{\determ, \mbf{n}}_{\mc{J}, \mbb{M}_1, b(\mbb{M}_1)} = \bar{H}^{\determ, \mbf{n}}_{\mc{J}, \mbb{M}, b(\mbb{M}) \backslash \{T\}}.
\end{equs}
Next, observe that $\tilde{N}(T) \cup \tilde{N}(\ovl{b(\mbb{M}_1)}) = \tilde{N}(\ovl{b(\mbb{M})})$, $K^{\downarrow}(T) \cup K^{\downarrow}(\ovl{b(\mbb{M}_1)}) = K^{\downarrow}(\ovl{b(\mbb{M})})$, and \mbox{$P \cup P^\ptl_{b(\mbb{M}_1)}(D_{2p}) = P^\ptl_{b(\mbb{M})}(D_{2p})$}. Thus,
\begin{equs}
X^{\tilde{N}(T)}_{\bar{\mathfrak{n}}, \ostar, \mbf{n}}(y)& \mc{J}_{\mbf{n}}^P(y) \mrm{Ker}_{\mbf{n}}^{{K}^{\downarrow}(T)}(y) \cdot \mc{J}_{\mbf{n}}^{P^\ptl_{b(\mbb{M}_1)}(D_{2p})} \mrm{Ker}_{\mbf{n}}^{K^\downarrow(\ovl{b(\mbb{M}_1)})} X^{\tilde{N}(\ovl{b(\mbb{M}_1)})}_{\bar{\mathfrak{n}}, \ostar, \mbf{n}} \\
&=\mc{J}_{\mbf{n}}^{P^\ptl_{b(\mbb{M})}(D_{2p})} \mrm{Ker}_{\mbf{n}}^{K^\downarrow(\ovl{b(\mbb{M})})} X^{\tilde{N}(\ovl{b(\mbb{M})})}_{\bar{\mathfrak{n}}, \ostar, \mbf{n}}.
\end{equs}
Finally, to see why we may take $X^{\tilde{N}(T)}_{\bar{\mathfrak{n}}, \ostar, \mbf{n}}(y) \mc{J}_{\mbf{n}}^P(y) \mrm{Ker}_{\mbf{n}}^{{K}^{\downarrow}(T)}(y)$ inside the argument of $\bar{H}^{\determ, \mbf{n}}_{\mc{J}, \mbb{M}, b(\mbb{M}) \backslash \{T\}}$, note that $\bar{H}^{\determ, \mbf{n}}_{\mc{J}, \mbb{M}, b(\mbb{M}) \backslash \{T\}}$ integrates its input over variables indexed by the nodes $\tilde{N}(b(\mbb{M}) \backslash \{T\})$, while the term $X^{\tilde{N}(T)}_{\bar{\mathfrak{n}}, \ostar, \mbf{n}}(y) \mc{J}_{\mbf{n}}^P(y) \mrm{Ker}_{\mbf{n}}^{{K}^{\downarrow}(T)}(y)$ only involves variables indexed by the nodes in $N(T) \cup (D_{2p} \backslash b(\mbb{M})) \cup \{\ostar\}$, which is disjoint from the former set. Thus, $X^{\tilde{N}(T)}_{\bar{\mathfrak{n}}, \ostar, \mbf{n}}(y) \mc{J}_{\mbf{n}}^P(y) \mrm{Ker}_{\mbf{n}}^{{K}^{\downarrow}(T)}(y)$ may be treated as a constant, and thus taken inside the input of the renormalization operator. This completes the proof of \eqref{eq:M-1-identity}.

Next, we show that for any cut set $\mathscr{C}$ such that $\mathscr{C} \sse \mathfrak{C} \backslash K(b(\mbb{M}))$, we have that
\begin{equs}
\sum_{\mc{F} \in \mbb{M}} \mc{W}^{\determ, \mbf{n}}_\lambda[\mc{J}, \{\mc{F}\}, \{\mathscr{C}\}] = \mc{W}^{\determ, \mbf{n}}_\lambda[\mc{J}, \mbb{M}, \{\mathscr{C}\}].
\end{equs}
This is shown by a slight extension of the previous argument, since now there are additional terms in the formula for $\mc{W}^{\mbf{n}}_\lambda$. One proceeds by the same inductive argument, except now the definition of $K$ is modified to
\begin{equs}
K := \mrm{Ker}_{\mbf{n}}^{K(b(\mbb{M}), D_{2p}) \backslash \mathscr{C}},
\end{equs}
and there is an additional factor $R$ defined by
\begin{equs}
R := \mrm{RKer}_{\mbf{n}}^{\mathscr{C} \backslash K^{\downarrow}(b(\mbb{M}))}.
\end{equs}
Additionally, the input to the renormalization operator is modified to
\begin{equs}
\mc{J}_{\mbf{n}}^{P^\ptl_{b(\mbb{M})}(D_{2p})}  \mrm{RKer}_{\mbf{n}}^{\mathscr{C} \cap K^\downarrow(\ovl{b(\mbb{M})}) } \mrm{Ker}_{\mbf{n}}^{K^\downarrow(\ovl{b(\mbb{M})}) \backslash \mathscr{C}} X^{\tilde{N}(\ovl{b(\mbb{M})})}_{\bar{\mathfrak{n}}, \ostar, \mbf{n}}.
\end{equs}
As before, we partition $\mbb{M} = \mbb{M}_1 \cup \mbb{M}_2$, and want to show
\begin{equs}
\mc{W}^{\determ, \mbf{n}}_\lambda[\mc{J}, \mbb{M}, \{\mathscr{C}\}] = \mc{W}^{\determ, \mbf{n}}_\lambda[\mc{J}, \mbb{M}_1, \{\mathscr{C}\}] + \mc{W}^{\determ, \mbf{n}}_\lambda[\mc{J}, \mbb{M}_2, \{\mathscr{C}\}].
\end{equs}
We then reduce to proving two separate identities, one each for $\mbb{M}_1, \mbb{M}_2$. The proof for $\mbb{M}_2$ is short as before, while the proof for $\mbb{M}_1$ involves some reorganizations. We omit the details.

Finally, to show the full result, it remains to show that for an interval of cutsets $\mbb{G}$ with $b(\mbb{G}) \sse \mathfrak{C} \backslash K(b(\mbb{M}))$, we have that
\begin{equs}
\sum_{\mathscr{C} \in \mbb{G}} \mc{W}^{\determ, \mbf{n}}_\lambda[\mc{J}, \mbb{M}, \{\mathscr{C}\}] = \mc{W}^{\determ, \mbf{n}}_\lambda[\mc{J}, \mbb{M}, \mbb{G}].
\end{equs}
This follows since $\widehat{\mrm{RKer}}_{\mbf{n}}^e = \mrm{Ker}_{\mbf{n}}^e + \mrm{RKer}_{\mbf{n}}^e$ and expanding out
\begin{equs}
\widehat{\mrm{RKer}}_{\mbf{n}}^{\delta(\mbb{G}) \backslash K^{\downarrow}(b(\mbb{M}))} = \sum_{\mathscr{C}_1 \sse \delta(\mbb{G}) \backslash K^{\downarrow}(b(\mbb{M}))} \mrm{Ker}_{\mbf{n}}^{\mathscr{C}_1} \cdot \mrm{RKer}_{\mbf{n}}^{\delta(\mbb{G}) \backslash (K^{\downarrow}(b(\mbb{M})) \cup \mathscr{C}_1)},
\end{equs}
and similarly
\begin{equs}
\widehat{\mrm{RKer}}_{\mbf{n}}^{\delta(\mbb{G}) \cap K^{\downarrow}(\ovl{b(\mbb{M}))}} = \sum_{\mathscr{C}_2 \sse \delta(\mbb{G}) \cap K^{\downarrow}(\ovl{b(\mbb{M}))}} \mrm{Ker}_{\mbf{n}}^{\mathscr{C}_2} \cdot \mrm{RKer}_{\mbf{n}}^{(\delta(\mbb{G}) \cap K^{\downarrow}\ovl{b(\mbb{M})}) \backslash \mathscr{C}_2}.
\end{equs}
Thus we obtain that $\mc{W}^{\determ, \mbf{n}}_\lambda[\mc{J}, \mbb{M}, \mbb{G}]$ is a sum over $\mathscr{C}_1, \mathscr{C}_2$, which we may alternatively view as a sum over $\mathscr{C}' = \mathscr{C}_1 \cup \mathscr{C}_2 \sse \delta(\mbb{G})$. We may further view this as a sum over $\mathscr{C} = s(\mbb{G}) \cup \mathscr{C}' \in [s(\mbb{G}), b(\mbb{G})] = \mbb{G}$. The desired identity then follows by the definition of $\mc{W}^{\mbf{n}}_\lambda[\mc{J}, \mbb{M}, \{\mathscr{C}\}]$.
\end{proof}

The content of \cite[Sections 4.2.1 and 4.2.2]{CHS2018} before Corollary 4.7 is entirely unchanged. The following corollary is the analog of \cite[Corollary 4.7]{CHS2018}, and the proof is exactly the same.

\begin{corollary}[Analog of Corollary 4.7 of \cite{CHS2018}]
We have that
\begin{equs}
\sum_{\substack{\mc{F} \in \mbb{F} \\ \mathscr{C} \sse \mathfrak{C} \backslash K(\mc{F})}} \mc{W}_\lambda^{\determ}[\mc{J}, \mc{F}, \mathscr{C}] &= \sum_{\mbf{n} \in \N^{\mc{E}}} \sum_{\substack{\mbb{M} \in \mathfrak{M}^{\mbf{n}} \\ \mbb{G} \in \mathfrak{G}^{\mbf{n}}(\mbb{M})}} \mc{W}^{\determ, \mbf{n}}_\lambda[\mc{J}, \mbb{M}, \mbb{G}] \\
&= \sum_{(\mbb{M}, \mbb{G}) \in \mathfrak{R}} \sum_{\mbf{n} \in \mc{N}_{\mbb{M}, \mbb{G}, \lambda}} \mc{W}^{\determ, \mbf{n}}_\lambda[\mc{J}, \mbb{M}, \mbb{G}].
\end{equs}
\end{corollary}

As in \cite{CHS2018}, the desired estimate \eqref{eq:bar-M-lambda-estimate} follows from the previous discussion combined with the following proposition.

\begin{proposition}[Analog of Proposition 4.8 of \cite{CHS2018}]\label{prop:moment-estimate-fixed-intervals}
For any $(\mbb{M}, \mbb{G}) \in \mathfrak{R}$, we have that uniformly in $\lambda \in (0, 1]$,
\begin{equs}
\sum_{\mbf{n} \in \mc{N}_{\mbb{M}, \mbb{G}, \lambda}} |\mc{W}_\lambda^{\determ, \mbf{n}}[\mc{J}, \mbb{M}, \mbb{G}]| \lesssim \|\mc{J}\| \lambda^{2p |\bar{T}^{\bar{\mathfrak{n}} \bar{\mathfrak{l}}}|_{\mathfrak{s}}}.
\end{equs}
\end{proposition}

In the following discussion, we discuss the proof of this proposition. Towards this end, we first discuss how \cite[Section 5]{CHS2018} needs to be adjusted. For $S \in \mc{B}$, $\mbf{j} \in \N^{\mc{E}'}$ with $\mc{E}' \supseteq \mc{E}^{\mrm{ext}}_{\mc{B}}(S)$, recursively define
\begin{align}
\relax [\hat{H}^{\determ, \mbf{j}}_{\mc{J}, S} \varphi](x) := &\sum_{\mbf{k} \in \mathring{\mc{N}}_S(\mbf{j})} \int_{\tilde{N}_{\mc{B}}(S)} dy \prod_{u \in \tilde{N}_{\mc{B}}(S)} e^{\icomplex \beta \nodelabel(u) \determ(y_u)} \mc{J}_{\mbf{k}}^{L_{\mc{B}}(S)^{(2)}}(y) \mrm{Ker}_{\mbf{k}}^{\mathring{K}_{\mc{B}}(S)}(y \sqcup x_{\varrho_S}) \notag \\
&\cdot \hat{H}^{\determ, \mbf{k}}_{\mc{J}, C_{\mc{B}}(S)}\Big[ \mc{J}_{\mbf{k}}^{P^\ptl_{\mc{F}}(S)} \mrm{Ker}_{\mbf{k}}^{K^\ptl_{\mc{B}}(S)}[\mathscr{Y}^{\#}_{S, \mbb{M}} \varphi] \Big](x_{\tilde{N}(S)^c} \sqcup y),
\end{align}
where the base case $\hat{H}^{\determ, \mbf{j}}_{\mc{J}, \varnothing}$ is defined to be the identity operator. Then, define $\hat{\mc{W}}^{\determ, \mbf{j}}_\lambda[\mc{J}, \mbb{M}, \mbb{G}]$ as in \cite[Equation (5.4)]{CHS2018}, except with an additional factor $\prod_{u \in N(\mc{B}, D_{2p})} e^{\icomplex \beta \nodelabel(u) \determ(y_u)}$, and with $\hat{H}^{\determ, \mbf{j}}_{\mc{J}, \mc{B}}$ replacing $\hat{H}^{\mbf{j}}_{\mc{J}, \mc{B}}$.

\begin{lemma}[Analog of Lemma 5.1 of \cite{CHS2018}]
For any $\lambda \in (0, 1]$, we have that
\begin{equs}
\sum_{\mbf{j} \in \ptl \mc{N}_{\mc{B}, \lambda}} \int_{N(\mc{B}, D_{2p}) \sqcup \{\ostar\}} \delta(y_{\ostar}) \hat{\mc{W}}^{\determ, \mbf{j}}_\lambda[\mc{J}, \mbb{M}, \mbb{G}](y) dy = \sum_{\mbf{n} \in \mc{N}_{\mbb{M}, \mbb{G}, \lambda}} \mc{W}^{\determ, \mbf{n}}_\lambda[\mc{J}, \mbb{M}, \mbb{G}].
\end{equs}
\end{lemma}

Lemmas 5.2 and 5.3 of \cite{CHS2018} are unchanged, because they do not involve our modified definitions. 

\begin{lemma}[Analog of Lemma 5.5 of \cite{CHS2018}]
Let $\mc{F} \sse \mc{B}$ with $\mrm{depth}(\mc{F}) \leq 1$. Then, uniform in $\mc{J} \in \mathfrak{J}$, $x \in (\R^d)^{\tilde{N}(\mc{F})^c}$, $\mbf{j} \in \mbf{N}^{\mc{E}'}$ with $\mc{E}' \supseteq \mc{E}^{\mrm{ext}}_{\mc{B}}(\mc{F})$, and $\varphi \in \mathscr{C}$, we have that
\begin{equs}
|\hat{H}^{\determ, \mbf{j}}_{\mc{J}, \mc{F}}[\varphi](x)| \lesssim \bigg(\prod_{S \in \mc{F}} 2^{-|S^0|_{\mrm{SG}} \mrm{ext}^{\mbf{j}}_{\mc{B}}(S)} \|\mc{J}\|_{L(S)}\bigg) \|\varphi\|_{\mc{F}, \mbf{j}}(x).
\end{equs}
\end{lemma}
\begin{proof}
The proof is essentially exactly the same as the proof of \cite[Lemma 5.5]{CHS2018}. The only difference is that the function $F^{\mbf{k}}$ appearing there is modified, in that we define $F^{\determ, \mbf{k}}$ by
\begin{align*}
F^{\determ, \mbf{k}}(y) 
&:= \prod_{u \in \tilde{N}_{\mc{B}}(S)} e^{\icomplex \beta \nodelabel(u) \determ(y_u)} \mc{J}_{\mbf{k}}^{L_{\mc{B}}(S)^{(2)}}(y) \mrm{Ker}_{\mbf{k}}^{\tilde{K}_{\mc{B}}(S)}(y) \\
&\hspace{3ex}\cdot \hat{H}^{\determ, \mbf{k}}_{\mc{J}, C_{\mc{B}}(S)} \Big[ \mc{J}_{\mbf{k}}^{P^\ptl_{\mc{F}}(S)} \mrm{Ker}_{\mbf{k}}^{K^\ptl_{\mc{B}}(S)}[\mathscr{Y}^{\#}_{S, \mbb{M}} \varphi] \Big](x_{\tilde{N}(S)^c} \sqcup y),
\end{align*}
with the only difference being the additional $\prod_{u \in \tilde{N}_{\mc{B}}(S)} e^{\icomplex \beta \nodelabel(u) \determ(y_u)}$ factor and the use of $\hat{H}^{\determ, \mbf{k}}_{\mc{J}, C_{\mc{B}}(S)}$ instead of $\hat{H}^{\mbf{k}}_{\mc{J}, C_{\mc{B}}(S)}$. With this definition, we have that
\begin{equs}
\hat{H}^{\determ, \mbf{j}}_{\mc{J}, S}[\varphi](x \sqcup y_{\varrho_S}) = \sum_{\mbf{k} \in \mathring{N}_S(\mbf{j})} \int_{\tilde{N}_{\mc{B}}(S)} dy F^{\determ, \mbf{k}}(y).
\end{equs}
The inductive argument to bound this is then exactly the same as in \cite{CHS2018}, because of the trivial bound 
\[
\bigg|\prod_{u \in \tilde{N}_{\mc{B}}(S)} e^{\icomplex \beta \nodelabel(u) \determ(y_u)}\bigg| \leq 1.
\]
We remark that \cite[Lemma 5.6]{CHS2018}, which is needed in the proof of \cite[Lemma 5.5]{CHS2018} as well as the current proof, remains exactly the same since it does not involve our modified definitions.
\end{proof}

Lemmas 5.7 of \cite{CHS2018} remains unchanged, as it does not involve our modified definitions. The analog of Lemma 5.8 of \cite{CHS2018} (with $\hat{H}^{\mbf{j}}_{\mc{J}, \ovl{\mc{B}}}$ in that lemma replaced by our modified $\hat{H}^{\determ, \mbf{j}}_{\mc{J}, \ovl{\mc{B}}}$) may be proven in the exact same manner\footnote{Technically, \cite{CHS2018} does not provide a proof of this lemma, with the point being that it is essentially the same proof as \cite[Lemma 8.15]{CH16}}. Finally, by combining the previous estimates, we may obtain 
\begin{equs}
\|\hat{\mc{W}}_\lambda^{\determ}\|_{\zeta, \mc{N}_{\mbf{G}}} \lesssim \lambda^{-2p|\mathfrak{s}|} \|\mc{J}\|,
\end{equs}
which is the analog of \cite[Lemma 5.9]{CHS2018}. The proof of Proposition \ref{prop:moment-estimate-fixed-intervals} now follows from this estimate and \cite[Lemmas 5.11, 5.12 and Theorem B.10]{CHS2018}.

\section{More calculations involving the modified model}\label{appendix:more-calculations}

In this section, we prove the identity \eqref{appendix:eq-model}. First, for a tree $\bar{T}^{\mathfrak{n} \mathfrak{l}} \in \mc{T}^-$, let $\hat{\model}_{x_\ostar}^{\determ, \varep}[\bar{T}^{\mathfrak{n} \mathfrak{l}}]$ be the model given by the right hand side of \eqref{appendix:eq-model}, i.e., when tested against a test function $\psi$, we obtain 
\begin{align}
\hat{\model}^{\determ, \varep}_{x_\ostar}[\bar{T}^{\bar{\mathfrak{n}} \bar{\mathfrak{l}}}](\psi) = \sum_{\substack{\mc{G} \in \mbb{F}_j \\ \mathscr{D} \sse \mathfrak{C} \backslash K(\mc{G})}} &\int_{N(\mc{G}, \bar{T})} dy \xi^{L(\mc{G}, \bar{T}), \varep}(y) \prod_{u \in N(\mc{G}, \bar{T})} e^{\icomplex \beta \nodelabel(u) \determ(y_u)} \cdot \mrm{Ker}^{K(\mc{G}, \bar{T}) \backslash \mathscr{D}}(y) \notag \\
&\cdot \psi(z_{\varrho_{\bar{T}}}) \cdot \mrm{RKer}^{K(\mc{G}, \bar{T}) \cap \mathscr{D}}(z) \cdot X_{\bar{\mathfrak{n}}, \ostar}^{N(\mc{G}, \bar{T})}(z) \label{appendix:eq-model-RHS}\\
&\cdot H_{\mc{J}, \mc{G}, \bar{\mc{G}}}^\determ \Big[ \mrm{RKer}^{K^{\downarrow}(\bar{\mc{G}}) \cap \mathscr{D}} \cdot \mrm{Ker}^{K^{\downarrow}(\bar{\mc{G}}) \backslash \mathscr{D}} X_{\bar{\mathfrak{n}}, \ostar}^{\tilde{N}(\bar{\mc{G}})}\Big](z).\notag 
\end{align}

With this, the identity \eqref{appendix:eq-model} can be stated as the equality $\model^{\determ, \varep}_{x_\ostar} = \hat{\model}^{\determ, \varep}_{x_\ostar}$. By the calculations of Lemma \ref{appendix:lem-premodel-to-model}, this will follow directly from the next lemma. Here, we only check three representative cases, as the proof for other dipoles and tripoles is very similar. 

\begin{lemma}\label{appendix:lem-model-via-recursive} 
For any $\varep>0$ and $x_\ostar \in (\R \times \T^2)^{\{ \ostar\}}$,  we have that
\begin{align}
\hat{\model}^{\determ, \varep}_{x_\ostar}\bigg[\,  \dipmp\, \bigg] &=  \xi^{\determ,\varep}_-  \big( K \ast \xi^{\determ,\varep}_+ \big) - \xi^{\determ,\varep}_- \big(K \ast \xi^{\determ,\varep}_+\big)(x_\ostar) - \E \Big[ \xi^{\determ,\varep}_- \big(K \ast \xi^{\determ,\varep}_+\big)\Big], 
\label{appendix:eq-calculation-model-dipole} \\
\hat{\model}^{\determ, \varep}_{x_\ostar}\bigg[\,  \vtripolempp\, \bigg] &= \xi^{\determ,\varep}_- \big( K \ast \xi^{\determ,\varep}_+ - (K \ast \xi^{\determ,\varep}_+ )(x_\ostar) \big)^2 \label{appendix:eq-calculation-model-vtripole} \\
&- 2\E \Big[ \xi^{\determ,\varep}_- \big( K \ast \xi^{\determ,\varep}_+\big) \Big] 
\, \big( K \ast \xi^{\determ,\varep}_+ -  (K \ast \xi^{\determ,\varep}_+ )(x_\ostar) \big), \notag  
 \\
\hat{\model}^{\determ, \varep}_{x_\ostar}\Bigg[\,  \ltripolepmp\, \Bigg] &= \xi^{\determ, \varep}_+ \bigg(K \ast \bigg( \model^{\determ, \varep}_{x_\ostar}\bigg[\,  \dipmp\, \bigg]\bigg) - K \ast \bigg(\model^{\determ, \varep}_{x_\ostar}\bigg[\,  \dipmp\, \bigg]\bigg)(x_{\ostar})\bigg)  \label{appendix:eq-calculation-model-ltripole}\\
&\quad - \E\Big[\xi^{\determ, \varep}_+ K \ast \xi^{\determ, \varep}_-\Big] \Big(K \ast \xi^{\determ, \varep}_+ - (K \ast \xi^{\determ, \varep}_+)(x_{\ostar})\Big) \\
&\quad - \sum_{|k|_{\mathfrak{s}} = 1} (\cdot - x_{\ostar})^k \xi^{\determ, \varep}_+ \bigg(D^k K \ast \model^{\determ, \varep}_{x_\ostar}\bigg[\,  \dipmp\, \bigg]\bigg)(x_{\ostar}). \label{appendix:eq-calculation-model-ltripole-end}
\end{align}
\end{lemma}

\begin{figure}[t]
\begin{center}
\begin{tabular}{P{1.5cm}|P{1.5cm}"P{1.5cm}|P{1.5cm}|P{1.5cm}|P{1.5cm}|P{1.5cm}|P{1.7cm}|}
$\bar{T}$ & $\mc{G}$ & $\bar{\mc{G}}$ & $K^\downarrow(\bar{\mc{G}})$
& $K(\mc{G},\bar{T})$ & $L(\mc{G},\bar{T})$ & $N(\mc{G},\bar{T})$ \\ \thickhline 
\multirow{2}{1.5cm}{\hspace{4ex}$\dipmp$}   & $\emptyset$ & $\emptyset$ & $\emptyset$ & $\big\{ e \big\}$
& $\big\{e_-,e_+\big\}$ & $\big\{e_-,e_+\big\}$ \\ \cline{2-7} 
& $\big\{\bar{T}\big\}$ & $\big\{ \bar{T} \big\}$ & $\emptyset$ & $\emptyset$
& $\emptyset$ & $\big\{e_-\big\}$ \\ \thickhline
\multirow{3}{1.5cm}{\hspace{2ex}$\vtripolempp$} 
& $\emptyset$ & $\emptyset$
& $\emptyset$ & $\big\{ e^1,e^2 \big\}$ & $\hspace{-1ex}\big\{ \varrho_{\bar{T}}, e^1_+, e^2_+ \big\}$ & $ \hspace{-1ex}\big\{ \varrho_{\bar{T}}, e^1_+, e^2_+ \big\}$ \\ \cline{2-7}
& $\big\{ S^1 \big\}$ & $\big\{ S^1 \big\} $ 
& $\big\{ e^2 \big\}$  & $\emptyset$
& $\big\{ e^2_+ \big\}$ 
& $\big\{ \varrho_{\bar{T}}, e^2_+ \big\} $ \\ \cline{2-7}
& $ \big\{ S^2 \big\}$ & $\big\{ S^2 \big\} $ 
& $\big\{ e^1 \big\} $ & $\emptyset$
& $\big\{ e^1_+ \big\}$ 
& $\big\{ \varrho_{\bar{T}}, e^1_+ \big\} $ \\[0ex]
\thickhline 
\multirow{3}{1.5cm}{\hspace{4ex}$\ltripolepmp$} & $\emptyset$ & $\emptyset$
& $\emptyset$ & $\big\{ e^1,e^2 \big\}$ & $\hspace{-1ex}\big\{ \varrho_{\bar{T}}, e^1_-, e^2_+ \big\}$ & $ \hspace{-1ex}\big\{ \varrho_{\bar{T}}, e^1_-, e^2_+ \big\}$ \\
\cline{2-7} & $\big\{ S^1 \big\}$ & $\big\{ S^1 \big\} $ 
& $\{e^2\}$ & $\emptyset$ 
& $\big\{ e^2_+ \big\}$ 
& $\big\{ \varrho_{\bar{T}}, e^2_+ \big\} $ 
\\
\cline{2-7}
& $ \big\{ S^2 \big\}$ & $\big\{ S^2 \big\} $ 
& $\emptyset$ & $\big\{ e^1 \big\} $
& $\big\{ \varrho_{\bar{T}} \big\}$ 
& $\big\{ \varrho_{\bar{T}}, e^1_- \big\} $ \\
\end{tabular}
\caption{We list the different sets appearing in \eqref{appendix:eq-model-RHS} for the dipole and tripoles from Lemma~\ref{appendix:lem-model-via-recursive}.}
\label{figure:dipole-forests}
\end{center}
\end{figure}

\begin{figure}
\begin{center}
\begin{tabular}{P{1.2cm}|P{1.2cm}|P{1.2cm}"P{1.2cm}|P{1.2cm}|P{1.2cm}|P{1.2cm}|P{1.2cm}|P{1.2cm}|}
$\bar{T}$ & $\mc{F}$ & $S$ & $C_{\mc{F}}(S)$ & $\mathring{K}_{\mc{F}}(S)$
& $K^\ptl_{\mc{F}}(S)$ & $L_{\mc{F}}(S)^{(2)}$ & $\tilde{N}_{\mc{F}}(S)$ \\ \thickhline 
\scalebox{0.85}{$\dipmp$} & $\big\{ \bar{T} \big\}$ & $\bar{T}$ &  $\emptyset$ & $\big\{e \big\} $ & $\emptyset$
& $\hspace{-1ex}\big\{(e_-,e_+)\big\}$ & $\big\{e_+\big\}$
\end{tabular}
\caption{We list the different sets  appearing in \eqref{eq:H-determ} for the dipole, the forest $\mc{F}=\{\bar{T}\}$, and the tree $S=\bar{T}$.}
\label{figure:dipole-tree}
\end{center}
\end{figure}

In the following proof, we heavily rely on the notation from Sections 2.6, 3.1, and 3.2 in \cite{CHS2018}. However, since it would take several pages, we do not recall it here.

\begin{proof}  
To match the notation in Proposition \ref{appendix:prop-model-and-moments}, we express all identities below using integrals against a test-function $\psi$. We separate our treatments of the dipole in \eqref{appendix:eq-calculation-model-dipole} and tripoles in \eqref{appendix:eq-calculation-model-vtripole} and \eqref{appendix:eq-calculation-model-ltripole}.\\

\emph{Case 1: The dipole from \eqref{appendix:eq-calculation-model-dipole}.} Throughout this argument, we label the nodes and edges of the dipole as in 
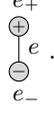
\begin{equation*}
\dipmplabeled.
\end{equation*} 
To simplify the notation, we denote this dipole by   $\bar{T}$. Since $4\pi \leq \beta^2 <6\pi$, it follows from \cite[Definition~2.17]{CHS2018} that $\gamma(e)=1$ and $\mathfrak{C}=\{ e \}$. As a result, the sum in \eqref{appendix:eq-model-RHS} ranges over 
\begin{equation}\label{appendix:eq-calculation-dipole-options}
\big( \mc{G}, \mathscr{D} \big) \in 
\Big \{ \big( \emptyset, \emptyset\big), \big( \emptyset, \{ e \} \big), \big(  \big\{ \bar{T} \big\}, \emptyset \big) \Big\}.
\end{equation}
For   the forests $\mc{G}=\emptyset, \{ \bar{T}\}$, the different sets of edges, nodes, and noise nodes appearing \eqref{appendix:eq-model-RHS} have been listed in Figure \ref{figure:dipole-forests}. Furthermore, for the forest $\mc{F}=\{ \bar{T}\}$ and the tree $S=\bar{T}$, the different sets appearing in~\eqref{eq:H-determ} have been listed in Figure~\ref{figure:dipole-tree}. We now write $\hat{\model}^{\determ, \varep}_{x_\ostar,\mc{G},\mathscr{D}}[\bar{T}](\psi)$ 
for the summands on the right-hand side of~\eqref{appendix:eq-model-RHS} and, using the information from Figure~\ref{figure:dipole-forests} and Figure~\ref{figure:dipole-tree}, compute them for the three options in \eqref{appendix:eq-calculation-dipole-options}.\\

\emph{Case 1.a: $\mc{G}=\emptyset$.} Using the first row of Figure \ref{figure:dipole-forests} and that $ H^{\determ}_{\mc{J}, \mc{F},\emptyset}=\operatorname{Id}$ for all forests $\mc{F}$, we obtain that 
\begin{align}
&\hspace{3ex} \hat{\model}^{\determ,\varep}_{x_\ostar,\emptyset,\mathscr{D}} \big[ \bar{T} \big](\psi) \notag \\
&= \int \dy_{e_-} \dy_{e_+} \, \psi(y_{e_-}) \xi^{\varep}_-(y_{e_-}) \xi^{\varep}_+(y_{e_+}) 
e^{\icomplex \beta (\determ(y_{e_+})-\determ(y_{e_-}))} \notag \\
&\hspace{10ex}\times \mrm{Ker}^{\{e\}\backslash \mathscr{D}}(y_{e_-},y_{e_+}) \, \mrm{RKer}^{\mathscr{D}}(y_{e_-},y_{e_+},x_{\ostar}) H^{\determ}_{\Jc_\varep,\emptyset,\emptyset}[1](y_{e_-},y_{e_+}) \notag \\
&= \int  \dy_{e_-} \dy_{e_+} \, \psi(y_{e_-}) \xi^{\determ,\varep}_-(y_{e_-}) \xi^{\determ,\varep}_+(y_{e_+})  \mrm{Ker}^{\{e\}\backslash \mathscr{D}}(y_{e_-},y_{e_+}) \, \mrm{RKer}^{\mathscr{D}}(y_{e_-},y_{e_+},x_{\ostar}). \label{appendix:eq-calculation-dipole-p1}
\end{align}
From \cite[(3.3) and (3.4)]{CHS2018}, it directly follows that 
\begin{equation}\label{appendix:eq-calculation-dipole-p2}
\mrm{Ker}^{\{e\}\backslash \mathscr{D}}(y_{e_-},y_{e_+}) \, \mrm{RKer}^{\mathscr{D}}(y_{e_-},y_{e_+},x_{\ostar})
= \begin{cases}
\begin{aligned}
K(y_{e_-}-y_{e_+}) \qquad &\text{if } \mathscr{D}=\emptyset, \\
-K(x_\ostar - y_{e_+}) \qquad &\text{if } \mathscr{D}=\{e\}.
\end{aligned}
\end{cases}
\end{equation}
By inserting \eqref{appendix:eq-calculation-dipole-p2} into \eqref{appendix:eq-calculation-dipole-p1}, we directly obtain that 
\begin{align*}
\hat{\model}^{\determ,\varep}_{x_\ostar,\emptyset,\emptyset} \big[ \bar{T} \big](\psi)
&= \int \dy_{e_-}\psi(y_{e_-}) \xi^{\determ,\varep}_-(y_{e_-})  \big( K \ast \xi^{\determ,\varep}_+ \big)(y_{e_-}) \\ 
\text{and} \qquad 
\hat{\model}^{\determ,\varep}_{x_\ostar,\emptyset,\{e\}} \big[ \bar{T} \big](\psi) 
&= - \big( K \ast \xi^{\determ,\varep}_+\big)(x_\ostar) \int \dy_{e_-}\psi(y_{e_-}) \xi^{\determ,\varep}_-(y_{e_-}). 
\end{align*}
Thus, $\hat{\model}^{\determ,\varep}_{x_\ostar,\emptyset,\emptyset} \big[ \bar{T} \big](\psi)$ and $\hat{\model}^{\determ,\varep}_{x_\ostar,\emptyset,\{e\}} \big[ \bar{T} \big](\psi) $ correspond to the first and second summand in \eqref{appendix:eq-calculation-model-dipole}, respectively.\\

\emph{Case 1.b: $\mc{G}=\{\bar{T}\}$.} From \eqref{appendix:eq-calculation-dipole-options}, it follows that $\mathscr{D}=\emptyset$. Using the second row of Figure \ref{figure:dipole-forests}, we obtain that 
\begin{align}\label{appendix:eq-calculation-dipole-p3}
\hat{\model}^{\determ,\varep}_{x_\ostar,\{\bar{T}\},\emptyset} \big[ \bar{T} \big](\psi) 
&= \int \dy_{e_-} \, \psi(y_{e_-}) 
e^{-\icomplex \beta \determ(y_{e_-})} H^{\determ}_{\Jc_\varep,\{\bar{T}\},\{\bar{T}\}}[1](y_{e_-}).
\end{align}
Using the definition from \cite[(3.6)]{CHS2018}, we have that $H^{\determ}_{\Jc_\varep,\{\bar{T}\},\{\bar{T}\}}= H^{\determ}_{\Jc_\varep,\{\bar{T}\},\bar{T}}$. Together with $H^\determ_{\Jc_\varep,\mc{F},\emptyset}=\operatorname{Id}$, 
\eqref{eq:H-determ}, and Figure \ref{figure:dipole-tree}, we then obtain
\begin{align}
 &\, H^{\determ}_{\Jc_\varep,\{\bar{T}\},\{\bar{T}\}}[1](y_{e_-}) \notag \\
 =&\, \int \dy_{e_+}\, e^{\icomplex \beta \determ(y_{e_+})} 
\Jc_\varep^{\{ (e_-,e_+) \}}(y_{e_-},y_{e_+}) 
\mrm{Ker}^{\{e \}}(y_{e_-},y_{e_+}) H^{\determ}_{\Jc_\varep,\{\bar{T}\},\emptyset}\big[ -\mathscr{Y}_{\bar{T}} 1\big](y_{e_-},y_{e_+}) \notag \\
=&\,  -  \int \dy_{e_+}\, e^{\icomplex \beta \determ(y_{e_+})} \Jc_\varep^{-}(y_{e_-}-y_{e_+}) K(y_{e_-}-y_{e_+}).
\label{appendix:eq-calculation-dipole-p4}
\end{align}
After inserting this back into \eqref{appendix:eq-calculation-dipole-p3}, we obtain that
\begin{equation*}
\hat{\model}^{\determ,\varep}_{x_\ostar,\{\bar{T}\},\emptyset} \big[ \bar{T} \big](\psi) 
= -  \int \dy_{e_-} \dy_{e_+}\, \psi(y_{e_-}) e^{\icomplex \beta (\determ(y_{e_+})-\determ(y_{e_-}))} \Jc_\varep^{-}(y_{e_-}-y_{e_+}) K(y_{e_-}-y_{e_+}),
\end{equation*}
which corresponds to the third summand in \eqref{appendix:eq-calculation-model-dipole}. \\ 

\emph{Case 2: The tripole from \eqref{appendix:eq-calculation-model-vtripole}.} In the following, we label the edges and nodes of the tripole as in 
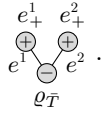
\begin{equation*}
\vtripolempplabeled.
\end{equation*}
We also write $S^{1}$ and $S^{2}$ for the sub-trees with the sole edges $e^1$ and $e^2$, respectively. To simplify the notation, we now write $\bar{T}$ for the tripole in \eqref{appendix:eq-calculation-model-vtripole}. Since $4\pi \leq \beta^2 <6\pi$, it follows from \cite[Definition~2.17]{CHS2018} that $\gamma(e^1)=\gamma(e^2)=1$ and $\mathfrak{C}=\{ e^1,e^2 \}$. As a result, the sum in \eqref{appendix:eq-model-RHS} ranges over 
\begin{equation}\label{appendix:eq-calculation-vtripole-options}
\begin{aligned}
\big( \mc{G}, \mathscr{D} \big) \in 
\Big \{ &\big( \emptyset, \emptyset\big), 
\big( \emptyset, \big\{ e^1 \big\} \big),
\big( \emptyset, \big\{ e^2 \big\} \big),
\big( \emptyset, \big\{ e^1,e^2 \big\} \big), \\
&\big( \big\{ S^1 \big\}, \emptyset \big),
\big( \big\{ S^1 \big\}, \big\{ e^2 \big\} \big),
\big( \big\{ S^2 \big\}, \emptyset \big),
\big( \big\{ S^2 \big\}, \big\{ e^1 \big\} \big) \Big\}
\end{aligned}
\end{equation}
The different sets of edges, nodes, and noise nodes for the forests in \eqref{appendix:eq-calculation-vtripole-options} have been listed in Figure~\ref{figure:dipole-forests}. As before, we
 write $\hat{\model}^{\determ, \varep}_{x_\ostar,\mc{G},\mathscr{D}}[\bar{T}](\psi)$ 
for the summands on the right-hand side of~\eqref{appendix:eq-model-RHS} and compute them using the information from Figure~\ref{figure:dipole-forests} and Figure~\ref{figure:dipole-tree}. \\

\emph{Case 2.a: $\mc{G}=\emptyset$.} Using the third row of Figure \ref{figure:dipole-forests} and $H^{\determ}_{\Jc_\varep,\mc{F},\emptyset}=\operatorname{Id}$ for all forests $\mc{F}$, we obtain that 
\begin{align}
&\hspace{3ex} \hat{\model}^{\determ, \varep}_{x_\ostar,\emptyset,\mathscr{D}}[\bar{T}](\psi) \notag \\
&= \int \dyroot \dy_{e^1_+} \dy_{e^2_+} \, \psi(\yroot) \xi^{\varep}_{-}(\yroot) \xi^{\varep}_{+}(y_{e^1_+})
\xi^{\varep}_{+}(y_{e^2_+}) 
e^{\icomplex \beta ( - \determ(\yroot) + \determ(y_{e^1_+})+\determ(y_{e^2_+}))} \notag \\
&\hspace{3ex}  \times \mrm{Ker}^{\{e^1,e^2\}\backslash \mathscr{D}}(\yroot,y_{e^1_+},y_{e^2_+}) \,
\mrm{RKer}^{\mathscr{D}}(x_{\ostar},\yroot,y_{e^1_+},y_{e^2_+}) \, 
H^{\determ}_{\Jc_\varep,\emptyset,\emptyset}[1](x_{\ostar},\yroot,y_{e^1_+},y_{e^2_+}) \notag \\
&= \int \dyroot \dy_{e^1_+} \dy_{e^2_+} \, \psi(\yroot)
\xi^{\determ,\varep}_{-}(\yroot) \xi^{\determ,\varep}_{+}(y_{e^1_+})
\xi^{\determ,\varep}_{+}(y_{e^2_+}) \label{appendix:eq-calculation-vtripole-p1} \\
&\hspace{6ex} \times \mrm{Ker}^{\{e^1,e^2\}\backslash \mathscr{D}}(\yroot,y_{e^1_+},y_{e^2_+}) \,
\mrm{RKer}^{\mathscr{D}}(x_{\ostar},\yroot,y_{e^1_+},y_{e^2_+}). \notag 
\end{align}
From \cite[(3.3) and (3.4)]{CHS2018}, it directly follows that 
\begin{align*}
&\, \sum_{\mathscr{D}\subseteq \{ e^1,e^2\}} 
\mrm{Ker}^{\{e^1,e^2\}\backslash \mathscr{D}}(\yroot,y_{e^1_+},y_{e^2_+}) \,
\mrm{RKer}^{\mathscr{D}}(x_{\ostar},\yroot,y_{e^1_+},y_{e^2_+}) \\
=& \, K(\yroot-y_{e^1_+}) K(\yroot-y_{e^2_+})
- K(x_{\ostar}-y_{e^1_+}) K(\yroot-y_{e^2_+}) 
 \\
-&\,  K(\yroot-y_{e^1_+}) K(x_{\ostar}-y_{e^2_+}) + 
K(x_{\ostar}-y_{e^1_+}) K(x_{\ostar}-y_{e^2_+}) 
\end{align*}
By inserting this into \eqref{appendix:eq-calculation-vtripole-p1} and computing the $y_{e^1_+}$ and $y_{e^2_+}$-integrals, it follows that 
\begin{align*}
 &\sum_{\mathscr{D}\subseteq \{ e^1,e^2\}} \hat{\model}^{\determ, \varep}_{x_\ostar,\emptyset,\mathscr{D}}[\bar{T}](\psi) \\
 =&\, \int \dyroot \psi(\yroot)
 \xi^{\determ,\varep}_-(\yroot) 
 \Big( \big( K \ast \xi^{\determ,\varep}_+ \big)^2(\yroot) 
 - 2 \big( K \ast \xi^{\determ,\varep}_+ \big)(x_\ostar)
 \big( K \ast \xi^{\determ,\varep}_+ \big)(\yroot) \\
 &\hspace{3ex}
 + \big( K \ast \xi^{\determ,\varep}_+ \big)^2(x_\ostar) \Big) \\
 =&\,  \int \dyroot \psi(\yroot)
 \xi^{\determ,\varep}_-(\yroot) 
 \Big( \big( K \ast \xi^{\determ,\varep}_+ \big)(\yroot) 
 - \big( K \ast \xi^{\determ,\varep}_+ \big)(x_\ostar) \Big)^2, 
\end{align*}
which corresponds to the first term in \eqref{appendix:eq-calculation-model-vtripole}. \\

\emph{Case 2.b: $\mc{G}=\{ S^1\}$.} \\
Using the fourth row of Figure \ref{figure:dipole-forests}, we obtain that 
\begin{equation}\label{appendix:eq-calculation-vtripole-p2}
\begin{aligned}
 \hat{\model}^{\determ, \varep}_{x_\ostar,\{ S^1 \},\mathscr{D}}[\bar{T}](\psi)  
 =&\, \int \dyroot \dy_{e^2_+} \, \psi(\yroot) \xi^{\varep}_{+}(y_{e^2_+}) e^{\icomplex \beta (-\determ(\yroot)+\determ(y_{e^2_+}))} \\
& \hspace{4ex}\times \, 
H^{\determ}_{\Jc_\varep,\{S^1\},\{S^1\}}\Big[ \mrm{RKer}^{\mathscr{D}} \cdot  \mrm{Ker}^{\{e^2\}\backslash \mathscr{D}}  \Big]
(x_{\ostar},\yroot,y_{e^2_+}). 
\end{aligned}
\end{equation}
Using a similar argument as in the derivation of \eqref{appendix:eq-calculation-dipole-p4}, we also have that 
\begin{equation}\label{appendix:eq-calculation-vtripole-p3}
\begin{aligned}
&\,H^{\determ}_{\Jc_\varep,\{S^1\},\{S^1\}}\Big[ \mrm{RKer}^{\mathscr{D}} \cdot  \mrm{Ker}^{\{e^2\}\backslash \mathscr{D}}  \Big](x_{\ostar},\yroot,y_{e^2_+}) \\
=&\,  -  \bigg( \int \dy_{e^1_+}\, e^{\icomplex \beta \determ(y_{e^1_+})} \Jc^{-}_\varep(\yroot-y_{e^1_+}) K(\yroot-y_{e^1_+}) \bigg)
 \mrm{RKer}^{\mathscr{D}}(x_\ostar,\yroot,y_{e^2_+})  \\
 &\hspace{2ex} \times \mrm{Ker}^{\{e^2\}\backslash \mathscr{D}}(\yroot,y_{e^2_+}). 
\end{aligned}
\end{equation}
We emphasize that the reason why the $ \mrm{RKer}^{\mathscr{D}}\cdot\mrm{Ker}^{\{e^2\}\backslash \mathscr{D}}$-term in \eqref{appendix:eq-calculation-vtripole-p3} can be pulled out of the $y_{e^1_+}$-integral is that the sole edge $e^2$ in $K^{\downarrow}(\bar{\mc{G}})$ does not contain the node $e^1_+$.  From \cite[(3.3) and (3.4)]{CHS2018}, it directly follows that 
\begin{equation}\label{appendix:eq-calculation-vtripole-p4}
 \sum_{\mathscr{D}\subseteq \{e^2\}} 
\mrm{RKer}^{\mathscr{D}}(x_\ostar,\yroot,y_{e^2_+}) \mrm{Ker}^{\{e^2\}\backslash \mathscr{D}}(\yroot,y_{e^2_+}) 
= K(\yroot- y_{e^2_+})
-  K(x_\ostar- y_{e^2_+}).
\end{equation}

By inserting \eqref{appendix:eq-calculation-vtripole-p3} and \eqref{appendix:eq-calculation-vtripole-p4} into \eqref{appendix:eq-calculation-vtripole-p2}, we obtain that 
\begin{align*}
&\, \sum_{\substack{\mathscr{D}\subseteq \{e^2 \}}}  \hat{\model}^{\determ, \varep}_{x_\ostar,\{ S^1 \},\mathscr{D}}[\bar{T}](\psi) \\
=&\, - \int \dyroot \Bigg( \psi(\yroot) \bigg( \int  \dy_{e^1_+}  e^{\icomplex \beta (- \determ(\yroot)+\determ(y_{e^1_+}))}
 \Jc^{-}_\varep(\yroot-y_{e^1_+}) K(\yroot-y_{e^1_+}) \bigg) \\
 &\hspace{4ex} \times 
 \bigg( \int \dy_{e^2_+} 
  \big( K(\yroot- y_{e^2_+})
-  K(x_\ostar- y_{e^2_+}) \big)
\xi^{\determ,\varep}_+(y_{e^2_+}) \bigg) \Bigg) \\
=&\, - \int \dyroot  \psi(\yroot) \, 
\E \Big[ \xi^{\determ,\varep}_- \big( K \ast \xi^{\determ,\varep}_+ \big) \Big](\yroot) \, \Big( \big( K \ast \xi^{\determ,\varep}_+ \big)(\yroot) - \big( K \ast \xi^{\determ,\varep}_+ \big)(x_\ostar) \Big), 
\end{align*}
which corresponds to half of the second summand in \eqref{appendix:eq-calculation-model-vtripole}.\\

\emph{Case 2.c: $\mc{G}=\{S^2\}$.} By symmetry, the case $\mc{G}=S^2$ yields the same contribution as the case $\mc{G}=S^1$, which we treated above.  \\

\emph{Case 3: The tripole from \eqref{appendix:eq-calculation-model-ltripole}.} In the following, we label the edges and nodes of the tripole as in 
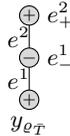
\begin{equation*}
\ltripolepmplabeled
\end{equation*}
We also write $S^{1}$ and $S^{2}$ for the sub-trees with the sole edges $e^1$ and $e^2$, respectively. To simplify the notation, we now write $\bar{T}$ for the tripole in \eqref{appendix:eq-calculation-model-ltripole}. Since $4\pi \leq \beta^2 <6\pi$, it follows from \cite[Definition~2.17]{CHS2018} that $\gamma(e^1) = 2$, $\gamma(e^2) = 1$, and $\mathfrak{C}=\{ e^1,e^2 \}$. As a result, the sum in \eqref{appendix:eq-model-RHS} ranges over 
\begin{equation}\label{appendix:eq-calculation-ltripole-options}
\begin{aligned}
\big( \mc{G}, \mathscr{D} \big) \in 
\Big \{ &\big( \emptyset, \emptyset\big), 
\big( \emptyset, \big\{ e^1 \big\} \big),
\big( \emptyset, \big\{ e^2 \big\} \big),
\big( \emptyset, \big\{ e^1,e^2 \big\} \big), \\
&\big( \big\{ S^1 \big\}, \emptyset \big),
\big( \big\{ S^1 \big\}, \big\{ e^2 \big\} \big),
\big( \big\{ S^2 \big\}, \emptyset \big),
\big( \big\{ S^2 \big\}, \big\{ e^1 \big\} \big) \Big\}
\end{aligned}
\end{equation}
The different sets of edges, nodes, and noise nodes for the forests in \eqref{appendix:eq-calculation-ltripole-options} have been listed in Figure~\ref{figure:dipole-forests}. As before, we
 write $\hat{\model}^{\determ, \varep}_{x_\ostar,\mc{G},\mathscr{D}}[\bar{T}](\psi)$ 
for the summands on the right-hand side of~\eqref{appendix:eq-model-RHS} and compute them using the information from Figure~\ref{figure:dipole-forests} and Figure~\ref{figure:dipole-tree}. \\

\emph{Case 3.a: $\mc{G}=\emptyset$.} Using the sixth row of Figure \ref{figure:dipole-forests} and $H^{\determ}_{\Jc_\varep,\mc{F},\emptyset}=\operatorname{Id}$ for all forests $\mc{F}$, we obtain that 
\begin{align}
&\hspace{3ex} \hat{\model}^{\determ, \varep}_{x_\ostar,\emptyset,\mathscr{D}}[\bar{T}](\psi) \notag \\
&= \int \dyroot \dy_{e^1_-} \dy_{e^2_+} \, \psi(\yroot) \xi^{\varep}_{+}(\yroot) \xi^{\varep}_{-}(y_{e^1_-})
\xi^{\varep}_{+}(y_{e^2_+}) 
e^{\icomplex \beta ( \determ(\yroot) - \determ(y_{e^1_-})+\determ(y_{e^2_+}))} \notag \\
&\hspace{4ex}  \times \mrm{Ker}^{\{e^1,e^2\}\backslash \mathscr{D}}(\yroot,y_{e^1_-},y_{e^2_+}) \,
\mrm{RKer}^{\mathscr{D}}(x_{\ostar},\yroot,y_{e^1_-},y_{e^2_+}) \, 
H^{\determ}_{\Jc_\varep,\emptyset,\emptyset}[1](\yroot,y_{e^1_-},y_{e^2_+}) \notag \\
&= \int \dyroot \dy_{e^1_-} \dy_{e^2_+} \, \psi(\yroot)
\xi^{\determ,\varep}_{+}(\yroot) \xi^{\determ,\varep}_{-}(y_{e^1_-})
\xi^{\determ,\varep}_{+}(y_{e^2_+}) \label{appendix:eq-calculation-ltripole-p1} \\
&\hspace{4ex} \times \mrm{Ker}^{\{e^1,e^2\}\backslash \mathscr{D}}(\yroot,y_{e^1_-},y_{e^2_+}) \,
\mrm{RKer}^{\mathscr{D}}(x_{\ostar},\yroot,y_{e^1_-},y_{e^2_+}). \notag 
\end{align}
From \cite[(3.3) and (3.4)]{CHS2018}, and recalling that $\gamma(e^1) = 2$, $\gamma(e^2) = 1$, it directly follows that 
\begin{align*}
&\hspace{3mm} \sum_{\mathscr{D}\subseteq \{ e^1,e^2\}} 
\mrm{Ker}^{\{e^1,e^2\}\backslash \mathscr{D}}(\yroot,y_{e^1_-},y_{e^2_+}) \,
\mrm{RKer}^{\mathscr{D}}(x_{\ostar},\yroot,y_{e^1_-},y_{e^2_+}) \\
&= K(\yroot-y_{e^1_-}) K(y_{e^1_-}-y_{e^2_+})
- K(\yroot - y_{e^1_-}) K(x_{\ostar}-y_{e^2_+}) \\
&\quad- \bigg(K(x_{\ostar} - y_{e^1_-}) + \sum_{|k|_{\mathfrak{s}} = 1} (\yroot - x_{\ostar}) D^k K(x_{\ostar} - y_{e^1_-})\bigg) K(y_{e^1_-} - y_{e^2_+}) \\
&\quad+ \bigg(K(x_{\ostar} - y_{e^1_-}) + \sum_{|k|_{\mathfrak{s}} = 1} (\yroot - x_{\ostar}) D^k K(x_{\ostar} - y_{e^1_-})\bigg) K(x_{\ostar} - y_{e^2_+}) \\
&= K_1 - K_2 - K_3 + K_4.
\end{align*}
By inserting this into \eqref{appendix:eq-calculation-ltripole-p1} and computing the $y_{e^1_-}$ and $y_{e^2_+}$-integrals, it follows that 
\begin{align*}
 &\hspace{4mm}\sum_{\mathscr{D}\subseteq \{ e^1,e^2\}} \hat{\model}^{\determ, \varep}_{x_\ostar,\emptyset,\mathscr{D}}[\bar{T}](\psi) \\
 &=\, \int \dyroot \dy_{e^1_-} \dy_{e^2_+} \, \psi(\yroot)
\xi^{\determ,\varep}_{+}(\yroot) \xi^{\determ,\varep}_{-}(y_{e^1_-})
\xi^{\determ,\varep}_{+}(y_{e^2_+}) (K_1 - K_2 - K_3 + K_4) \\
&=:\, I_1 - I_2 - I_3 + I_4.
\end{align*}
In order to obtain the formula in \eqref{appendix:eq-calculation-model-ltripole}--\eqref{appendix:eq-calculation-model-ltripole-end}, we will need to combine these terms with ones obtained later. This will be done at the end. \\

\emph{Case 3.b: $\mc{G}=\{ S^1\}$.} Using the seventh row of Figure \ref{figure:dipole-forests}, and noting that $K(\mc{G}, \bar{T}) = K(\bar{T}) \backslash \bar{K}^{\downarrow}(\mc{G}) = \varnothing$, we obtain that
\begin{equation}\label{appendix:eq-calculation-ltripole-p2}
\begin{aligned}
 \hat{\model}^{\determ, \varep}_{x_\ostar,\{ S^1 \},\mathscr{D}}[\bar{T}](\psi)  
 =&\, \int \dyroot \dy_{e^2_+} \, \psi(\yroot) \xi^{\varep}_{+}(y_{e^2_+}) e^{\icomplex \beta (\determ(\yroot)+\determ(y_{e^2_+}))} \\
 &~\times H^{\determ}_{\Jc_\varep,\{S^1\},\{S^1\}}\Big[\mrm{RKer}^{\{e_2\} \cap \mathscr{D}} \mrm{Ker}^{\{e_2\} \backslash \mathscr{D}}\Big](x_{\ostar},\yroot,y_{e^2_+}). 
\end{aligned}
\end{equation}
Using the recursive definition \eqref{eq:H-determ} of $H^{\determ}_{\Jc_\varep, \{S^1\}, \{S^1\}}$, we have that
\begin{equation}\label{appendix:eq-calculation-ltripole-p3}
\begin{aligned}
&\hspace{4mm} H^{\determ}_{\Jc_\varep,\{S^1\},\{S^1\}}\big[ \mrm{Ker}^{\{e_2\}}\big](x_{\ostar},\yroot, y_{e^2_+})
\\
&= -  \int \dy_{e^1_-}\, e^{-\icomplex \beta \determ(y_{e^1_-})} \Jc_\varep^{-}(\yroot-y_{e^1_-}) K(\yroot-y_{e^1_-}) (\mathscr{Y}_{S^1} \mrm{Ker}^{\{e_2\}})(\yroot, y_{e^1_-}, y_{e^2_+}) \\
&= - K(\yroot - y_{e^2_+})\int \dy_{e^1_-}\, e^{-\icomplex \beta \determ(y_{e^1_-})} \Jc_\varep^{-}(\yroot-y_{e^1_-}) K(\yroot-y_{e^1_-})  .
\end{aligned}
\end{equation}
Similarly, we have that
\begin{equation}\label{appendix:eq-calculation-ltripole-p4}
\begin{aligned}
&\, H^{\determ}_{\Jc_\varep,\{S^1\},\{S^1\}}\big[ \mrm{RKer}^{\{e_2\}}\big](x_{\ostar},\yroot, y_{e^2_+})\\
=&\, K(x_{\ostar} - y_{e^2_+}) \int \dy_{e^1_-} \, e^{-\icomplex \beta \determ(y_{e^1_-})} \Jc_\varep^{-}(\yroot-y_{e^1_-}) K(\yroot-y_{e^1_-}).
\end{aligned}
\end{equation}
By inserting \eqref{appendix:eq-calculation-ltripole-p3} and \eqref{appendix:eq-calculation-ltripole-p4} into \eqref{appendix:eq-calculation-ltripole-p2}, we obtain that 
\begin{align*}
&\hspace{4mm} \sum_{\substack{\mathscr{D}\subseteq \{e^2 \}}}  \hat{\model}^{\determ, \varep}_{x_\ostar,\{ S^1 \},\mathscr{D}}[\bar{T}](\psi) \\
&=\, -\int \dyroot \dy_{e_2^+} \psi(\yroot) \xi^{\determ,\varep}_+(y_{e^2_+})     
  \big( K(\yroot- y_{e^2_+})
-  K(x_\ostar- y_{e^2_+}) \big)
\\
 &\hspace{4ex} \times \bigg( \int  \dy_{e^1_-}  e^{\icomplex \beta (\determ(\yroot)-\determ(y_{e^1_-}))}
 \Jc_\varep^{-}(\yroot-y_{e^1_-}) K(\yroot-y_{e^1_-}) \bigg)\\
&= -\int \dyroot  \psi(\yroot) \, 
\E \Big[ \xi^{\determ,\varep}_+ \big( K \ast \xi^{\determ,\varep}_- \big) \Big](\yroot) \, \Big( \big( K \ast \xi^{\determ,\varep}_+ \big)(\yroot) - \big( K \ast \xi^{\determ,\varep}_+ \big)(x_\ostar) \Big). \\
&=: -(I_5 - I_6),
\end{align*} 
where $I_5$ is the integral containing the $K \ast \xi^{\determ, \varep}_+(\yroot)$ term, and $I_6$ is the integral containing the $K \ast \xi^{\determ, \varep}_+(x_{\ostar})$ term.
\\

\emph{Case 3.c: $\mc{G}=\{S^2\}$.} 
Using the eighth row of Figure \ref{figure:dipole-forests}, we obtain that
\begin{equation}\label{appendix:eq-calculation-ltripole-p10}
\begin{aligned}
 &\, \hat{\model}^{\determ, \varep}_{x_\ostar,\{ S^2 \},\mathscr{D}}[\bar{T}](\psi)  \\ 
 =&\, \int \dyroot \dy_{e^1_-} \, \psi(\yroot) \xi^{\varep}_{+}(\yroot) e^{\icomplex \beta (\determ(\yroot)-\determ(y_{e^1_-}))}
 \mrm{Ker}^{\{e^1\}\backslash \mathscr{D}}(\yroot,y_{e^1_-}) \\
& \hspace{4ex}\times 
\mrm{RKer}^{\mathscr{D}}(x_\ostar,\yroot,y_{e^1_-})\,
H^{\determ}_{\Jc_\varep,\{S^2\},\{S^2\}}[1](x_{\ostar},\yroot,y_{e^1_-}, y_{e^2_+}). 
\end{aligned}
\end{equation}
From \cite[(3.3) and (3.4)]{CHS2018} (and recalling that $\gamma(e^1) = 2$), it directly follows that
\begin{equation}\label{appendix:eq-calculation-ltripole-p11}
\begin{aligned}
&\sum_{\mc{D}\sse \{e^1\}} \mrm{Ker}^{\{e^1\}\backslash \mathscr{D}}(\yroot,y_{e^1_-}) 
\mrm{RKer}^{\mathscr{D}}(x_\ostar,\yroot,y_{e^1_-}) \\
&= K(\yroot - y_{e^1_-}) - K(x_{\ostar} - y_{e^1_-}) - \sum_{|k|_{\mathfrak{s}} = 1} (\yroot - x_{\ostar})^k D^k K(x_{\ostar} - y_{e^1_-}).
\end{aligned}
\end{equation}
Similar to \eqref{appendix:eq-calculation-ltripole-p4}, we have that
\begin{equation}\label{appendix:eq-calculation-ltripole-p12}
\begin{aligned}
&\,H^{\determ}_{\Jc_\varep,\{S^2\},\{S^2\}}[1](x_{\ostar},\yroot,y_{e^1_-},y_{e^2_+}) \\
=&\, -  \int \dy_{e^2_+}\, e^{\icomplex \beta \determ(y_{e^2_+})} \Jc_\varep^-(y_{e^1_-}-y_{e^2_+}) K(y_{e^1_-}-y_{e^2_+}).
\end{aligned}
\end{equation}
Inserting \eqref{appendix:eq-calculation-ltripole-p11} and \eqref{appendix:eq-calculation-ltripole-p12} into \eqref{appendix:eq-calculation-ltripole-p10}, we obtain that 
\begin{align*}
&\hspace{4mm}\sum_{\mathscr{D} \sse \{e^1\}} \hat{\model}^{\determ, \varep}_{x_\ostar,\{ S^2 \},\mathscr{D}}[\bar{T}](\psi) \\
&=\, -\int \dyroot \dy_{e^1_-} \psi(\yroot) \xi^{\determ, \varep}_+(\yroot)  \\
&\quad \times \bigg(K(\yroot - y_{e^1_-}) - K(x_{\ostar} - y_{e^1_-}) -  \sum_{|k|_{\mathfrak{s}} = 1} (\yroot - x_{\ostar})^k D^k K(x_{\ostar} - y_{e^1_-}) \bigg)  \\
&\quad \times \bigg(\int dy_{e^2_+} e^{\icomplex \beta (-\determ(y_{e^1_-}) + \determ(y_{e^2_+}))} \Jc_\varep^-(y_{e^1_-} - y_{e^2_+}) K(y_{e^1_-} - y_{e^2_+}) \bigg)\\
&=\, - \int \dyroot \psi(\yroot) \xi^{\determ, \varep}_+(\yroot) \bigg(K \ast \Big(\E\Big[ \xi^{\determ, \varep}_- (K \ast \xi^{\determ, \varep}_+)\Big]\Big)(\yroot)  \\
&\hspace{3.5ex} - K \ast \Big(\E\Big[ \xi^{\determ, \varep}_- (K \ast \xi^{\determ, \varep}_+)\Big]\Big)(x_{\ostar}) - \sum_{|k|_{\mathfrak{s}}=1} (\yroot - x_{\ostar})^k D^k K \ast \E\Big[ \xi^{\determ, \varep}_- (K \ast \xi^{\determ, \varep}_+)\Big]\Big)(x_{\ostar})\bigg) \\
&=:\, - (I_7 - I_8 - I_9),
\end{align*}
where $I_7$ is the integral containing the $K \ast \Big(\E\Big[ \xi^{\determ, \varep}_- (K \ast \xi^{\determ, \varep}_+)\Big]\Big)(\yroot)$ term, $I_8$ is the integral containing the $K \ast \Big(\E\Big[ \xi^{\determ, \varep}_- (K \ast \xi^{\determ, \varep}_+)\Big]\Big)(x_{\ostar})$ term, and $I_9$ is the integral containing the $\sum_{|k|_{\mathfrak{s}}=1}\cdots$ term.

To finish, we indicate how to combine $I_1, \ldots, I_9$ to obtain the terms in \eqref{appendix:eq-calculation-model-ltripole}-\eqref{appendix:eq-calculation-model-ltripole-end}. We have that 
\begin{equs}
I_1 - I_2 - I_7 &= \int \dyroot \psi(\yroot) \xi^{\determ, \varep}_+(\yroot) K \ast \bigg( \model^{\determ, \varep}_{x_\ostar}\bigg[\,  \dipmp\, \bigg]\bigg)(\yroot), \\
-I_3 + I_4 + I_8 + I_9 &= - \int \dyroot \psi(\yroot) \xi^{\determ, \varep}_+(\yroot) \bigg(K \ast \bigg( \model^{\determ, \varep}_{x_\ostar}\bigg[\,  \dipmp\, \bigg]\bigg)(x_{\ostar})  \\
&\quad \quad + \sum_{|k|_{\mathfrak{s} = 1}} (\yroot - x_{\ostar})^k D^k K \ast \bigg( \model^{\determ, \varep}_{x_\ostar}\bigg[\,  \dipmp\, \bigg]\bigg)(x_{\ostar})\bigg), 
\end{equs}
and 
\begin{equs}
-(I_5 - I_6) &= -\int \dyroot \psi(\yroot) \E\Big[\xi^{\determ, \varep}_+ K \ast \xi^{\determ, \varep}_-\Big](\yroot) \Big((K \ast \xi^{\determ, \varep}_+)(\yroot) - (K \ast \xi^{\determ, \varep}_+)(x_{\ostar})\Big).
\end{equs}
The desired result now follows.
\end{proof} 

Using \eqref{appendix:eq-model}, we can show the following result about our modified model, which was needed in Section \ref{section:remainder-bound} when computing the reconstruction of the lifted equation in the proof Proposition~\ref{prop:reconstruction}.

\begin{corollary}\label{cor:reconstruction-calculation}
For any $x_{\ostar} \in (\R \times \T^2)^{\ostar}$, we have that 
\begin{equs}
\model^{\determ, \varep}_{x_\ostar}\bigg[\,  \dipmp\, \bigg](x_{\ostar}) &= \ovl{\model^{\determ, \varep}_{x_\ostar}\bigg[\,  \dippm\, \bigg](x_{\ostar})} = -\int \dzprime K(x_{\ostar} - \zprime) \Jc_\varep^-(x_{\ostar} - z') e^{-\icomplex \beta (\determ(x_{\ostar}) - \determ(z'))}, \\
\model^{\determ, \varep}_{x_\ostar}\bigg[\,  \dippp\, \bigg](x_{\ostar}) &= \model^{\determ, \varep}_{x_\ostar}\bigg[\,  \dipmm\, \bigg](x_{\ostar})  = 0, \\
\big(\model^{\determ, \varep}_{x_\ostar} \tau \big)(x_{\ostar}) &= 0 \text{ for any tripole $\tau$.}
\end{equs}
\end{corollary}
\begin{proof}
The first identity follows directly from \eqref{appendix:eq-premodel-to-model-dipole} and \eqref{eq:Jc-covariance}. For the second two identities, we first observe that for $\tau$ which is a dipole or tripole which moreover is {\it not} one of the following trees:
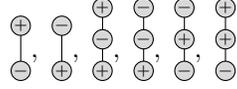
\begin{equs}
\dipmp, ~~ \dippm, ~~ \ltripolepmp, ~~ \ltripolepmm, ~~ \ltripolempm, ~~ \ltripolempp,
\end{equs}
one can see that $\model^{\determ, \varep}_{x_\ostar}[\tau](x_{\ostar}) = 0$ quite directly from the general formula \eqref{appendix:eq-model}. Indeed, for such a $\tau$, one has that $K^{\downarrow}(\bar{\mc{G}}) = \varnothing$ for any $\mc{G} \in \mbb{F}_j$, and thus the function inside the $H^{\determ}_{\mc{J}, \mc{G}, \bar{\mc{G}}}$ operator is always just 1 (note that the tree $\tau$ is assumed to have zero polynomial decoration). Thus for any fixed $\mc{G} \in \mbb{F}_j$, the summation over $\mathscr{D} \sse \mathfrak{C} \backslash K(\mc{G})$ in \eqref{appendix:eq-model} can be written 
\begin{equs}
&\sum_{\mathscr{D} \sse \mathfrak{C} \backslash K(\mc{G})} \int_{N(\mc{G}, \bar{T})} \dy  \xi^{L(\mc{G}, \bar{T}), \varep}(y) \prod_{u \in N(\mc{G}, \bar{T})} e^{\icomplex \beta \nodelabel(u) \determ(y_u)} \cdot \mrm{Ker}^{K(\mc{G}, \bar{T}) \backslash \mathscr{D}}(y) \cdot \psi(z_{\varrho_{\bar{T}}})  \\
&\quad \quad \cdot   \mrm{RKer}^{K(\mc{G}, \bar{T}) \cap \mathscr{D}}(z) \cdot X_{\bar{\mathfrak{n}}, \ostar}^{N(\mc{G}, \bar{T})}(z) \cdot H_{\mc{J}, \mc{G}, \bar{\mc{G}}}^\determ \Big[ \mrm{RKer}^{K^{\downarrow}(\bar{\mc{G}}) \cap \mathscr{D}} \cdot \mrm{Ker}^{K^{\downarrow}(\bar{\mc{G}}) \backslash \mathscr{D}} X_{\bar{\mathfrak{n}}, \ostar}^{\tilde{N}(\bar{\mc{G}})}\Big](z) \\
&= \int_{N(\mc{G}, \bar{T})} \dy  \xi^{L(\mc{G}, \bar{T}), \varep}(y) \prod_{u \in N(\mc{G}, \bar{T})} e^{\icomplex \beta \nodelabel(u) \determ (y_u)} \cdot \psi(z_{\varrho_{\bar{T}}}) H_{\mc{J}, \mc{G}, \bar{\mc{G}}}^\determ [1](z) \\
&\quad \quad \quad \quad \quad \quad \quad \quad \times\sum_{\mathscr{D} \sse \mathfrak{C} \backslash K(\mc{G})} \mrm{Ker}^{K(\mc{G}, \bar{T}) \backslash \mathscr{D}}(y) \cdot  \mrm{RKer}^{K(\mc{G}, \bar{T}) \cap \mathscr{D}}(z).
\end{equs}
The claim that $\model^{\determ, \varep}_{x_\ostar}[\tau](x_{\ostar}) = 0$ then follows because
\begin{equs}
\sum_{\mathscr{D} \sse \mathfrak{C} \backslash K(\mc{G})} \mrm{Ker}^{K(\mc{G}, \bar{T}) \backslash \mathscr{D}}(y) \cdot  \mrm{RKer}^{K(\mc{G}, \bar{T}) \cap \mathscr{D}}(z) = 0 \quad \text{ if $\yroot = x_{\ostar}$.}
\end{equs}
To finish the proof, it remains to address the cases where $\tau$ is a tripole of the form
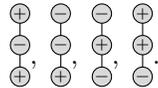
\begin{equs}
\ltripolepmp, ~~ \ltripolepmm, ~~ \ltripolempm, ~~ \ltripolempp.
\end{equs}
The cases where $\tau$ is the first or third tripole above follow by the explicit calculation \eqref{appendix:eq-premodel-to-model-ltripole-begin}-\eqref{appendix:eq-premodel-to-model-ltripole-end}. The other cases follow by a similar (and indeed, simpler) calculation, which we omit.
\end{proof}

\end{appendix}

\begin{acks}[Acknowledgments]
The authors thank Martin Hairer, Sarah-Jean Meyer, Hao Shen, and Younes Zine for helpful discussions. \revision{They also thank the anonymous referee whose comments improved the manuscript.   }
\end{acks}

\begin{funding}
 S.C. was partially supported by the NSF under Grant No. DMS-2303165. S.C. would also like to thank the IAS for its hospitality during a visit where some of the work in this paper was carried out.
\end{funding}

\bibliographystyle{imsart-nameyear} 
\bibliography{Sine-Gordon-Library-AOP}

\end{document}